%% file: main.tex
\documentclass{article}

\input{preamble.tex}

\title{Correlation decay in area-tilted line ensembles}
\author{Shirshendu Ganguly and Vilas Winstein}
\date{}

\begin{document}

\maketitle

\begin{abstract}
Random surfaces on a hard substrate often exhibit entropic repulsion wherein the
surface is propelled upwards to allow entropically preferable downward
fluctuations. A classical, albeit one dimensional, example is a Brownian
excursion. A much richer two dimensional example arises from a low-temperature
3D Ising model and the related solid-on-solid model in the presence of an
appropriate boundary. A powerful approach to studying such surfaces is through
their level curves, which form a family of non-intersecting random curves. In
\cite{CIW18Tightness,CIW19Confinement}, an ensemble of Brownian lines with geometrically increasing area tilts,
viewed as a Gibbs measure on an infinite family of random curves, was proposed
as a putative limiting model in this case. While this perspective, in
conjunction with algebraic inputs and connections to explicit SDEs such as Dyson
Brownian motion, has proven to be a major success story in the analysis of the
Airy line ensemble \cite{CH14Brownian,AH26Strong}, the area-tilted model lies outside the scope of such
techniques. Nonetheless, there have been important recent developments \cite{CG25Uniqueness,BCG25Characterizing}.

A particularly intriguing question about such line ensembles concerns their
relaxation and mixing properties when viewed as a Markov process, and in
particular the rate of decay of correlations in time. For the Airy line
ensemble, this decay is known to be inverse quadratic, for instance via
determinantal techniques \cite{PS02Scale}. The first quantitative bound on the decay of correlations in the area-tilted model,
established in \cite{CG25Uniqueness}, was slower than polynomial in time (the exponential of a power less
than one of the logarithm). An earlier result \cite{DLZ24Limiting} had established positivity
of the spectral gap for the finite-line version of the ensemble, without quantitative bounds,
using an abstract functional-analytic approach. This left open the important
question of the true decay rate of correlations for the infinite
ensemble, with no consensus prediction in the literature. Settling this question
for sufficiently large area-tilt strength, corresponding to sufficiently low
temperature for the 3D Ising model, we prove exponential decay of correlations
for the infinite ensemble and a uniform (in the number of lines) positive
spectral gap for the finite ensemble. Our proof is based on establishing a precise form of separation of
scales between curves of different indices, using a novel probabilistic
approach involving embedding supercritical branching processes in the line ensemble.

\end{abstract}%

\thispagestyle{empty}%
\vfill%
\noindent%
\makebox[\textwidth][c]{%
\includegraphics[width=\paperwidth]{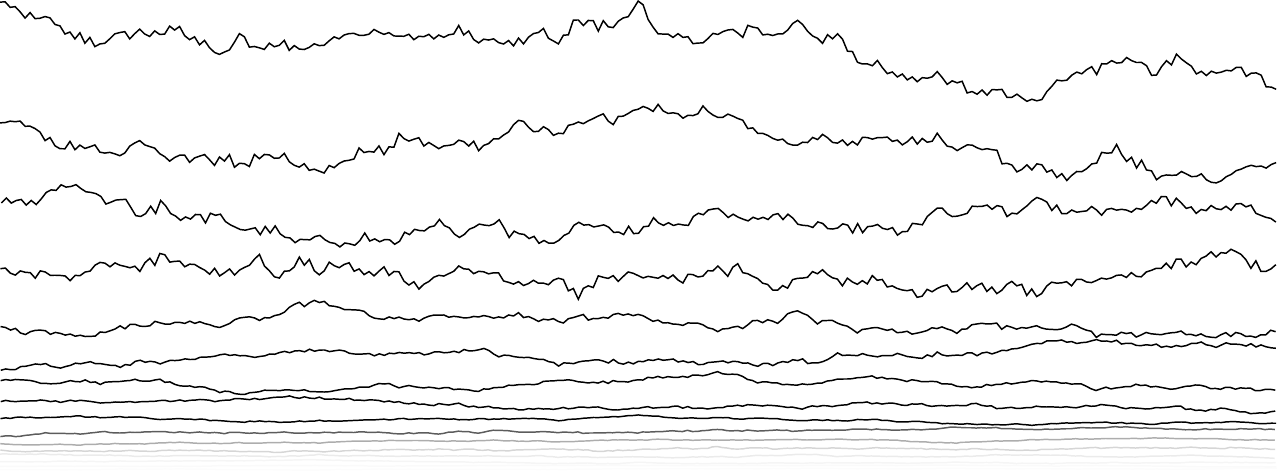}%
}
\vspace{-3cm}%

\clearpage

\setcounter{tocdepth}{2}
\tableofcontents

\vfill

\input{sections/1_intro_0}
\input{sections/2_inputs_0}
\input{sections/3_topline_0}
\input{sections/4_correlation_0}
\input{sections/5_spectralgap_0}

\bibliographystyle{plain}
\bibliography{references}

\end{document}

%% file: preamble.tex
\usepackage[margin=1in]{geometry}
\usepackage{amsmath, mathtools, amsfonts, mathrsfs, amsthm, amssymb, cite, suffix, enumitem, bm, graphbox, graphicx}
\mathtoolsset{showonlyrefs}
\usepackage[colorlinks]{hyperref}
\hypersetup{linkcolor=blue}
\usepackage{xcolor}

\usepackage{multirow}

\newtheorem{theorem}{Theorem}[section]
\newtheorem{lemma}[theorem]{Lemma}
\newtheorem{corollary}[theorem]{Corollary}

\newtheorem{proposition}[theorem]{Proposition}

\theoremstyle{definition}
\newtheorem{definition}[theorem]{Definition}
\newtheorem{remark}[theorem]{Remark}

\renewcommand{\geq}{\geqslant}
\renewcommand{\leq}{\leqslant}

\usepackage{mleftright}
\renewcommand{\left}{\mleft}
\renewcommand{\right}{\mright}

\renewcommand{\emptyset}{\varnothing}
\renewcommand{\P}{\mathbb{P}}
\newcommand{\E}{\mathbb{E}}
\newcommand{\Z}{\mathbb{Z}}
\newcommand{\R}{\mathbb{R}}
\newcommand{\N}{\mathbb{N}}

\newcommand{\cG}{\mathcal{G}}
\newcommand{\cD}{\mathcal{D}}
\newcommand{\cH}{\mathcal{H}}
\newcommand{\cA}{\mathcal{A}}
\newcommand{\cM}{\mathcal{M}}

\newcommand{\cL}{\mathcal{L}}

\newcommand{\cF}{\mathcal{F}}

\newcommand{\cW}{\mathcal{W}}

\newcommand{\cP}{\mathcal{P}}

\newcommand{\Cov}{\operatorname{Cov}}

\newcommand{\ind}[1]{\mathbf{1}_{\{#1\}}}
\newcommand{\eps}{\varepsilon}
\newcommand{\Exp}[1]{\exp\left(#1\right)}
\WithSuffix\newcommand\ind*[1]{\mathbf{1}_{#1}}

\newcommand{\Lip}{\operatorname{Lip}}

\let\temp\phi
\let\phi\varphi
\let\varphi\temp

\newcommand{\X}{\mathbf{X}}
\newcommand{\Y}{\mathbf{Y}}
\newcommand{\W}{\mathbf{W}}

\newcommand{\x}{\mathbf{x}}
\newcommand{\y}{\mathbf{y}}
\newcommand{\f}{\mathbf{f}}
\newcommand{\g}{\mathbf{g}}
\newcommand{\bl}{\bm{\lambda}}

\newcommand{\A}{\mathbb{A}}
\newcommand{\bridgemeasure}{\mathcal{B}}
\newcommand{\atmeasure}{\mathcal{L}}

\newcommand{\generator}{\mathscr{L}}
\newcommand{\gennl}{{\generator_n^\lambda}}
\newcommand{\pinl}{{\pi_n^\lambda}}
\newcommand{\pnl}{\mathscr{P}_{n,\lambda}}
\newcommand{\gapnl}{{\gamma_n^\lambda}}

\newcommand{\extalg}{\mathcal{E}}
\newcommand{\extextalg}{\widetilde{\mathcal{E}}}
\newcommand{\FS}{Y_{\mathrm{FS}}}

\newcommand{\Ai}{\mathrm{Ai}}

\newcommand{\treeb}{\mathcal{T}_b}

\newcommand{\BP}{\cP}

\newcommand{\sX}{\widetilde{X}}
\newcommand{\sY}{\widetilde{Y}}
\newcommand{\sW}{\widetilde{W}}

\newcommand{\ssX}{\widetilde{\X}}
\newcommand{\ssY}{\widetilde{\Y}}
\newcommand{\ssW}{\widetilde{\W}}

\newcommand{\Low}{\mathsf{Low}}
\newcommand{\High}{\mathsf{High}}
\newcommand{\Reverse}{\mathsf{Reverse}}
\newcommand{\AllReversed}{\mathsf{AllReversed}}

\newcommand{\Endpoints}{\mathsf{LowEnds}}

%% file: sections/1_intro_0.tex
\section{Introduction}
Line ensembles are collections of naturally occurring families of interacting
random curves. A classical example is Dyson Brownian motion (DBM), which
describes the motion of the eigenvalues of an $n\times n$ Gaussian unitary
ensemble (GUE) as its complex Gaussian entries perform independent (up to
Hermitian symmetry) Brownian motions. An alternate probabilistic description
\cite{G99Brownian} states that the DBM is the same as a collection of independent Brownian
motions conditioned not to intersect. This immediately implies that they possess
a simple yet powerful Markovian property, namely that the conditional law of a
finite subset of the curves on a compact interval, given everything else, is
that of non-intersecting Brownian bridges respecting the prescribed boundary
data. 

Non-crossing Brownian motions possess a rich algebraic structure via the
Karlin--McGregor formula \cite{KM59Coincidence}, which expresses correlation functions as
determinants.  Indeed, such formulas, along with probabilistic methods, were
employed in \cite{J01Discrete}, where the Airy line ensemble, whose top line is the
parabolic $\mathrm{Airy}_2$ process, was constructed as a scaling limit of the DBM. The
aforementioned resampling invariance property passes to the limit.  This
Markovian property, in this context referred to as the Brownian Gibbs (BG)
property, allows one to think of such line ensembles as special cases of
infinite-volume Gibbs measures. Natural questions about such objects entail a
precise understanding of various observables, including one-point tail
estimates, decay of correlations, ergodic properties, relaxation to
equilibrium, and so on. Probabilistic techniques, in conjunction with crucial
integrable inputs, have greatly facilitated the study of such questions for the
Airy line ensemble; see e.g.\ \cite{CH14Brownian,AH26Strong,GH22Sharp}.

\subsection{Low-temperature level curves and entropic repulsion}

A related class of examples arises from the study of local restrictions of level
curves of discrete random interfaces. Perhaps the most canonical example is
provided by the low-temperature three-dimensional Ising model with a hard floor,
that is, with positive boundary conditions everywhere except on the floor, where
the spins are fixed to be negative. This creates an interface separating the two
phases, which is pushed upward in order to permit entropically favorable
downward fluctuations. This phenomenon is known as \emph{entropic repulsion}.

A rigorous analysis of entropic repulsion in the $(2+1)$-dimensional
Solid-On-Solid (SOS) model---a low-temperature approximation of the
three-dimensional Ising model---was initiated by Bricmont, El Mellouki, and
Fr\"ohlich in 1986 \cite{BEF86Random}. Subsequently, it was shown in \cite{CLMST16Scaling} that the model
exhibits a sequence of nested level lines, each enclosing a macroscopic fraction
of the system (see also \cite{CMT17Entropic} for analogous results for other gradient
interface models). As a consequence of entropic repulsion, the energy associated
with the $i$th level line is proportional to the area enclosed between the $i$th
and $(i+1)$st lines, with a proportionality constant that grows exponentially
with $i$.

To investigate the finer properties of these level lines, the authors of
\cite{CIW18Tightness,CIW19Confinement}, in
a series of two papers, introduced and initiated the study of a conjectural
scaling limit consisting of an infinite collection of non-intersecting Brownian
bridges constrained above a hard wall and subject to geometrically increasing
area tilts. Consequently, the lower curves in the stack (that is, those with
larger indices) experience progressively stronger pressure toward the wall.
Recently, in \cite{HKS25Scaling}, a class of discrete area tilted models of
random walks have been shown to converge to this limiting model,
which we refer to as the \emph{$\lambda$-tilted line ensemble} (LE), denoted
by $\X=(X^1 > X^2 > \cdots)$; a precise definition is given shortly.

Since its introduction, there have been substantial advances in the study of
this model, including results on ergodicity, decay of correlations, tail
behavior, and the characterization of infinite-volume Gibbs measures. Many of
these questions were resolved in the two papers \cite{CG25Uniqueness,BCG25Characterizing}.
However, the
fundamental problem of determining the true decay of correlations of $\X$ remained open.
This question is the primary focus of the present paper, to which we now turn.

\subsubsection*{The question of correlation decay}

For the $\mathrm{Airy}_2$ process (the top line of the Airy line ensemble), it is
known, via determinantal techniques, that
\[
\Cov[\mathrm{Airy}_2(0),\mathrm{Airy}_2(t)] = C t^{-2} + O(t^{-4}),
\]
for some explicit constant $C$ \cite{PS02Scale}.
This slow decay of correlations is a manifestation of the long-range effects
that the bulk curves induce on the top curve.

In the area-tilted model, there are two competing forces at play.  On the one
hand, there are arbitrarily many curves, creating the possibility of long-range
effects as above. On the other hand, the geometrically increasing area tilts
create a confinement of the curves along with a degree of separation of scales
and, consequently, a sense of independence. This leads to the intriguing
question of deciding which effect is dominant.

Another model which shares the feature of being confined despite having
arbitrarily many lines is the \emph{Dyson--Ornstein--Uhlenbeck process}, i.e.\ a
collection of Brownian particles evolving under a confinement potential and a
repulsive interaction. In \cite{BCL23Universal}, relying on an explicit SDE
representation, it was shown that the process has a spectral gap which is
bounded away from zero uniformly in $n$, the number of curves, leading to a
strong form of exponential decay of correlations.

The finite-line version of our area-tilted process also satisfies an SDE, but it
is defined implicitly via the Doob $h$-transform and is not particularly
tractable. Thus, for this model, the two well-known approaches relying on
determinantal or SDE structures are both ruled out.

Nonetheless, there have been important developments. The first work in this
direction \cite{DLZ24Limiting} studied the $n$-line version of $\X$ as a
diffusion process and showed that the generator has a positive spectral gap
\cite[Lemma 2.1]{DLZ24Limiting}, leading to exponential mixing.  However, their
methods, based on abstract functional analysis, were rather non-quantitative,
and no bound on the behavior of the spectral gap as a function of $n$ was given.

The first quantitative statement of correlation decay was given by \cite[Theorem 3.5]{CG25Uniqueness},
where it was shown that
\begin{equation}
\label{eq:37}
    \left| \Cov\left[X^1(0), X^1(t)\right] \right| \leq C \Exp{- c (\log t)^{3/7}}
\end{equation}
for some constants $C, c > 0$, where $X^1$ denotes the top line of $\X$.
One should expect that this covariance is in fact nonnegative via the FKG inequality,
and indeed we will prove this in the present article.
While the ideas leading to the above bound had multiple other strong applications
including an axiomatic characterization of $\X$, the bound itself is 
unsatisfying, as the true nature of correlations in $\X$ is likely to be either
polynomial or exponential.  Of course, \eqref{eq:37} itself does not suggest
which of these should hold, nor has there been even a heuristic presented in
the literature predicting the true correlation decay behavior.

\vspace{3mm}

We now turn to formally defining the area-tilted line ensemble and stating our main results.

\input{sections/1_intro_1_setup}
\input{sections/1_intro_2_results}
\input{sections/1_intro_3_iop}
\input{sections/1_intro_4_acknowledgements}

%% file: sections/1_intro_1_setup.tex
\subsection{Setup}
\label{sec:intro_setup}

In this section we present the definition of the infinite-volume
$\lambda$-tilted line ensemble $\X$ introduced informally above.
We will be relatively brief, only presenting context which is essential to the definition and
to our proofs.
For further background on $\X$ and other line ensembles, we refer the reader to the works
\cite{CG25Uniqueness,BCG25Characterizing,CH14Brownian}.

For any $n \in \N$, let us denote by $\A_+^n \subseteq \R_+^n$ the set
of decreasing positive sequences, i.e. $\x = (x^1, \dotsc, x^n)$ with $x^1 > \dotsb > x^n > 0$,
and similarly let $\A_+^\infty$ denote the set of infinite decreasing positive sequences
with the topology of pointwise convergence, i.e. $\x_k \to \x$ if $x_k^n \to x^n$ for all $n$.
The line ensembles we define will be random continuous functions from $\R$
(or an interval $[\ell,r]$ or $(\ell,r)$) to $\A_+^n$ (or eventually $\A_+^\infty$).
We endow these function spaces with the topology of uniform convergence on compact sets,
and the corresponding measure spaces with the Borel $\sigma$-algebra.

\subsubsection{The finite-line ensemble}
\label{sec:intro_setup_finite}

For any $n \in \N$, $\x,\y \in \A_+^n$, and $\ell < r$,
let us denote by $\bridgemeasure^{\x,\y}_{n,\ell,r}$ the distribution of $n$ independent
Brownian bridges on $[\ell,r]$, with the $i$th bridge starting at $x^i$ and ending at $y^i$.

\begin{definition}[Finite $\lambda$-tilted line ensemble]
\label{def:finiteLE}
For any $\lambda > 0$, let $\atmeasure^{\x,\y,\lambda}_{n,\ell,r}$ denote the probability distribution
on functions $\X_n=(X^1_n,X^2_n,\ldots, X^n_n) : [\ell,r] \to \A_+^n$ defined by the Radon--Nikodym derivative
\begin{equation}
    \frac{d \atmeasure^{\x,\y,\lambda}_{n,\ell,r}}{d \bridgemeasure^{\x,\y}_{n,\ell,r}}(\X_n)
    \propto \Exp{-2 \sum_{i=1}^n \lambda^{i-1} \int_\ell^r X^i_n(t) \,dt}
    \ind{\X_n(t) \in \A_+^n \text{ for all } t \in [\ell,r]}.
\end{equation}
\end{definition}

Note that the implicit normalization constant (termed the partition function) is
finite and nonzero, meaning that
the above definition is indeed valid for any $\x,\y \in \A_+^n$.
We may also extend the definition to all $\x,\y \in \overline{\A_+^n}$, the closure
of $\A_+^n$, in which strict inequality between adjacent elements is replaced by nonstrict inequality.
This is done via a limiting procedure
using a monotonicity property to be discussed in Section \ref{sec:inputs_monotonicity} below,
see for instance \cite{CH14Brownian,CIW18Tightness,CIW19Confinement} for more details.
When ignoring the boundary points $\ell$ and $r$, the
resulting measure is supported on functions $(\ell,r) \to \A_+^n$, i.e.\ the lines will not touch
in the interior of the interval.

Further, in many cases, we will need to impose a nonzero floor.
For any function $f : [\ell,r] \to \R$, we will denote by $\atmeasure^{\x,\y,\lambda}_{n,\ell,r,f}$
the measure as in Definition \ref{def:finiteLE} where the condition $\X_n(t) \in \A_+^n$
is replaced by $X_n^1(t) > \dotsb > X_n^n(t) > f(t)$ for all $t \in [\ell,r]$.
If we have two continuous functions $f, g : [\ell,r] \to \R$
with $f(t) < g(t)$ for all $t \in [\ell,r]$ as well as
$x^i \in (f(\ell), g(\ell))$ and $y^i \in (f(r),g(r))$ for all $i$, we
may also define $\atmeasure^{\x,\y,\lambda}_{n,\ell,r,f,g}$ to be the measure with
a floor $f$ and a ceiling $g$, defined analogously.

As a final remark on Definition \ref{def:finiteLE},
note that the constant $2$ inside the exponential
may be replaced by an arbitrary constant $a > 0$ throughout, as was done in previous works.
All of our results continue to hold for any $a > 0$, but we will not need to
emphasize this overall tilt strength parameter in the present work so we omit it for clarity.
The value $2$ is chosen for consistency with \cite{FS05Constrained}, and will be further commented
on later in Section \ref{sec:inputs_tailbounds}.

\subsubsection{Infinite-volume limits}
\label{sec:intro_setup_infinite}

One may obtain infinite volume line ensembles by taking a local
weak limit of finite ones. For instance, the  limit of
$\atmeasure^{\mathbf{0},\mathbf{0},\lambda}_{n,-T,T}$ as $T \to \infty$,
which we will denote simply by $\atmeasure_n^\lambda$ is an infinite volume
ensemble with $n$ lines. That the local limit exists in this case is a
consequence of monotonicity and this is argued precisely in \cite{CIW18Tightness}. 
This may be viewed as a diffusion process in $\A_+^n$, and it turns out to be the
Langevin diffusion whose stationary distribution may be expressed implicitly in terms of
the solution to a certain PDE; see \cite{IVW18Dyson,DLZ24Limiting} for more information.
While we will mostly be interested in the $n=\infty$ case (again defined
through monotone limits) for the moment let us discuss the simplest case of
$n=1$ which will feature centrally in our analysis.

\begin{remark}[The Ferrari--Spohn diffusion]
\label{rmk:fs}
In the special case of $n=1$, the measure $\atmeasure_1^\lambda$
does not depend on $\lambda$ and describes the celebrated \emph{Ferrari--Spohn diffusion},
which was introduced by \cite{FS05Constrained} as the local weak limit of
Brownian bridges conditioned to lie above an increasingly large semicircle or parabola
(after shifting down by the floor). That this shift causes an area tilt is a
consequence of the Girsanov theorem (we expand on this further in
Section \ref{sec:topline_strategy}).
Since then, the Ferrari--Spohn diffusion has been shown to arise as a scaling limit of 
a variety of area-tilted random walks \cite{ISV15Invariance}, as well as the interface
in the low-temperature 2D Ising model in the critical pre-wetting regime
\cite{GG21Local,IOSV22Critical}.
We will often denote the stationary Ferrari--Spohn diffusion by $\FS$,
though we will also use the term ``Ferrari--Spohn diffusion'' more generally to refer to
instances of Definition \ref{def:finiteLE} with $n=1$, having boundary conditions on a finite domain.
\end{remark}

\subsubsection{The infinite-line ensemble}

We could simply define the $\lambda$-tilted line ensemble measure $\atmeasure_\infty^\lambda$
as the local \emph{monotone} weak limit of the measures $\atmeasure_n^\lambda$
as $n \to \infty$ and this is indeed how it was first constructed in
\cite{CIW18Tightness}. Thus, $\atmeasure_\infty^\lambda$ should be thought of
as the \emph{zero-boundary} infinite volume line ensemble.
However, it was shown in \cite{CG25Uniqueness} to also be the limit of various other
finite line ensembles.
To help build further intuition let us also present an axiomatic viewpoint
introduced by \cite{CG25Uniqueness} which \emph{uniquely characterizes} the
distribution of the $\lambda$-tilted ensemble $\X$.

We begin by introducing the main property which the $\lambda$-tilted line ensemble $\X$
is characterized by, the \emph{$\lambda$-tilted Brownian-Gibbs property} or $\lambda$-BG
property for short.
To set up some notation for this, for any $\x \in \A_+^\infty$ and $n \in \N$,
let us define $\x^{\leq n} = (x^1,\dotsc,x^n) \in \A_+^n$ and
$\x^{>n} = (x^{n+1},x^{n+2},\dotsc) \in \A_+^\infty$, and for any
$\y \in \A_+^n$ and $\x \in \A_+^\infty$ let us define
$\y \x =(y^1,\dotsc,y^n,x^1,x^2,\dotsc) \in \A_+^\infty$. 
We will also use natural extensions of these notations to functions with values
in $\A_+^n$ or $\A_+^\infty$. For instance, $\X^{\le n}$ will denote the top $n$
curves of the ensemble $\X.$
Further, for a line ensemble $\X : \R \to \A_+^\infty$ (i.e. a random continuous function)
let us define
\begin{equation}
\label{eq:defextalg}
    \extalg^n_{\ell,r}(\X) \coloneqq \sigma \left(
        X^i(t) : \text{either } t \notin (\ell,r) \text{ or } i > n
    \right),
\end{equation}
the \emph{exterior $\sigma$-algebra}, for any $\ell < r$ and $n \in \N$.

\begin{definition}[$\lambda$-tilted Brownian-Gibbs property]
\label{def:BG}
A line ensemble $\X : \R \to \A_+^\infty$ satisfies the $\lambda$-BG property if
for any $n \in \N$, any $\ell < r$, and any
bounded measurable function $F$ on the space of paths $[\ell,r] \to \A_+^\infty$, we have
\begin{equation}
\label{eq:BG}
    \E_\X \left[ F(\X) \middle| \extalg^n_{\ell,r}(\X) \right]
    = \E_{\Y \sim \cL^{\X^{\leq n}(\ell), \X^{\leq n}(r),\lambda}_{n,\ell,r,X^{n+1}|_{[\ell,r]}}}
    \left[F(\Y \X^{> n}) \right]
\end{equation}
almost surely as $\extalg^n_{\ell,r}(\X)$-measurable random variables.
\end{definition}

In other words, $\X$ satisfies the $\lambda$-BG property if its distribution is fixed by
the operation of erasing the curves $X^i(t)$ for $t \in (\ell,r)$ and $i \leq n$, and then resampling
them as Brownian bridges under geometrically increasing area tilts conditioned to not intersect.

Note that the $\lambda$-BG property only specifies the conditional law of finitely many paths,
and shifting a line ensemble with the $\lambda$-BG property up deterministically by a constant
would result in another ensemble with the $\lambda$-BG property.
To remove this source of non-uniqueness, we will require that the lines tend to zero as the
index increases.

\begin{definition}[Asymptotically pinned to zero]
\label{def:APZ}
A line ensemble $\X$ is asymptotically pinned to zero if for any $ \eps, T > 0$ there is some
$k \in \N$ such that
\begin{equation}
    \P \left[ \sup_{t \in [-T,T]} X^k(t) \leq \eps \right] \geq 1 - \eps.
\end{equation}
\end{definition}

As it turns out, there is one more source of non-uniqueness for line ensembles satisfying
both Definitions \ref{def:BG} and \ref{def:APZ}.
Namely, there are such ensembles for which the top line grows to infinity like
\begin{equation}
    t^2 + L |t| \text{ as } t \to -\infty
    \qquad \text{ and } \qquad
    t^2 + R t \text{ as } t \to \infty
\end{equation}
whenever $L + R < 0$; see \cite{BCG25Characterizing} for more information.
To rule out this behavior, we introduce one final assumption, that the top line
does not grow to infinity.

\begin{definition}[Uniformly tight]
\label{def:UT}
A line ensemble $\X$ is uniformly tight if for any $\eps > 0$ there is some $C$ such that,
for all $t \in \R$ we have
\begin{equation}
    \P \left[X^1(t) \geq C\right] \leq \eps.
\end{equation}
\end{definition}

Now \cite[Theorem 3.7]{CG25Uniqueness} states that there is a \emph{unique}
line ensemble $\X$ which satisfies the $\lambda$-BG property, is asymptotically pinned
to zero, and is uniformly tight.

\begin{definition}[The $\lambda$-tilted line ensemble]
\label{def:main}
The $\lambda$-tilted line ensemble $\X$ is the unique line ensemble
which satisfies Definitions \ref{def:BG}, \ref{def:APZ}, and \ref{def:UT}.
\end{definition}

In the sequel we will also often use the phrase``$\lambda$-tilted'' for a line
ensemble (with finitely or infinitely many lines) simply to denote that they have
geometrically increasing area tilts with ratio $\lambda$.

%% file: sections/1_intro_2_results.tex
\subsection{Results}
\label{sec:intro_results}
Given the above preparation we can now state our main results. 

\subsubsection{Correlation decay}
\label{sec:intro_results_correlation}

\begin{theorem}[Exponential decay of correlations]
\label{thm:correlation}
    There are some $\lambda_0 > 1$ and $C, \gamma > 0$ such that for all $\lambda \geq \lambda_0$
    and all $i, j \in \N$,
    if $\X$ denotes the $\lambda$-tilted stationary line ensemble
    then for all $t > 0$, we have
    \begin{equation}
        \label{eq:mainbound}
        0 \leq \Cov[X^i(0), X^j(t)] \leq C e^{-\gamma t} \lambda^{-(i+j-2)/3}.
    \end{equation}
    The same bound holds if $\X$ is replaced by the stationary
    $n$-line $\lambda$-tilted ensemble $\X_n \sim \atmeasure_n^\lambda$, for any $n$.
\end{theorem}

A couple of remarks are in order. First note that this is a significant
improvement over the prior best-known bound \cite[Theorem 3.5]{CG25Uniqueness},
which stated that
\begin{equation}\label{previous1234}
    \left| \Cov[X^1(0), X^1(t)] \right| \leq C \Exp{- c (\log t)^{3/7}}
\end{equation}
for some constants $C, c > 0$. However, the latter was stated for any
$\lambda > 1$, while the above theorem assumes $\lambda$ to be large enough.  We
expect the bound \eqref{eq:mainbound} to be true for all $\lambda>1$ too with $\gamma \to 0$ as
$\lambda \to 1$. We expand on this later in Section \ref{sec:closeto1}.
Additionally, we state the result for correlations between different lines,
not just the top line, and this generalization will be used crucially in the proof of our next result.

Next, note that our result is an upper bound on the covariance rather than the correlation.
For $i,j$ fixed (not depending on $t$), this immediately implies a similar bound
on the correlation as well since they have unit order variances, but for
arbitrary $i,j$ one would need to argue a lower bound on the fluctuations of
$X^i(0)$ and $X^j(t)$ of orders $\lambda^{-(i-1)/3}$ and $\lambda^{-(j-1)/3}$
respectively. Unfortunately, such a lower bound is not recorded anywhere in the
literature.  While it may be proved by carefully analyzing the partition
function of the model, we do not pursue this presently, and further the above
formulation will already suffice for the spectral gap analysis. 

Finally, the nonnegativity of the covariance in \eqref{eq:mainbound} is a
consequence of the FKG inequality as mentioned below \eqref{eq:37}.
This inequality has not yet appeared in the literature for area-tilted line ensembles,
but we need it for our next result as will be explained in Section \ref{sec:intro_iop_spectral},
so we state and prove it in Section \ref{sec:spectralgap}.

\subsubsection{Spectral gap for the diffusion process}
\label{sec:intro_results_spectralgap}

Our second result is a lower bound on the spectral gap for the finite-line process 
which is uniform in the number of lines. Since we are
in an infinite-dimensional setting, some care has to be exercised to define
things properly which is what we first turn to.
Consider the $n$-line $\lambda$-tilted ensemble $\X_n$ for some $n \in \N$.
As mentioned in Section \ref{sec:intro_setup_finite}, this is a stationary
Langevin diffusion process with an invariant distribution $\pinl$ which may be expressed
implicitly as {the square of} the solution to a particular PDE \cite[end of Section 1]{DLZ24Limiting}.
We will consider the generator $\gennl$ of this diffusion process,
which is defined by the following equation, where $\x \in \A_+^n$:
\begin{equation}
\label{eq:gennldef}
    (\gennl f)(\x) \coloneqq \lim_{t \to 0} \frac{\E \left[ f(\X_n(t)) \middle| \X_n(0) = \x \right] - f(\x)}{t}.
\end{equation}
Note that this limit may not exist, even for $f \in L^2(\A_+^n, \pinl)$, but it is well-defined
on a dense subspace $\cD(\gennl)$, e.g.\ by the Hille--Yosida theorem
\cite[Theorem II.3.5]{EN00One}. The latter is a general statement which applies
as soon as the Markov semigroup exhibits continuity in a strong sense at $0$
which in this case is a consequence of the almost sure continuity of the paths
in $\X_n.$ We refer the reader to \cite{EK09Markov} for the complete functional analytic
details.
Thus $\gennl$ is a densely defined unbounded operator, and moreover it is
self-adjoint and negative semidefinite
with respect to the inner product of $L^2(\A_+^n,\pinl)$, which is given by
\begin{equation}
    \left< f, g \right>_\pinl \coloneqq \E\left[f(\X_n(0)) g(\X_n(0)) \right].
\end{equation}
As $\pinl$ is a finite measure, the constant vectors are in this Hilbert space and
form a one-dimensional $0$-eigenspace for $\gennl$. Given this, we define the
\emph{spectral gap} of $\gennl$ through the commonly adopted variational
characterization as
\begin{equation}
\label{eq:gapnldef}
    \gapnl \coloneqq \inf \left\{
        \left< - \gennl f, f \right>_\pinl :
        f \in {\cD(\gennl)} \text{ with } \left< f, 1 \right>_\pinl = 0
        \text{ and } \left< f, f \right>_\pinl = 1
    \right\}.
\end{equation}
Note that because $\gennl$ is densely defined we must have $\gapnl < \infty$.

\begin{theorem}[Uniformly positive spectral gap]
\label{thm:spectralgap}
There are some $\lambda_0 > 1$ and $\gamma > 0$ such that
for all $\lambda \geq \lambda_0$ and all $n \in \N$,
the spectral gap $\gapnl$ of the generator $\gennl$
of the $n$-line $\lambda$-tilted diffusion is at least $\gamma$.
\end{theorem}
As may be expected given the suggestive notation, the values of $\lambda_0$ and $\gamma$ in
this result may be taken to be the same as those in Theorem
\ref{thm:correlation}. Further, as remarked in \cite{CG25Uniqueness}, it is
natural to view the infinite-line
$\lambda$-tilted ensemble of Definition \ref{def:main} as an \emph{infinite-dimensional}
diffusion process.
This perspective has not been made fully rigorous, but Theorem \ref{thm:spectralgap}
may be viewed heuristically as providing a spectral gap for the generator of this
infinite-dimensional diffusion.

The only prior work on the spectral gap was \cite[Lemma 2.1]{DLZ24Limiting}, which proved
that the spectral gap is positive for any $n$ by showing that $e^{\gennl}$
is a \emph{compact operator} on $L^2(\A_+^n)$, in particular, that it has discrete spectrum with
no accumulation points other than zero. Note that this already implies that the
correlations in the $n$-line process decay exponentially, albeit with possibly
an $n$ dependent rate which could go to zero as $n\to \infty$ and hence prevents
any assertion about the ensemble with infinitely many lines.

%% file: sections/1_intro_3_iop.tex
\subsection{Proof ideas}
\label{sec:intro_iop}

In this section we provide a brief sketch of the key ideas in the proofs of
Theorems \ref{thm:correlation} and \ref{thm:spectralgap}.

\vspace{3mm}

At a high level, in order to prove a quantitative version of independence at times $0$ and $t$,
we will compare $\X$ to an auxiliary ensemble $\Y$ which has full independence at these times.
Specifically, we will let $\Y$ be a $\lambda$-tilted ensemble which is pinned to zero
at times $-\frac{t}{2}$, $\frac{t}{2}$, and $\frac{3t}{2}$, so that its behavior on
the interval $\left[-\frac{t}{2},\frac{t}{2}\right]$, centered at $0$, is completely independent
from that on the interval $\left[\frac{t}{2}, \frac{3t}{2}\right]$, centered at $t$.
See Figure \ref{fig:auxiliary_ensemble} for an illustration of this ensemble.

\begin{figure}
    \centering
    \includegraphics[width=0.75\textwidth]{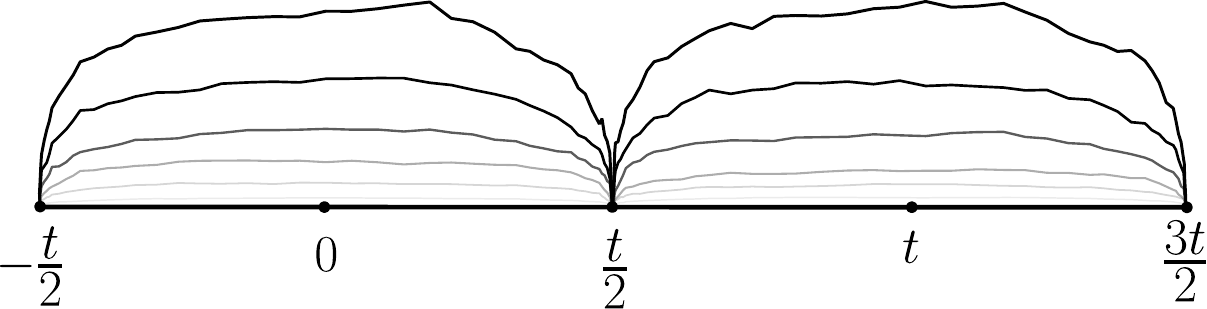}
    \caption{
        The auxiliary line ensemble $\Y$ with pinning which
        ensures that $\Y(0)$ and $\Y(t)$ are independent.
    }
\label{fig:auxiliary_ensemble}
\end{figure}

\subsubsection{Comparison via reversing stochastic dominance}
\label{sec:intro_iop_reverse}

As will be discussed in Section \ref{sec:inputs_monotonicity} below, there is a
monotonicity property for area-tilted line ensembles which implies that
$\X$ stochastically dominates $\Y$.
In other words, there is a coupling between the two
such that $X^k(s) \geq Y^k(s)$ for all $k \in \N$ and all $s \in \left[-\frac{t}{2},\frac{3t}{2}\right]$.
If we can find a coupling which \emph{reverses} this ordering at times $0$ and $t$,
then the joint behavior of $\left( \X(0), \X(t) \right)$ must be similar to that
of $\left( \Y(0), \Y(t) \right)$, which is independent.
In order to achieve this, we first sample $\X$ and $\Y$ independently and then try to find
\emph{random} times $\tau_\ell^0$, $\tau_r^0$, $\tau_\ell^t$, $\tau_r^t$ such that
\begin{equation}
    -\tfrac{t}{2} < \tau_\ell^0 < 0 < \tau_r^0 < \tfrac{t}{2} < \tau_\ell^t < t < \tau_r^t < \tfrac{3t}{2}
\end{equation}
and $X^k(\tau) \leq Y^k(\tau)$ for all $k \in \N$ and $\tau \in \{\tau_\ell^0,
\tau_r^0, \tau_\ell^t, \tau_r^t\}$.  We may then resample both $\X$ and $\Y$ on
the intervals $\left[\tau_\ell^0, \tau_r^0\right]$ and $\left[\tau_\ell^t,
\tau_r^t\right]$ and use another application of the monotonicity property to
couple these resamplings to ensure that the reversed ordering at the boundaries
of the intervals is maintained throughout their interiors.

It is worth pointing out that, while this basic strategy was also employed by
\cite{CG25Uniqueness} to prove \eqref{eq:37}, the key new idea in this paper
entails estimating the \emph{probability} of existence of the above stopping
times, since it is the latter which directly controls the resulting correlation
bound.  In \cite{CG25Uniqueness}, each integer point $r \in \left[-\frac{t}{2},
\frac{3t}{2}\right] \cap \Z$ was considered separately as a possible time of
reversal.  However, having $k$ lines reversed at any given integer point has
probability which is exponentially small in $k$ and hence $k$ must be
logarithmic in $t$. Since there are infinitely many lines, in \cite{CG25Uniqueness}
a priori control was used to estimate the effect of all the lines with index
bigger than $k$ and the eventual conclusion was the  bound presented in
\eqref{previous1234}.

The key realization is that one must take advantage of the spatial correlations
inherent to this model to improve upon the above bound.  The precise observation
we rely on is the fact that when the top lines of $\X$ and $\Y$ reverse, they
will remain reversed for some amount of time, which allows the second lines to
reverse, et cetera. We next see how to devise a strategy based on this
intuition.

\subsubsection{A branching process for finding reversal times}
\label{sec:intro_iop_branching}

As will be discussed in Section \ref{sec:inputs_scaling}, there is a ``1,2,3''-type
scaling relation 
for area-tilted lines relating the area tilt strength to the typical height and
the characteristic width of its Brownian-type correlations.
This ultimately leads to the fact that $X^{k+1}$, which has an area tilt of
strength $2 \lambda^k$, locally looks like a Ferrari-Spohn diffusion on the
spatial scale $\Theta(\lambda^{-2k/3}),$  with height fluctuating at scale
$\Theta(\lambda^{-k/3})$.  The same is true for $Y^{k+1}$.
The expert reader might already recognize this as strongly reminiscent of the
scaling one sees in models in the Kardar-Parisi-Zhang (KPZ) universality class
\cite{Q11Introduction}. This is not a coincidence. The area tilted line involves (via a
Girsanov transformation, see \eqref{eq:girsanov} in Section \ref{sec:topline}
below) a characteristic interaction between
local Brownian fluctuations and a global parabolic constraint which is also the
central feature determining the KPZ universality class.

This natural scale of fluctuations suggests that for any interval of length $\lambda^{-2k/3}$ there is some
positive probability $p$ for the curves $X^{k+1}$ and $Y^{k+1}$ to remain reversed on the whole interval.
We will refrain from attempting to optimize $p$, but rather will establish that,
importantly, $p$ does not depend on $\lambda$ or $k$ as the event being
considered is ``on-scale'' in that after rescaling all relevant quantities are
of order $1$.

Thus we may assemble these events into a branching process where the second line gets $\lambda^{2/3}$
tries to reverse while the first line remains reversed on its interval of length $1$.
Similarly, the third line gets $\lambda^{2 \cdot 2/3}$ tries to reverse under each interval
of length $\lambda^{-2/3}$ where the second line is reversed, et cetera.
See Figure \ref{fig:branching_process} for an illustration of this branching process.

As long as $\lambda^{2/3} p > 1$, this branching process is supercritical meaning that
with positive probability there is some infinite branch, and an infinite branch corresponds to
a point $s$ where $X^k(s) \leq Y^k(s)$ for all $k \in \N$.
Armed with this, we will repeat this entire branching process in each unit length interval
to conclude that the desired reversal times exist as soon as one of the branching
processes (which will be independent by construction) succeeds.
This will
occur with probability $1 - C e^{-ct}$, leading to Theorem
\ref{thm:correlation}.
See Figure \ref{fig:manytrials} for an illustration of this last step.

\begin{figure}
    \centering
    \includegraphics[width=0.4\textwidth, align=b]{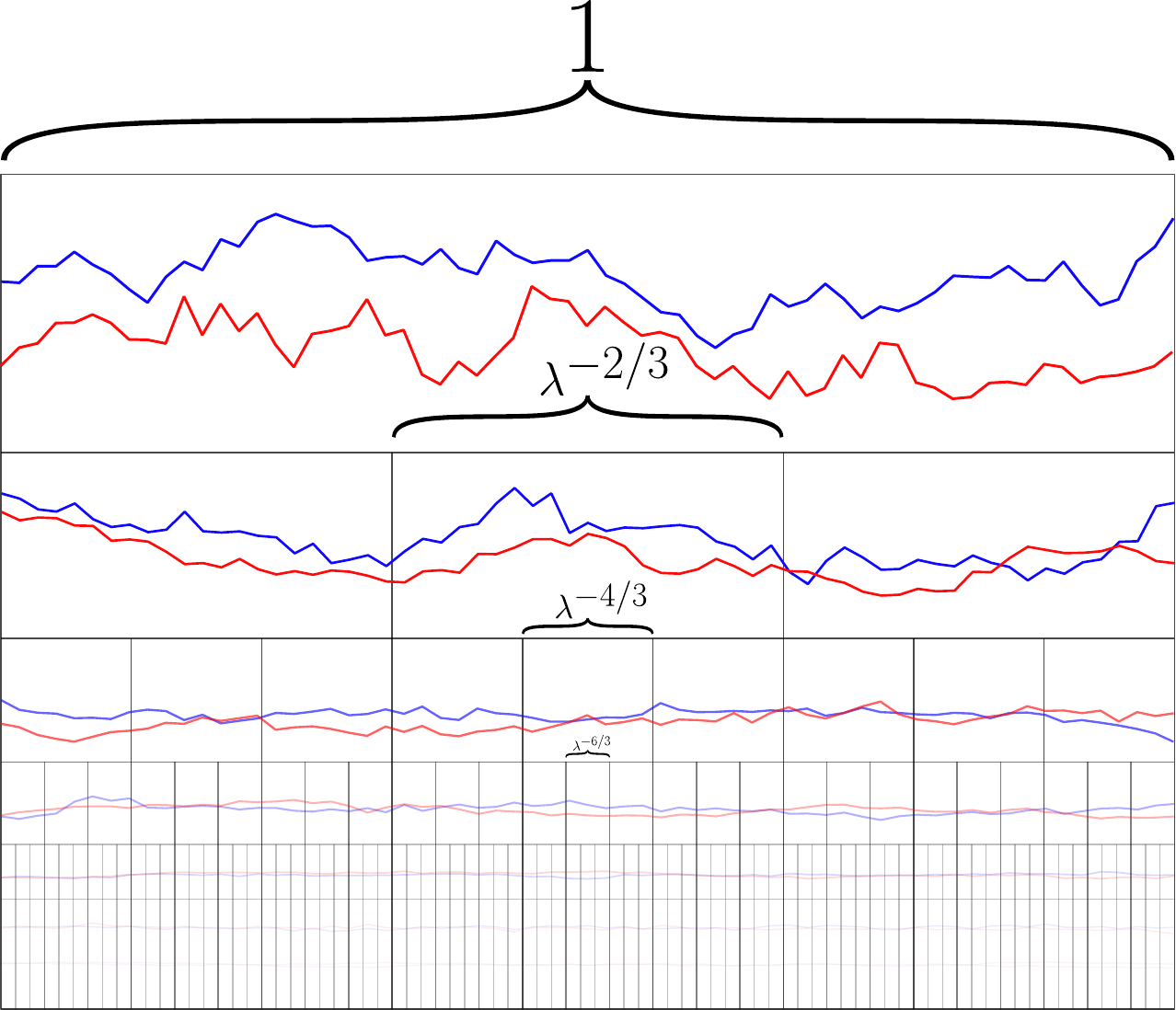}
    \hspace{1cm}
    \includegraphics[width=0.4\textwidth, align=b]{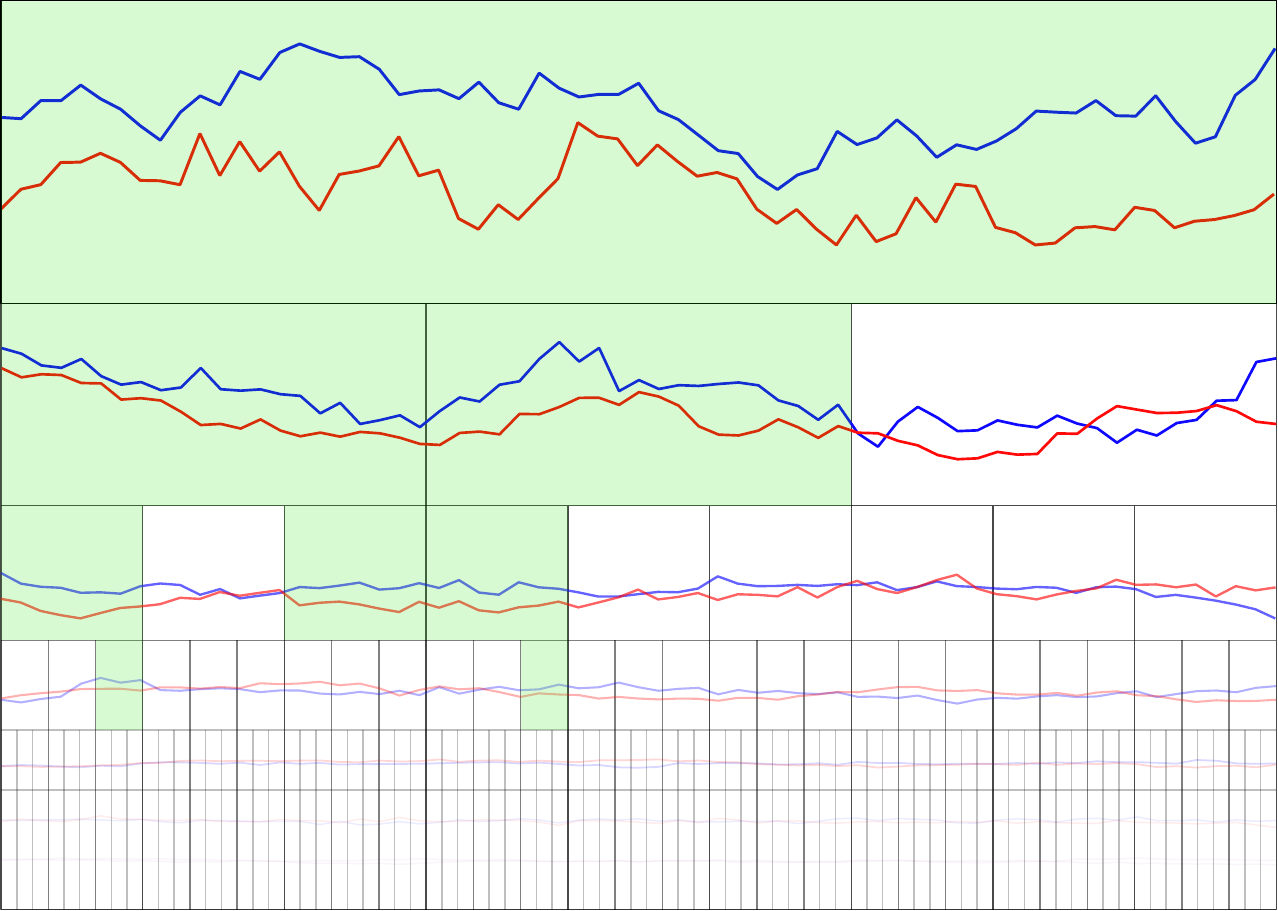}
    \caption{
        Left: the branching process structure of the reversal events
        for lines at the correct scales.
        Here $\X$ is represented in red while $\Y$ is represented in blue.
        Right: the initial few success nodes in the branching process are
        highlighted in green.
    }
\label{fig:branching_process}
\end{figure}

\subsubsection{Uniform spectral gap}
\label{sec:intro_iop_spectral}

Theorem \ref{thm:spectralgap} will be proved in Section \ref{sec:spectralgap}
relying on Theorem \ref{thm:correlation}.
The key observation is that the spectral gap is controlled by the
decay of correlations, not just of the lines in $\X_n$, but of more general \emph{functions}
of the values of $\X_n$ at different times. As will be shown, by standard
approximation theory, one can assume the function to be Lipschitz. 
The key step in the proof then will be to show that for \emph{any} Lipschitz continuous function
$h : \A_+^n \to \R$ there is some $C_h$ for which we have
\begin{equation}
\label{eq:lipschitzdecay}
    \left| \Cov \left[ h(\X_n(0)), h(\X_n(t)) \right] \right| \leq C_h e^{- \gamma t}.
\end{equation}
This is explicitly stated as Proposition \ref{prop:lipschitzdecay} below,
and may be of independent interest as a stronger form of Theorem \ref{thm:correlation}.
A key step in the proof of \eqref{eq:lipschitzdecay} is an application of a classical
covariance inequality \cite{N80Normal,BS98Asymptotical} which holds for systems
satisfying the FKG inequality. While various monotonicity properties for
area-tilted line ensembles have appeared in the literature, as will be discussed
in Section \ref{sec:inputs_monotonicity} below, we could not find a statement of
the FKG inequality in the form which we will need. Hence, we include the proof
in Proposition \ref{prop:FKG} for completeness.

\subsubsection{The case of $\lambda$ close to $1$}
\label{sec:closeto1}

We end this section with a brief discussion of the correlation decay of $\X$
when $\lambda$ is close to $1$. It is known, see for instance \cite{FS23Airy,DS25Uniform}, that when
$\lambda=1$, the ensemble $\X_n$, up to a global shift, essentially behaves as a
DBM with $n$ lines on a domain of size $n^{1/3}$, with fluctuations on scale
$n^{2/3}$. One intuitive way to see this is via the Girsanov transformation (see
Remark \ref{rmk:fs} or \eqref{eq:girsanov}), which converts the model into $n$
non-intersecting Brownian motions above a standard parabola, i.e., a DBM above a
parabola. Indeed, without the parabola, a DBM with $n$ paths over a domain of
size $t$ fluctuates on the scale $\sqrt{nt}$. Balancing this with the effect of
the parabolic curvature, which is of order $t^2$, gives $ \sqrt{nt}\approx t^2,
$ and hence $t\approx n^{1/3}$. It can also be shown that this model exhibits
exponential correlation decay on precisely this scale.

Let us now consider how this picture might inform the behavior of the
$\lambda$-tilted model when $\lambda\approx 1$. Write $\lambda=1+\varepsilon$.
Since $ \lambda^i \approx e^{\varepsilon i}, $ the area-tilt strength remains of
order one for $i=O(\varepsilon^{-1})$. Thus one might speculate that this model
exhibits a coarse version of the scale separation but now across every block of
$O(\varepsilon^{-1})$ consecutive lines, with the top block consisting of
$O(\varepsilon^{-1})$ lines and approximately exhibiting the $\lambda=1$
behavior. The above then suggests that the correlation decay of the full system
is governed by that of this top block and hence we should obtain exponential decay of
correlations at scale $O(\varepsilon^{-1/3}).$ We leave this to future work.

\begin{figure}
    \centering
    \includegraphics[width=0.85\textwidth]{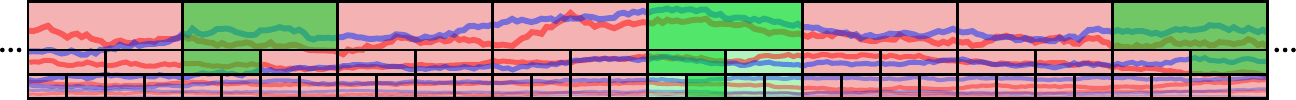}
    \caption{
        We repeat the branching process pictured in Figure \ref{fig:branching_process} in
        each unit length interval, after applying a construction to make these trials independent.
        Since each one has some positive chance to have an infinite branch,
        if there are $\Omega(t)$ independent trials then there will be a successful trial somewhere
        with probability at least $1 - C e^{-ct}$.
    }
\label{fig:manytrials}
\end{figure}

%% file: sections/1_intro_4_acknowledgements.tex
\subsection{Acknowledgements}
\label{sec:intro_acknowledgements}
SG thanks Pietro Caputo and Milind Hegde for useful discussions. 
SG was supported by a Miller Research Professorship at the Miller Institute for Basic Research in Science and NSF Career grant-1945172.
VW was supported by the NSF Graduate Research Fellowship grant DGE 2146752.

%% file: sections/2_inputs_0.tex
\section{Inputs from the literature}
\label{sec:inputs}

Before proceeding with our main analysis, in this section we collect
a few key inputs from the literature on area-tilted line ensembles,
as well as some mild extensions thereof.
These include the strong area-tilted Brownian Gibbs property in Section \ref{sec:inputs_strongbg}
allowing us to resample on so-called \emph{stopping domains},
as well as  the important monotonicity and scaling relations in Sections \ref{sec:inputs_monotonicity}
and \ref{sec:inputs_scaling} respectively,
and some upper tail estimates for $\X$ in Section \ref{sec:inputs_tailbounds}.

Finally, in Section \ref{sec:inputs_comingdown} we produce a bound showing that the top line
of a $\lambda$-tilted line ensemble will come down to an $O(1)$ height in the middle
of the interval $[-T,T]$ under appropriate boundary conditions. 
This follows by iterating \cite[Proposition 5.5]{BCG25Characterizing}, which
states that it will come down to a height of $O(T^\delta)$ for any $\delta > 0$.

\input{sections/2_inputs_1_strongbg}
\input{sections/2_inputs_2_monotonicity}
\input{sections/2_inputs_3_scaling}
\input{sections/2_inputs_4_tailbounds}
\input{sections/2_inputs_5_comingdown}

%% file: sections/2_inputs_1_strongbg.tex
\subsection{Strong area-tilted Brownian Gibbs property}
\label{sec:inputs_strongbg}

Much like the situation discussed in the pioneering work \cite{CH14Brownian} on line ensembles with
a Brownian Gibbs property (with no area tilt), the $\lambda$-BG property of Definition \ref{def:BG}
extends to a \emph{strong} $\lambda$-BG property, where when sampling the top $n$ lines
we may replace the deterministic domain $[\ell,r]$ by a random \emph{stopping domain}
$[\tau_\ell,\tau_r]$ which satisfies $\{ \tau_\ell \leq t \} \cap \{ \tau_r \geq
s \} \in \extalg^n_{t,s}(\X)$ for any $t < s$, recalling the exterior $\sigma$-algebra
defined in \eqref{eq:defextalg}.
This follows from the argument for Lemma 2.5 in \cite{CH14Brownian}.
In fact, the same argument as presented there also shows that we may extend the
notion of stopping domain to allow it to depend on external randomness, in some cases.

We now briefly expand upon this idea; as it is somewhat standard by now, for instance being used
already in \cite{CG25Uniqueness}, we will be somewhat informal.
Suppose that we have a family of $\sigma$-algebras
$\{ \extextalg_{\ell,r}^n : \ell < r, n \in \N \}$
for which \eqref{eq:BG} holds with $\extalg_{\ell,r}^n(\X)$ replaced by $\extextalg_{\ell,r}^n$.
If the random variables $\tau_\ell, \tau_r$ satisfy
\begin{equation}
    \{ \tau_\ell \leq t \} \cap \{ \tau_r \geq s \}
    \in \extextalg_{t,s}^n
\end{equation}
for all $t \leq s$,
then we may define the $\sigma$-algebra
\begin{equation}
    \extextalg_{\tau_\ell,\tau_r}^n = \left\{ A : A \cap \{ \tau_\ell \leq t \} \cap \{ \tau_r \geq s \} \in
    \extextalg_{t,s}^n \text{ for all } t \leq s \right\}.
\end{equation}
With this definition, the proof of \cite[Lemma 2.5]{CH14Brownian} shows that for
any bounded measurable function $F$ on the space of tuples $\{(\ell,r,\X) : \ell
< r, \X \text{ a continuous function } [\ell,r] \to \A_+^\infty \}$, we have
\begin{equation}
\label{eq:strongBG}
    \E_{\X} \left[ F(\tau_\ell, \tau_r, \X) \middle| \extextalg_{\tau_\ell, \tau_r}^n \right]
    = \E_{\Y \sim \atmeasure^{\X^{\leq n}(\tau_\ell), \X^{\leq n}(\tau_r), \lambda}_{n,\tau_\ell,\tau_r,X^{n+1}|_{[\tau_\ell,\tau_r]}}}
    \left[ F(\tau_\ell, \tau_r, \Y \X^{>n}) \right]
\end{equation}
as $\extextalg_{\tau_\ell,\tau_r}^n$-measurable random variables.
{For the right-hand side above,
recall the notation in Definition \ref{def:finiteLE} of the finite $\lambda$-tilted line ensembles
as well as the notation preceding Definition \ref{def:BG}.}

In practice, we will have two independent ensembles $\X$ and $\Y$ and we will
use the extended exterior $\sigma$-algebras
\begin{equation}
    \extextalg^n_{\ell,r} = \extalg^n_{\ell,r}(\X) \otimes \extalg^n_{\ell,r}(\Y).
\end{equation}
Because $\X$ and $\Y$ are independent, the standard
Brownian Gibbs property does hold for this choice, and so we may also
use the strong Brownian Gibbs property with a domain which is a stopping domain
with respect to this class of $\sigma$-algebras to resample either $\X$ or $\Y$ (or both).

Finally, by conditioning on the behavior of the top lines to stay the same as before resampling,
we may remove the restriction that we must resample \emph{all} of the top $n$ lines.
We will use this in our arguments freely by saying that we are resampling only a particular line or
lines.
When resampling the lines with indices between $m$ and $n$, for instance, we will condition on
the following extended notion of the external $\sigma$-algebra
\begin{equation}
    \extalg^{m,n}_{\ell,r}(\X) \coloneqq
    \sigma \left( X^k(s) : k \notin [m,n] \text{ or } s \notin (\ell,r) \right),
\end{equation}
or a version of this which allows for external randomness as above.
If $I = [\ell,r]$, we will often use $\extalg^{m,n}_I$ to denote the same $\sigma$-algebra.

%% file: sections/2_inputs_2_monotonicity.tex
\subsection{Monotonicity}
\label{sec:inputs_monotonicity}

Area-tilted line ensembles satisfy a very useful monotonicity property which states, informally,
that when one raises the boundary conditions, floor or ceiling, or decreases the area tilt strength,
the height of curves will stochastically increase.

To state this formally, let us define the partial order $\preceq$ between functions
$\f, \g : [\ell,r] \to \A_+^n$ by $\f \preceq \g$
if $f^i(t) \leq g^i(t)$ for $i = 1, \dotsc, n$ and for all $t \in [\ell,r]$.
Similarly, $\f \prec \g$ if the nonstrict inequality is replaced by a strict inequality.
For two probability measures $\cL$ and $\cM$ on functions $[\ell,r] \to \A_+^n$,
we will say that $\cL \preceq \cM$, i.e.\ $\cM$ stochastically dominates $\cL$
under the ordering $\preceq$ between functions, if the samples $\X_n \sim \cL$
and $\Y_n \sim \cM$ may be coupled so that $\X_n \preceq \Y_n$ almost surely.

A basic statement of monotonicity, recalling the notation of Definition \ref{def:finiteLE},
is that
\begin{equation}
    \atmeasure^{\x_-, \y_-, \lambda_-}_{n,\ell,r} \preceq
    \atmeasure^{\x_+, \y_+, \lambda_+}_{n,\ell,r}
\end{equation}
if $\x_- \preceq \x_+$, $\y_- \preceq \y_+$, and $\lambda_- \geq \lambda_+$.
We will however need a slight extension of the above result which allows for
more general area tilts which vary across space,
as well as differing floors and ceilings for each line in the ensemble.
In particular, we will need to slightly extend our definition of the
finite $\lambda$-tilted line ensemble to allow for these features.

\begin{definition}[Extended finite $\lambda$-tilted line ensemble]
\label{def:extFLE}
For $\x, \y \in \A_+^n$ and functions
$\f, \g: (\ell,r) \to \A_+^n$ with $\f \prec \g$ and $\bl : (\ell,r) \to \R^n$,
let  $\atmeasure^{\x,\y,\bl}_{n,\ell,r,\f,\g}$ denote the probability measure
with the following Radon--Nikodym derivative relative to the
Brownian bridge measure $\bridgemeasure^{\x,\y}_{n,\ell,r}$:
\begin{equation}
    \frac{d\atmeasure^{\x,\y,\bl}_{n,\ell,r,\f,\g}}{d\bridgemeasure^{\x,\y}_{n,\ell,r}}(\X_n)
    \propto \Exp{-2 \sum_{i=1}^n \int_\ell^r \lambda^i(t) X^i_n(t) \,dt}
    \ind{\X_n(t) \in \A_+^n \text{ for all } t \in [\ell,r]}
    \ind{\f \prec \X_n \prec \g}.
\end{equation}
Note that to recover the geometrically increasing area tilts, we should
take $\bl(t) \equiv (1,\lambda,\lambda^2,\dotsb,\lambda^{n-1})$.
\end{definition}

Now we may state the general monotonicity result which will be used frequently
in the sequel, often simply referred to as ``monotonicity'' without referencing
the following lemma directly.
This lemma is stated in \cite[Lemma 1.4]{CIW19Confinement} which in turn is based
on \cite[Lemma 2.6]{CH14Brownian}.

\begin{lemma}[Monotonicity]
\label{lem:monotonicity}
If $\x_- \preceq \x_+$, $\y_- \preceq \y_+$, $\bl_- \succeq \bl_+$,
$\f_- \preceq \f_+$, and $\g_- \preceq \g_+$, then
\begin{equation}
    \atmeasure^{\x_-,\y_-,\bl_-}_{n,\ell,r,\f_-,\g_-}
    \preceq
    \atmeasure^{\x_+,\y_+,\bl_+}_{n,\ell,r,\f_+,\g_+}.
\end{equation}
\end{lemma}

\begin{remark}
\label{rmk:sameboundary}
As mentioned below Definition \ref{def:finiteLE}, this monotonicity principle
allows us to extend that definition or Definition \ref{def:extFLE} to allow for
boundary values $\x,\y \in \overline{\A_+^n}$ (i.e.\ with non-unique entries) by
taking a monotone limit of boundary conditions in the interior of $\A_+^n$.  We
may then extend this to ensembles with infinitely many lines on a finite domain
with boundary conditions which may not be pairwise distinct, which we will use
frequently in our arguments, via weak convergence.  For instance, this procedure
allows us to define the ensemble $\Y$ mentioned in Section \ref{sec:intro_iop}
which has infinitely many lines pinned to zero at three points. Thus, throughout
the rest of the paper we will freely use such auxiliary objects in our arguments
without further commenting on how they are defined.
\end{remark}

Another important consequence is that the ensemble $\overline{\X} = \{ X^2, X^3,
\dotsc \}$ is stochastically dominated by an ensemble of lines with area tilts
$\lambda, \lambda^2, \dotsc$ and no ceiling, which in turn is stochastically
dominated by $\X$ itself.  Inequalities of this type have featured crucially in
arguments {for other line ensembles} in the past.  See e.g.\
\cite{GH22Sharp,GHZ25Van}, where they have often been termed as a counterpart to
the van den Berg-Kesten (BK) inequality for line ensembles.  This property has
also been useful in prior work on area-tilted line ensembles specifically, in
\cite[Remark 2.6]{CG25Uniqueness}, and so we state it as a lemma.

\begin{lemma}
\label{lem:secondline_domination}
Let $\X$ be the stationary infinite-line $\lambda$-tilted ensemble.  Then for
any $k \geq 1$, the distribution of $X^{k+1}(t)$ conditioned on $X^1, \dotsc,
X^k$ (and hence unconditionally) is stochastically dominated by the top line of
the stationary ensemble with area tilts $\lambda^k, \lambda^{k+1}, \dotsc$.  In
particular,  $X^2$ even when conditioned on $X^1$ is stochastically dominated by
$X^1$.
\end{lemma}

We will only sketch the idea briefly here for $k=1$ with the same argument
working for any $k$.  Let $\Y = (Y^1, \dotsc, Y^n)$ be an ensemble with $n$
lines on a finite domain $[-T,T]$ with zero boundary conditions.  Then the top
line $Y^1$ acts as a nontrivial ceiling, and the distribution of $\overline{\Y}
= (Y^2, \dotsc, Y^n)$ conditioned on any realization of $Y^1$ is stochastically
dominated by the case when $Y^1$ is set to infinity, i.e., without any ceiling.
Then we may conclude the same for the zero boundary infinite ensemble by taking
a  (monotone) local weak limit by sending $n,T \to \infty$ to obtain the result.
Note that in this argument it is important that we use constant (in particular
zero) boundary conditions in the prelimit, as the distribution of random
boundary conditions may change when removing the ceiling.

%% file: sections/2_inputs_3_scaling.tex
\subsection{Scaling relation}
\label{sec:inputs_scaling}
As already mentioned in Section \ref{sec:intro_iop_branching},
the $\lambda$-tilted line ensemble satisfies a ``1,2,3''-type scaling relation,
in the sense that when the height is scaled by a factor of $\rho$,
the width should be scaled by a factor of $\rho^2$ and the area tilt should
be scaled by a factor of $\rho^{-3}$.
This identifies the correct scales for working with various lines in the ensemble,
which is crucial for our branching process setup.

To state this scaling relation rigorously, let us define the \emph{$\rho$-rescaling} $\f^{(\rho)}$
of a function $\f : [\ell,r] \to \A_+^n$ by setting
\begin{equation}
\label{eq:scaling}
    \f^{(\rho)}(t) \coloneqq \rho \cdot \f\left(\rho^{-2}t\right).
\end{equation}
Note that $\f^{(\rho)}$ is a function on $[\rho^2 \ell, \rho^2 r]$.

In the sequel, we will always apply the following scaling relation
for $\rho = \lambda^{-k/3}$, where $\lambda$ is the tilt strength in the
geometrically-increasing area tilted line ensemble of Definition \ref{def:finiteLE},
and $k$ is an integer.
Indeed, this is the setting of the proof of the scaling relation
in \cite[Lemma 1.1]{CIW19Confinement}.
However, the same proof applies to the more general setting of
Definition \ref{def:extFLE}, which allows for a slightly cleaner statement
which we present now.

\begin{lemma}[Scaling relation]
\label{lem:scaling}
For any $\rho > 0$, we have
\begin{equation}
    \X \sim \atmeasure^{\x,\y,\bl}_{n,\ell,r,\f,\g}
    \qquad \text{if and only if} \qquad
    \X^{(\rho)} \sim \atmeasure^{\rho \x, \rho \y, \rho^{-3} \bl^{(\rho)}}_{n, \rho^2 \ell, \rho^2 r, \f^{(\rho)}, \g^{(\rho)}}.
\end{equation}
\end{lemma}

This important relation will allow us to prove uniform bounds on various
probabilities of ``on-scale'' events appearing in our argument,
such as the event that the line $X^{k+1}$ in the $\lambda$-tilted line ensemble $\X$
remains above height $\Theta(\lambda^{-k/3})$ for time $\Theta(\lambda^{-2k/3})$.

%% file: sections/2_inputs_4_tailbounds.tex
\subsection{Upper tail bounds}
\label{sec:inputs_tailbounds}

In this section we present various upper tail bounds which will be used
throughout the article, both for the Ferrari--Spohn diffusion and the
infinite-line $\lambda$-tilted ensemble.

\subsubsection{Tail bounds for the Ferrari--Spohn diffusion}
\label{sec:inputs_tailbounds_fs}

Recall from Remark \ref{rmk:fs} that the Ferrari--Spohn diffusion $\FS$
is a stationary Langevin diffusion which is the local weak limit as $T \to \infty$
of the distribution in Definition \ref{def:finiteLE} with $n=1$.
As shown in \cite{FS05Constrained}, the stationary distribution of $\FS$
has a density which is proportional to
\begin{equation}
\label{eq:fssd}
    \Ai(x -\omega_1)^2 \cdot \ind{x > 0},
\end{equation}
where $\Ai$ is the Airy function and $-\omega_1$ is the first zero of $\Ai$.
Note that we keep the original conventions of \cite{FS05Constrained},
which is why we assumed in Definition \ref{def:finiteLE} that the constant
in front of the area tilt is $2$.
When comparing with other works on line ensembles with geometrically increasing
area tilts such as \cite{CG25Uniqueness,BCG25Characterizing},
the reader should note that this convention may not be the same.
The various constants, for instance in the tail bounds which we will state shortly,
differ under this change, but the values of these constants will not be important for
the present article.

Simply using the formula \eqref{eq:fssd} for the stationary distribution
and the known tail behavior of the Airy function \cite[Equation 10.4.59]{AS48Handbook}:
\begin{equation}
    \Ai(x) = \Exp{-\left(\tfrac{2}{3} - o(1)\right) x^{3/2}}
\end{equation}
as $x \to \infty$, one may derive a one-point upper tail bound for $\FS$:

\begin{lemma}[One-point upper tail bound for Ferrari--Spohn diffusion]
\label{lem:fs_uppertail_onepoint}
As $y \to \infty$, we have
\begin{equation}
    \P \left[ \FS(0) > y \right] \leq
    \Exp{- \left( \tfrac{4}{3} - o(1) \right) y^{3/2}}.
\end{equation}
\end{lemma}

We will also need to bound the maximum value of the Ferrari--Spohn diffusion
on an interval.
The following corollary was proved in \cite[Lemma 5.3]{CG25Uniqueness}.

\begin{corollary}[Interval upper tail bound for Ferrari--Spohn diffusion]
\label{cor:fs_uppertail_interval}
There exist constants $C,c > 0$ such that for all $S \geq 1$ and $y > 0$ we have
\begin{equation}
    \P \left[ \max_{s \in [-S,S]} \FS(s) > y \right] \leq
    C S e^{- c y^{3/2}}
\end{equation}
\end{corollary}

\subsubsection{Tail bounds for the infinite-line ensemble}
\label{sec:inputs_tailbounds_x}

In this subsection let us use $\X$ to denote the $\lambda$-tilted line ensemble
of Definition \ref{def:main}.
The optimal tail bound for the top line at a single point was pinned down
in \cite[Theorem 3.1]{CG25Uniqueness}, showing that to first order in the exponent
we have the same tail behavior as the Ferrari--Spohn diffusion.
Note that our constant differs from the statement of \cite{CG25Uniqueness}
due to our normalization assumption, but for the present work that constant
will not be important.

\begin{theorem}[One-point upper tail bound for $\lambda$-tilted line ensemble]
\label{thm:tailbound_onepoint}
As $y \to \infty$, we have
\begin{equation}
    \P \left[ X^1(0) > y \right] \leq \Exp{- \left(\tfrac{4}{3} - o(1)\right) y^{3/2}}.
\end{equation}
\end{theorem}

A useful corollary is an upper tail bound for the maximum of the top line on any
fixed interval.
This was proved as \cite[Corollary 5.5]{CG25Uniqueness}.

\begin{corollary}[Interval upper tail bound for $\lambda$-tilted line ensemble]
\label{cor:tailbound_interval}
There exist constants $C, c > 0$ such that for all $S \geq 1$ and $y > 0$,
we have
\begin{equation}
    \P \left[ \max_{s \in [-S,S]} X^1(s) > C \log(S) +  y \right]
    \leq C S e^{- c y^{3/2}}.
\end{equation}
\end{corollary}

Using Lemma \ref{lem:secondline_domination}, we will often apply this bound to lower lines as
well, after applying the scaling relation so that the area tilt strength of the $k$th line becomes
$2\lambda^0 = 2$.

%% file: sections/2_inputs_5_comingdown.tex
\subsection{Coming down estimate}
\label{sec:inputs_comingdown}

In this section we provide an estimate which shows that the top line of an area-tilted
ensemble will come down to an $O(1)$ height from height $T$ within time $T$.

\begin{lemma}[Improved coming down estimate]
\label{lem:comingdown}
    For any $\lambda_0 > 1$ there are constants $C, c > 0$
    such that the following holds for all $T > 0$ and $\lambda \geq \lambda_0$.
    Let $\X$ be a $\lambda$-tilted line ensemble on $[-T, T]$ with infinitely
    many lines, each with boundary conditions $\leq T$ at both ends of the interval.
    Then for any $M > 0$,
    \begin{equation}
        \P \left[
            X^1\left(0\right) \leq M
        \right]
        \geq 1 - C e^{- M^c}.
    \end{equation}
\end{lemma}

Our starting point is the following estimate from \cite{BCG25Characterizing}, which
states that the top line will come down to a height of $O(T^\delta)$ for any $\delta > 0$. We will then iterate this.

\begin{proposition}[Coming down estimate, {\cite[Proposition 5.5]{BCG25Characterizing}}]
\label{prop:comingdown_input}
    Let $\X$ be the $\lambda$-tilted line ensemble on $[-T,T]$
    with infinitely many lines, all of them having boundary conditions $T$
    at both ends of the interval.
    For any $\delta > 0$ there exist constants $\eta = \eta(\lambda) > 0$
    and $C,c > 0$ depending on $\lambda$ and $\delta$ such that
    \begin{equation}\label{comingdown432}
        \P \left[X^1(\pm \eta T) \leq (\eta T)^\delta \right]
        \geq 1 - C e^{-T^c}.
    \end{equation}
\end{proposition}

The original formulation of \cite[Proposition 5.5]{BCG25Characterizing} was stated for the
\emph{second} line, but with a much higher boundary condition namely $T^2 - 2KT$. For such a high boundary condition indeed a result such as \eqref{comingdown432} can only hold from the second line onwards. This is because the decay of the $i^{th}$ line essentially occurs along the parabola $x \rightarrow \lambda^{i-1}x^2$. Consequently, for the top line, the parabolic decay is not sufficiently rapid for the above result to hold. 
In contrast, for all the other lines, since $\lambda^{i-1}$ is strictly greater than $1$, the stronger parabolic decay forces the line to come down. Nonetheless, since the boundary condition in the above lemma is only linear in $T$ and not quadratic, the same argument as in \cite[Proposition 5.5]{BCG25Characterizing} goes through verbatim for all the lines, including the top one.

We also point out that \cite[Proposition 5.5]{BCG25Characterizing} was stated only for the $n$-line
ensemble rather than the infinite-line ensemble, but the estimate was uniform in
$n$.
Thus taking a local weak limit to arrive at the infinite-line ensemble,
monotonicity ensures that the estimate carries over.

\begin{proof}[Proof of Lemma \ref{lem:comingdown}]
By monotonicity it suffices to prove the result for $\lambda = \lambda_0$,
so let us set $\eta = \eta(\lambda_0) \in (0,1)$ as defined in
Proposition \ref{prop:comingdown_input}.
We will repeatedly apply that proposition with $\delta = 1$.
Let us define the events
\begin{equation}
    \cA_k(\X) = \left\{
        X^1\left(\pm \eta^k T\right) \leq \eta^k T
    \right\}.
\end{equation}
Then $\cA_0(\X)$ holds by assumption, and if $\cA_k(\X)$ holds then we may resample
$\X$ inside of the interval $[-\eta^k T, \eta^k T]$ and apply
Proposition \ref{prop:comingdown_input} to observe that
\begin{equation}
    \P \left[ \cA_{k+1}(\X) \middle| \cA_k(\X) \right]
    \geq 1 - C e^{- (\eta^k T)^c }.
\end{equation}
Now let us define $K$ to be the largest integer such that $\eta^K T > \frac{M}{2}$.
Then $\eta^{K+1}T \leq \frac{M}{2}$, so if $\cA_{K+1}(\X)$ holds we have
$X^1(\pm \eta^{K+1}T) \leq \frac{M}{2}$.
If on the other hand $\cA_{K+1}(\X)$ fails, then there must have been some first
value of $k \leq K$ for which $\cA_k(\X)$ failed, and so we have
\begin{align}
    \P \left[ X^1\left(\pm \eta^{K+1} T\right) > \tfrac{M}{2} \right]
    &\leq \sum_{k=0}^K \left( 1 - \P \left[ \cA_{k+1}(\X) \middle| \cA_k(\X) \right] \right) \\
    &\leq C \sum_{k=0}^K e^{-(\eta^k T)^c}.
\end{align}
Now since $\eta^K T > \frac{M}{2}$, have $\eta^k T > \eta^{k-K} \frac{M}{2}$, and so
substituting $j = K-k$ we find that
\begin{align}
    \P \left[ X^1\left(\pm \eta^{K+1} T\right) > \tfrac{M}{2} \right]
    &\leq C \sum_{j=0}^K e^{- \left( \eta^{-j} \frac{M}{2} \right)^c} \\
    &= C \sum_{j=0}^K \left( e^{- \left(\frac{M}{2}\right)^c} \right)^{\eta^{-cj}}.
\end{align}
This is a sum of terms shrinking faster than geometrically since $\eta \in (0,1)$,
and thus by adjusting the constants $C,c>0$ we have
\begin{equation}
    \P \left[ X^1\left(\pm \eta^{K+1} T\right) > \tfrac{M}{2} \right]
    \leq C e^{- \left(\frac{M}{2}\right) ^c}.
\end{equation}
Now if $X^1(\pm \eta^{K+1}T) \leq \frac{M}{2}$,
then $\X$ is stochastically dominated on $[-\eta^{K+1}T, \eta^{K+1}T]$
by $\Y + \frac{M}{2}$, where $\Y$ is a zero-boundary-condition infinite-line $\lambda$-tilted
ensemble on this interval.
Thus applying the tail bound Theorem \ref{thm:tailbound_onepoint} for $\Y$ (which is stochastically
dominated by $\X$) we find that
\begin{equation}
    \P \left[ X^1(0) > M \right]
    \leq C e^{- \left(\frac{M}{2}\right)^c} + \Exp{-\left(\tfrac{4}{3} -o(1)\right) \left(\tfrac{M}{2}\right)^{3/2}},
\end{equation}
which finishes the proof.
\end{proof}

%% file: sections/3_topline_0.tex
\section{A priori control on the top line}
\label{sec:topline}

In this section we demonstrate an a priori control
on the top line, showing that it must be below some $O(1)$ height
at most points in a lattice.

\begin{proposition}[A priori control on the top line]
\label{prop:topline}
    For any fixed $\eta > 0$ and $\lambda_0 > 1$, there are some $H > 0$ and
    $C,c > 0$ such that for all $\lambda \geq \lambda_0$, all $\Delta \geq 1$,
    and all large enough $T$,
    if $\X$ denotes the $\lambda$-tilted line ensemble, then
    \begin{equation}
        \P \left[
            \# \left\{
                j \in \Z : |j \Delta| \leq T \text{ and } X^1(j\Delta) > H
            \right\} > \eta \frac{2T}{\Delta}
        \right] \leq C e^{-c T}.
    \end{equation}
\end{proposition}
This will provide an important starting point for our branching process analysis
in the next section, where we will assume that $\lambda$ is large enough.
Note that in this regime, a concentration statement as above would have been
straightforward to obtain if we already had at our disposal a correlation decay
statement as Theorem \ref{thm:correlation}. However, we need the former to prove
the latter, and hence a new idea is needed to prove the above.
Further, note that the above result holds for $\lambda$
close to $1$ as well, and hence might be a useful tool for probing the behavior
of the model in this regime.

\subsection{Proof via Girsanov transformation}
\label{sec:topline_strategy}

In this section we prove Proposition \ref{prop:topline} modulo a list of inputs.
A key tool will be a Girsanov transformation which will be described shortly. We
will also use the single line version of Theorem \ref{thm:correlation}, i.e.,\ the
exponential decay of correlations for the Ferrari--Spohn diffusion.  As a Markov
process, the Ferrari--Spohn diffusion equilibrates to its
stationary distribution rapidly, and it will frequently be convenient to
forget about the boundary conditions when considering it on a bounded interval
far away enough from the boundary,
which is captured by the following lemma. The following lemma will be useful not only for
the proof of Proposition \ref{prop:topline}
but also separately in the proof of Theorem \ref{thm:correlation}.

\begin{lemma}[Coupling with a stationary process]
\label{lem:fs_coupling}
There are some constants $D, C, c > 0$ such that the following holds
for all $T, S > 0$ satisfying $D S^{3/2} < T$.
Let $X$ denote a Ferrari--Spohn diffusion on $[-T,T]$ with boundary conditions
$X(\pm T) \leq S$.
Then
for all $t \in [D S^{3/2},T]$,
$X$ may be coupled with the stationary Ferrari--Spohn diffusion $\FS$
such that we have
\begin{equation}
    \P \left[ X(s) = \FS(s) \text{ for all } |s| \leq T - t \right]
    \geq 1 - C e^{- c t}.
\end{equation}
\end{lemma}

This result is not optimal since $X$ should come down to $O(1)$ height after only time $S^{1/2}$.
Nevertheless, Lemma \ref{lem:fs_coupling} will suffice for our purposes.
Although this lemma captures the mixing behavior of the Ferrari--Spohn diffusion which has
been understood since \cite{FS05Constrained}
explicitly calculated the spectral gap for this process, we could not find an explicit
statement of the above form in the literature so we provide a short proof
shortly in Section \ref{sec:topline_fs}.

Secondly, to use the above single line result in the proof of Proposition
\ref{prop:topline}, we would like to view the top line $X^1$ as a Ferrari--Spohn
diffusion with the second line acting $X^2$ as a floor. 
However, to pursue this approach one needs some regularity conditions on the floor. To this end, 
shortly we will state a result controlling $X^2$ by a random smooth function
$\Phi$ such that $X^2(t) \leq \Phi(t)$. Given this, we will then control the gap
process $X^1-\Phi(t)$. To do this we will apply a  a \emph{Girsanov
transformation} which allows us to stochastically dominate the latter by a
Ferrari-Spohn diffusion with a certain area tilt strength depending on the
smoothness of $\Phi$. An argument of this flavor in this context had
first appeared in  \cite[Section 3.2]{CIW19Confinement}.

We now expand a bit more on this. Recall that Girsanov's theorem provides the
Radon-Nikodym derivative of a Brownian motion with a drift, with respect to a
standard Brownian motion. Using this one can verify the following. For an interval
$[\ell,r]$ and a twice differentiable function $f$ with $f(\ell)=0, f(r)=0$,
letting $B$ denote the standard Brownian bridge on the interval, the law of
$B+f$ admits the following Radon-Nikodym derivative 
\begin{equation}\label{rn23}
    \frac{d \left(B+f\right) }{d B}(X)
    \propto \Exp{\int_\ell^r f'(t)dX(t)} \propto \Exp{\int_\ell^r -f''(t)X(t)dt}
\end{equation}
where the first step is by Girsanov and the last step is integration by parts
(for brevity we denote by $B$ and $B+f$ their respective laws). Note that the
assumption of the endpoints being zero can be easily removed by adding an affine
function which has a vanishing second derivative. Thus the second derivative of
$f$ acts as an area tilt term. One can now condition on positivity to get that
the law of $B+f$ conditioned to be positive is the same as that of an
area-tilted Brownian excursion.

Thus the Girsanov
transformation can be thought of as a relationship between imposing a floor and
changing the area tilt strength. When $f= t^2$ is a parabola,
$f''=2$ is a constant and hence one gets the Ferrari--Spohn diffusion as
already mentioned in Remark \ref{rmk:fs}. We next record the more general fact
that we will make use of in our analysis. If $Y$ is a single Brownian line with
variable area tilt strength given by $\lambda : [\ell,r] \to \R$ and a floor
given by a $C^2$ function $\Phi : [\ell,r] \to \R$, then $Y - \Phi$ is a
Brownian line with variable area tilt strength given by $t \mapsto \lambda(t) -
\Phi''(t)$.  In symbols, recalling the notation of Definition \ref{def:extFLE},
we have
\begin{equation}
\label{eq:girsanov}
    Y \sim \atmeasure^{x,y,\lambda(\cdot)}_{1,\ell,r,\Phi}
    \qquad \Longleftrightarrow \qquad
    Y - \Phi \sim \atmeasure^{x-\Phi(\ell),y-\Phi(r),\lambda(\cdot) - \Phi''(\cdot)}_{1,\ell,r,0},
\end{equation}
which has appeared as \cite[Equation 3.14]{CIW19Confinement}.  If we have a
uniform lower bound on $\lambda(t) - \Phi''(t)$ then we may apply
monotonicity to stochastically dominate $Y - \Phi$ by a single line with a
constant area tilt strength, at which point we have exponential decay of
correlations allowing us to achieve the desired bound.

\vspace{3mm}

The following lemma provides the desired control on $X^2$.

\begin{lemma}
\label{lem:smoothcontrol}
For any fixed $\alpha > 0$ and $\lambda_0 > 1$, there are some constants $H > 0$ and $C, c > 0$ such that
the following statement holds for all $\lambda > \lambda_0$, all $\Delta \geq 1$, and all large enough $T > 0$.
Let $\X$ denote the $\lambda$-tilted line ensemble.
Then there is a random smooth function $\Phi : \R \to \R_+$ satisfying the following three properties:
\begin{enumerate}
    \item Almost surely, we have $X^2(t) \leq \Phi(t)$ for all $t \in \R$. \label{item:bound}
    \item Almost surely, we have $\Phi''(t) \leq \frac{1}{2}$ for all $t \in \R$. \label{item:smallcurv}
    \item Defining the event
    \begin{equation}
        \cH \coloneqq \left\{ \# \left\{ j \in \Z : |j \Delta| \leq T \text{ and } \Phi(j\Delta) \leq \tfrac{H}{2} \right\}
        \geq \frac{\alpha-2}{\alpha} \left( \frac{2T}{\Delta} - 1\right) \right\},
    \end{equation}
    we have $\P [\cH | X^1]  \geq \left( \frac{1-\Delta/2T}{\alpha} \right)^2$ almost surely.
    \label{item:ghp}
\end{enumerate}
\end{lemma}

While the above two lemmas will be proven in the two forthcoming sections, in
the remainder of this section we use the above to prove Proposition
\ref{prop:topline}.
The argument is also illustrated in Figure \ref{fig:bk}.

\begin{figure}
    \centering
    \includegraphics[width=0.85\textwidth]{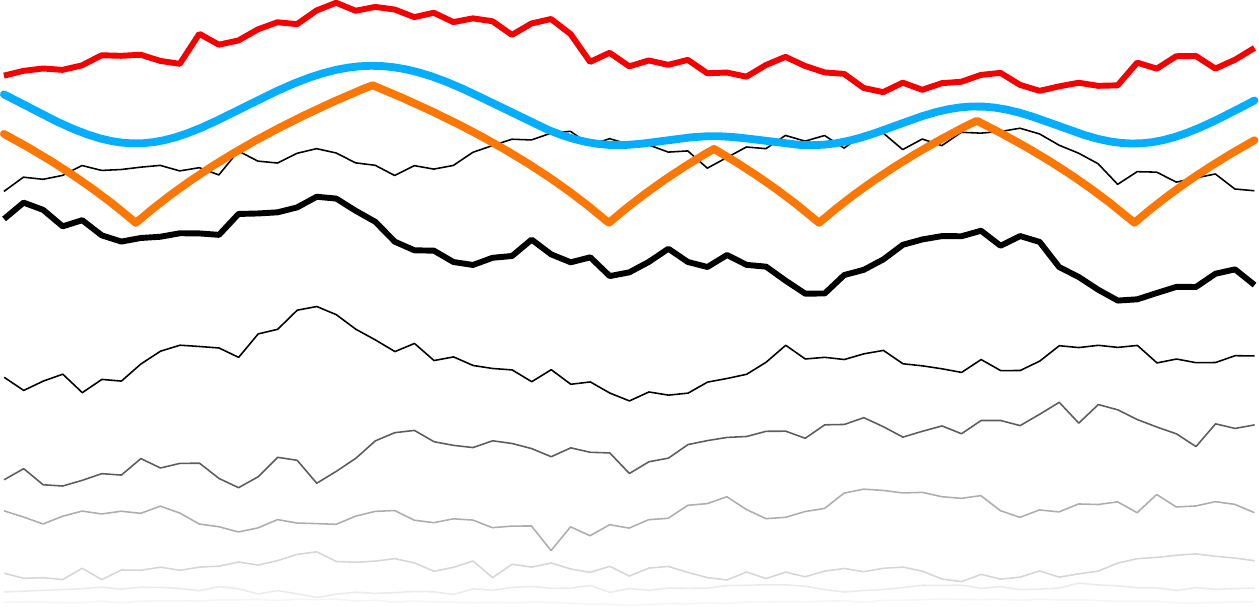}
    \caption{
        An illustration of the proof of Proposition \ref{prop:topline}.
        The ensemble $\X$ appears in black.
        We bound the second curve (thicker in the figure) by a random smooth function $\Phi$
        (blue curve) which is controlled at many points by Lemma \ref{lem:smoothcontrol}.
        Then the top line $X^1$ is stochastically dominated by a Ferrari--Spohn
        diffusion with a floor given by $\Phi$; this is the
        top red Brownian curve in the figure.
        The random smooth function $\Phi$ is constructed in Section \ref{sec:topline_bk}
        by applying a one-point confinement estimate (given by Lemma \ref{lem:confinement} below)
        at many points and taking the minimum; this gives the orange curve in the figure,
        which is then smoothed to obtain the blue curve.
    }
\label{fig:bk}
\end{figure}

\begin{proof}[Proof of Proposition \ref{prop:topline}]
Let $\cF$ be the event in the statement of this proposition, namely 
\begin{equation}
    \cF \coloneqq \left\{ \# \left\{ j \in \Z : |j \Delta| \leq T \text{ and } X^1(j\Delta) > H \right\}
    > \eta \frac{2T}{\Delta} \right\}.
\end{equation}
We aim to show that $\P[\cF] \leq C e^{-c T}$ for some constants $C, c > 0$.

Let us apply Lemma \ref{lem:smoothcontrol} with some choice of $\alpha > 0$ to be fixed later.
Note that by averaging item \ref{item:ghp} of that lemma over all $X^1$ satisfying $\cF$,
we obtain $\P [ \cH | \cF ]  \geq \left( \frac{1 - \Delta/2T}{\alpha} \right)^2$ as well.
By reassembling the conditional expectation we find
\begin{equation}
\label{eq:br}
    \P \left[ \cF \right] \leq \left( \frac{\alpha}{1 - \Delta/2T} \right)^2 \P \left[ \cH \cap \cF \right]
    \leq \left( \frac{\alpha}{1 - \Delta/2T} \right)^2 \P \left[ \cF \middle| \cH \right].
\end{equation}
Therefore to prove $\P[\cF] \leq C e^{-c T}$, it suffices
to get an exponential upper bound for $\P \left[ \cF \middle| \cH \right]$ instead.

For this, we first apply monotonicity using item \ref{item:bound} of Lemma \ref{lem:smoothcontrol}
to see that $X^1$ on the interval $[-T,T]$
is stochastically dominated by $W$, a Ferrari--Spohn diffusion with a floor given by $\Phi$
and boundary conditions given by $W(\pm T) = \max\{X^1(\pm T), \Phi(\pm T)\}$.
Next, we use the Girsanov transformation \eqref{eq:girsanov} to find that $W - \Phi$ is a single Brownian
line with variable area tilt strength given by $\lambda(t) = 1 - \Phi''(t)$.
Since $1 - \Phi''(t) \geq \frac{1}{2}$ by item \ref{item:smallcurv}, another application of
monotonicity ensures that $W - \Phi(t)$ is stochastically dominated by a single Brownian line $Z$ with area
tilt strength $\frac{1}{2}$ and boundary conditions given by $Z(\pm T) = \max\{X^1(\pm T) - \Phi(\pm T), 0\}$.

Now $Z$ is simply a rescaled version of the standard Ferrari--Spohn diffusion,
and so Lemma \ref{lem:fs_coupling} applies to $Z$ after adjusting the constants appropriately.
Note that by the upper tail bound of Theorem \ref{thm:tailbound_onepoint} applied to $X^1(\pm T)$,
for any $\delta > 0$ (to be fixed later)
we have have $Z(\pm T) \leq (\delta T)^{2/3}$ with probability $1 - e^{- c_\delta T}$ for some $c_\delta > 0$,
so we may assume that this occurs and take $S = (\delta T)^{2/3}$ in that lemma.
Thus under this assumption Lemma \ref{lem:fs_coupling} states that there are some constants $D,C,c>0$
such that we may couple $Z$ with a stationary
rescaled Ferrari--Spohn diffusion $2^{1/3} \FS$ so that
\begin{equation}
    \P \left[ Z(s) = 2^{1/3} \FS(s) \text{ for all } |s| \leq T - t \right] \geq 1 - C e^{-c t}
\end{equation}
whenever $D \delta T \leq t \leq T$.
Let us take $\delta$ small enough so that $D\delta \leq \frac{\eta}{4}$, $\eta$ being the constant
in the hypothesis of the proposition we are proving.
We thus have $Z = 2^{1/3} \FS$ on $[- T + \frac{\eta}{4} T, T - \frac{\eta}{4} T]$ with probability at least $1 - C e^{-c T}$.

Now, since the process $2^{1/3} \FS$ has a stationary probability distribution,
we may increase the constant $H$ given by Lemma \ref{lem:smoothcontrol}
(which does not change the bounds given by that lemma by monotonicity) until we have
\begin{equation}
    \P \left[ 2^{1/3} \FS(0) > \tfrac{H}{2} \right] \leq \frac{\eta}{8}.
\end{equation}
Further, the spectral gap of $2^{1/3} \FS$ is known to be positive by \cite[Section 2, Equation (2.11)]{FS05Constrained}.

Appealing thus to the standard analysis of a Markov chain with a positive spectral gap, the number of
points $j \in \Z$ with $|j \Delta| \leq T$ for which $2^{1/3} \FS(j\Delta) > \frac{H}{2}$ must
concentrate around its mean.
For instance, we may apply the Bernstein-type inequality of \cite[Theorem 3.3]{P15Concentration},
specifically equation (3.21) of that result, to the Markov chain $(2^{1/3} \FS(j\Delta))_j$
with the function $f(x) = \ind{x > \frac{H}{2}}$,
to conclude that 
\begin{align}
    &\P \left[
        \left|
        \sum_{j= - \lfloor T/\Delta \rfloor}^{\lfloor T/\Delta \rfloor}
        \ind{2^{1/3} \FS(j\Delta) > \frac{H}{2}}
        - (2 \lfloor T/\Delta \rfloor + 1)
        \P \left[ 2^{1/3} \FS(0) > \tfrac{H}{2} \right]
        \right|
        > \xi
    \right] \\
    &\qquad \qquad \qquad \qquad
    \qquad \qquad \qquad \qquad
    \qquad \qquad \qquad \qquad
    \leq 2 \Exp{- \frac{\xi^2 \gamma}{4 (2 \lfloor T/\Delta \rfloor + 1) + 10 \xi}},
    \label{eq:bernstein}
\end{align}
where $\gamma$ is the spectral gap of the chain $(2^{1/3} \FS(j\Delta))_j$.
We may take $\xi$ to be any positive multiple of $T$ to obtain a bound
which is exponentially small in $T$.
In particular, we find that
\begin{equation}
    \P \left[
        \# \left\{ j \in \Z : |j\Delta| \leq T \text{ and } 2^{1/3} \FS(s) \geq \tfrac{H}{2} \right\}
        \geq \frac{\eta}{4} \frac{2T}{\Delta}
    \right] \leq C e^{-c T}.
\end{equation}
By the couplings described above,
this means that with probability at least $1 - C e^{-cT}$, there are at most $\frac{\eta}{2} \frac{2T}{\Delta}$
points $j \in \Z$ with $|j\Delta| \leq T$ for which $X^1(j\Delta) \geq \Phi(j\Delta) + \tfrac{H}{2}$.

Finally, we set $\alpha$ large enough at the beginning of the argument
so that $\frac{\alpha-2}{\alpha} \geq 1 - \frac{\eta}{4}$, which means that
under $\cH$ we have
\begin{equation}
    \# \left\{ j \in \Z : |j\Delta| \leq T \text{ and } \Phi(j\Delta) \leq \tfrac{H}{2} \right\} \geq
    \left(1 - \frac{\eta}{4} \right) \left( \frac{2T}{\Delta} - 1 \right).
\end{equation}
Thus under $\cH$, with probability at least $1 - C e^{- c T}$, there are at most $\eta \frac{2T}{\Delta}$
points $j \in \Z$ with $|j\Delta| \leq T$ for which $X^1(j\Delta) \geq H$.
In other words, $\P[\cF|\cH] \leq C e^{-cT}$, which finishes the proof by \eqref{eq:br}.
\end{proof}

It remains to prove Lemma \ref{lem:fs_coupling} and Lemma
\ref{lem:smoothcontrol} which is done in the forthcoming Sections
\ref{sec:topline_fs} and \ref{sec:topline_bk} respectively. 

\input{sections/3_topline_1_fs}
\input{sections/3_topline_2_bk}

%% file: sections/3_topline_1_fs.tex
\subsection{Exponential mixing of the Ferrari--Spohn diffusion}
\label{sec:topline_fs}

In this section, we will prove Lemma \ref{lem:fs_coupling}, showing that we may couple
a Ferrari--Spohn diffusion $X$ with boundary conditions to a stationary one $\FS$
such that they agree on a large interval with high probability.

We will make use of the following lemma, proved in \cite[Lemma 2.7]{BCG25Characterizing}
for the Ferrari--Spohn  diffusion with boundary conditions,
which mitigates the effect of high boundary conditions away from the boundary.
As already mentioned in \cite{BCG25Characterizing}, the following lemma does not give the
optimal result, and thus our Lemma \ref{lem:fs_coupling} is not optimal,
since $X$ should come down to $O(1)$ height after only time $S^{1/2}$.

\begin{lemma}[{\cite[Lemma 2.7]{BCG25Characterizing}}]
\label{lem:fs_comingdown}
There are some constants $M, C, c > 0$ such that for all $S > 0$
and $T \geq S^{3/2}$, if $X$ denotes a Ferrari--Spohn diffusion on $[-T,T]$
with boundary conditions $X(\pm T) = S$, then
\begin{equation}
    \P \left[ X(t) \geq M \text{ for all } t \in [-T,T] \right]
    \leq C e^{-c T}.
\end{equation}
\end{lemma}

\begin{proof}[Proof of Lemma \ref{lem:fs_coupling}]
First, $X$ is stochastically dominated by $\FS + S$ by monotonicity.
So Lemma \ref{lem:fs_uppertail_onepoint} implies that for any constant $D$,
$X(\pm T \mp \frac{t}{2}) \leq S + \frac{1}{D} t^{2/3}$ with probability at least
$1 - C e^{-c t}$ for some constants $C,c > 0$.
Thus we may apply Lemma \ref{lem:fs_comingdown} to the intervals
$[-T,-T+\frac{t}{2}]$ and $[T-\frac{t}{2},T]$ as
$\frac{t}{4} \geq (S + \frac{1}{D} t^{2/3})^{3/2}$
if $t$ is large enough, also using our assumption that $t \geq D S^{3/2}$.
This shows that with probability at least $1 - C e^{-ct}$
there is some stopping domain $[\tau_\ell,\tau_r]$ with
$\tau_\ell \leq -T+\frac{t}{2}$ and $\tau_r \geq T-\frac{t}{2}$ such that
$X(\tau_\ell), X(\tau_r) \leq M$.
Now another application of monotonicity
implies that $X$ restricted to $[\tau_\ell,\tau_r]$ is stochastically
dominated by a Ferrari--Spohn diffusion shifted up by $M$,
and we may as well restrict the interval to $[-T+\frac{t}{2},T-\frac{t}{2}]$.

So consider two independent stationary Ferrari--Spohn diffusions
$\FS$ and $\FS'$ on $[-T+\frac{t}{2},T-\frac{t}{2}]$.
Under the high-probability event described above,
we may couple $X$ to $\FS'$ so that $X \leq \FS'+M$ on this interval.
Thus if there are times $s_\ell \in [-T+\frac{t}{2}, -T+t]$ and $s_r \in [T-t, T-\frac{t}{2}]$
for which $\FS(s_\ell) \geq \FS'(s_\ell) + M$ and
$\FS(s_r) \geq \FS'(s_r) + M$, then there must also be a stopping
domain $[\tau_\ell', \tau_r']$ with $\tau_\ell \leq s_\ell$
and $\tau_r \geq s_r$ for which $\FS(\tau_\ell) = X(\tau_\ell)$
and $\FS(\tau_r) = X(\tau_r)$.
We may thus resample both processes in this stopping domain
and couple them perfectly.

It just remains to prove the existence of the aforementioned
times $s_\ell$ and $s_r$.
But since $\FS$ and $\FS'+M$ are independent stationary Markov
processes and the product of their stationary distributions
assigns positive measure to the event $\{ \FS(0) > \FS'(0)+M\}$,
such times exist with probability $1 - C e^{-ct}$ by standard
Markov process analysis.
For instance, we may again apply the Bernstein-type inequality of
\cite[Theorem 3.3]{P15Concentration},
as in \eqref{eq:bernstein} in the proof of Proposition \ref{prop:topline},
but this time with the Markov chain $W_n = (\FS(n), \FS'(n))$
and the function $f((x,y)) = \ind{x > y + M}$.
In this context, that result states that
\begin{equation}
\label{eq:markovchain}
    \P \left[
        \left|
            \sum_{n=\lceil T - t \rceil}^{\lfloor T - t/2 \rfloor} \ind{\FS(n) > \FS'(n) + M}
            - \lfloor t/2 \rfloor \P[\FS(0) > \FS'(0) + M]
        \right| > \xi
    \right]
    \leq 2 \Exp{- \frac{\xi^2 \gamma}{4 \lfloor t / 2 \rfloor + 10 \xi}},
\end{equation}
where $\gamma$ is the spectral gap of the chain $(W_n)$.
As this is a product of two chains which have positive spectral gap, again by
\cite[Section 2, Equation (2.11)]{FS05Constrained}, we may take $\xi$ to be some
small multiple of $t$ and let $t$ be large enough to
get an exponentially small (in $t$) upper bound on the probability that the time $s_r$ 
does not exist.
Similarly, we see that $s_\ell$ exists with all but exponentially small probability,
which concludes the proof.
\end{proof}

%% file: sections/3_topline_2_bk.tex
\subsection{Controlling the second line by a smooth function}
\label{sec:topline_bk}

In this section we will prove Lemma \ref{lem:smoothcontrol}, showing that the second line
of $\X$ is bounded by a smooth function $\Phi$ which satisfies a constant bound at many points in
a $\Delta$-spaced lattice inside of $[-T,T]$.

For this, we will make use of the following confinement estimate which gives control at individual points.
This estimate is a slight optimization of \cite[Theorem 3.1]{CIW19Confinement}
which was stated explicitly as follows in \cite[Theorem 2.7]{CG25Uniqueness},
though we will use it as a lemma in our arguments and state it as such. The original statement of \cite[Theorem 2.7]{CG25Uniqueness} was for
a finite-line ensemble with $n$ lines,
and the supremum was taken over a finite interval $[-T,T]$.
However, the bound is uniform in both $n$ and $T$ so these parameters may both be
taken to infinity by monotonicity and the monotone convergence theorem.

\begin{lemma}[Logarithmic confinement estimate]
\label{lem:confinement} 
{For any $\lambda > 1$} there is some $M > 0$ such that for the $\lambda$-tilted line ensemble $\X$, we have
\begin{equation}
\label{expectedcontrol1234}
    \E \left[ \sup_{t \in \R} \left[X^1(t) - \psi(t) \right]_+ \right] \leq M,
\end{equation}
where $\psi(t) = M \log(1 + |t|)$ and $[a]_+ = \max\{a,0\}$.
\end{lemma}

By Lemma \ref{lem:secondline_domination}, the above statement applies to the
second line $X^2$, even conditionally on $X^1.$ Further by a second moment
argument, we will show that with high (but constant) probability, the above
control applies centered at multiple points which allows us to create an
efficient envelope by taking the minimum of many shifted copies of the function $\psi$
in the above lemma.
A smoothening process is then employed to obtain $\Phi$ since the
function $\psi$ has a sharp downwards-pointing corner at the origin, and
additionally taking the minimum will result in upwards-pointing corners.

Now it just remains to prove Lemma \ref{lem:smoothcontrol}.
We do this in the following two subsections: first, in Section \ref{sec:topline_bk_bound} we exhibit a function $\Psi$
which bounds $X^2$ and satisfies a version of item \ref{item:ghp} in Lemma \ref{lem:smoothcontrol}, but may not be smooth.
Then in Section \ref{sec:topline_bk_smooth}, we smoothen $\Psi$ to obtain $\Phi$ while retaining the desired properties.

\subsubsection{Bounding the second line}
\label{sec:topline_bk_bound}

Start by noting that in Lemma \ref{lem:confinement} we may take $M$ to be uniform over all $\lambda$ greater than
any fixed $\lambda_0 > 1$ by monotonicity.
We would like to apply this bound at many points and take the minimum over shifts of $\psi$ as our initial bound.
To that end, let us fix $\alpha$ as in the statement of Lemma \ref{lem:smoothcontrol} and define for any $j \in \Z$
the event
\begin{equation}
\label{eq:Ajdef}
    \cA_j \coloneqq \left\{ \sup_{t \in \R} [X^1(t) - \psi(t - j\Delta)]_+ \leq \alpha M \right\}.
\end{equation}
We then define the nonsmooth bound $\Psi$ as follows:
\begin{equation}
\label{eq:bigpsidef}
    \Psi(t) \coloneqq \min \left\{ \psi(t - j\Delta) : j \in \Z \text{ and } \cA_j \text{ holds} \right\} + \alpha M,
\end{equation}
and note that by definition we have $X^2(t) \leq \Psi(t)$ for all $t \in \R$.
Now we control $\Psi(t)$ by using the second moment method to show
that $\cA_j$ holds for many values of $j$ with $|j\Delta| \leq T$ with some reasonable probability,
even after conditioning on the first line.
Note that whenever $\cA_j$ holds, we have $\Psi(j\Delta) \leq \alpha M$ since $\psi(0) = 0$.

\begin{lemma}
\label{lem:2mmnew}
For $\alpha > 0$ fixed and $\cA_j$ defined as above, let us define the event
\begin{equation}
    \cG \coloneqq \left\{ \# \left\{ j \in \Z : |j \Delta| \leq T \text{ and } \cA_j \text{ holds} \right\}
    \geq \frac{\alpha-2}{\alpha} \left( \frac{2T}{\Delta} - 1 \right) \right\}.
\end{equation}
Then we have $\P[\cG | X^1 ] \geq \left( \frac{1 - \Delta/2T}{\alpha} \right)^2$ almost surely.
\end{lemma}

\begin{proof}[Proof of Lemma \ref{lem:2mmnew}]
By Lemma \ref{lem:secondline_domination}, it suffices to show a similar statement for the top line
of $\X$ instead, since $\cG$ is a decreasing event.
More precisely, let us set
\begin{equation}
    \cA_j' = \left\{ \max_{t \in \R} \left[ X^1(t) - \phi(t - j\Delta) \right]_+ \leq \alpha M \right\}
\end{equation}
and 
\begin{equation}
    \cG' \coloneqq \left\{ \# \left\{ j \in \Z : |j \Delta| \leq T \text{ and } \cA_j' \text{ holds} \right\}
    \geq \frac{\alpha-2}{\alpha} \left( \frac{2T}{\Delta} - 1 \right) \right\}.
\end{equation}
Then since $\cG'$ is decreasing and $\overline{\X} = \{ X^2, X^3, \dotsc, \}$ conditioned on $X^1$
is stochastically dominated by $\X$ as mentioned in Lemma \ref{lem:secondline_domination},
we have $\P\left[ \cG \middle| X^1 \right] \geq \P\left[ \cG' \right]$ almost surely.

Now by Markov's inequality applied to \eqref{expectedcontrol1234} and stationarity of $\X$, we have
\begin{equation}
    \E \left[ \sum_{j : |j \Delta| \leq T} \ind*{\cA_j'} \right] \geq 
    \frac{\alpha-1}{\alpha} \left( \frac{2T}{\Delta} - 1 \right),
\end{equation}
since there are at least $\frac{2T}{\Delta} - 1$ terms in the sum.
And, since there are at most $\frac{2T}{\Delta} + 1$ terms, we also have
\begin{equation}
    \E \left[ \left( \sum_{j : |j \Delta| \leq T} \ind*{\cA_j'} \right)^2 \right]
    \leq \left( \frac{2T}{\Delta} + 1 \right)^2.
\end{equation}
So by the Paley--Zygmund inequality we have
\begin{equation}
    \P \left[
        \sum_{j : |j \Delta| \leq T} \ind*{\cA_j'} \geq
        \theta \frac{\alpha-1}{\alpha} \left( \frac{2T}{\Delta} - 1 \right)
    \right] \geq (1 - \theta)^2 \left( \frac{\alpha-1}{\alpha} \right)^2
    \left( \frac{\frac{2T}{\Delta}-1}{\frac{2T}{\Delta}+1} \right)^2.
\end{equation}
Plugging in $\theta = \frac{\alpha-2}{\alpha-1}$ for instance, and bounding the fraction,
we have
\begin{equation}
    \P \left[
        \sum_{j : |j \Delta| \leq T} \ind*{\cA_j'}
        \geq \frac{\alpha-2}{\alpha} \left( \frac{2T}{\Delta} - 1 \right)
    \right] \geq \left( \frac{1 - \Delta/2T}{\alpha} \right)^2.
\end{equation}
This finishes the proof.
\end{proof}

\subsubsection{Smoothing the bound}
\label{sec:topline_bk_smooth}

We now apply multiple smoothing operations to $\Psi$ in order to obtain the function $\Phi$ of
Lemma \ref{lem:smoothcontrol}.
Of primary importance is the downwards-pointing corner at $0$ of $\psi(t) = M \log(1 + |t|)$, which we fix by replacing the function $\psi$
in the minimum \eqref{eq:bigpsidef} defining $\Psi$ by a smoothed version $\phi$, defined in Lemma \ref{lem:smallphi} below.

Note that after doing this replacement and taking the minimum over various shifted copies of $\phi$,
we will arrive at a function $\tilde{\Phi}$ which is still not smooth, having upwards-pointing corners
at locations where the minimizing function changes (see the upwards-pointing corners in the orange curve of Figure \ref{fig:bk} for instance).
As such, we will perform a further smoothing operation to obtain the desired function $\Phi$.
So the function $\phi$ defined in the following lemma is only an intermediate stage;
as such, we do not aim for the best possible result.
In particular, the somewhat artificial looking item \ref{item:deriv2} below could be easily improved, but this will not
be necessary for our purposes.

\begin{lemma}
\label{lem:smallphi}
Let $\psi(t) = M \log(1 + |t|)$ for some $M > 0$ as in Lemma \ref{lem:confinement}.
Then for any $\alpha > 0$, there is some $H > 0$ for which there exists
a symmetric $C^1$ function $\phi$ on $\R$ such that
\begin{enumerate}[label=(A.\arabic*)]
    \item For all $t \in \R$, we have $\phi(t) \geq \psi(t)$, \label{item:phigtrpsi}
    \item For all $t \in [-\frac{1}{2},\frac{1}{2}]$, we have $\phi(t) \leq \frac{H}{2} - \alpha M - 1$, \label{item:lessh2}
    \item For all $t \in \R$, we have $|\phi'(t)| \leq 1$, \label{item:deriv1}
    \item For all but two values of $t \in \R$, the function $\phi$ is twice differentiable
    and satisfies $\phi''(t) \leq \frac{1}{2}$. \label{item:deriv2}
\end{enumerate}
\end{lemma}

\begin{proof}[Proof of Lemma \ref{lem:smallphi}]
To see that such a function $\phi$ exists, one possible construction is to glue in a wide parabola,
i.e.\ set
\begin{equation}
    \phi(t) = \begin{cases}
        a t^2 + b & \text{if } |t| \leq c, \\
        \psi(t) & \text{otherwise},
    \end{cases}
\end{equation}
for some values of $a,b,c$.
As long as we take $a \in (0,\frac{1}{4})$, then item \ref{item:deriv2} holds because
$\psi''(t)$ is negative away from $0$; the two possible points where $\phi$ is not twice differentiable
will be $\pm c$.
To ensure that the function is continuous we must take $b = \psi(c) - a c^2$,
and to ensure that the function is $C^1$ on all of $\R$, we must take $c > 0$ to satisfy
\begin{equation}
    2 a c = \frac{M}{1+c}, \qquad \text{i.e.} \qquad
    c = \tfrac{1}{2} \left( \sqrt{1 + \frac{2 M}{a}} - 1 \right).
\end{equation}
Now since $\phi'(t)$ increases from $0$ to $2 a c = \frac{M}{1+c}$ as $t$ increases from $0$ to $c$ and then
decreases back to $0$ as $t$ increases beyond $c$, by taking $c$ large enough in the construction
(which corresponds to taking $a$ small enough) we may ensure that item \ref{item:deriv1} holds.
The constant $H$ of item \ref{item:lessh2} is then determined by the choices made in the construction.
Finally, since $\psi$ is concave on $(-c,0)$ and $(0,c)$ while $\phi$ is convex on $(-c,c)$,
and their derivatives agree at $\pm c$, item \ref{item:phigtrpsi} holds.
\end{proof}

Having fixed $\alpha$ and letting $\phi$ be as in Lemma \ref{lem:smallphi},
let us now define a partially smoothed function $\tilde{\Phi}$ via
\begin{equation}
\label{eq:tildephidef}
    \tilde{\Phi}(t) \coloneqq \min \left\{ \phi(t - j\Delta) : j \in \Z \text{ and } \cA_j \text{ holds} \right\} + \alpha M.
\end{equation}
By item \ref{item:phigtrpsi} of Lemma \ref{lem:smallphi}, we have $\tilde{\Phi}(t) \geq \Psi(t) \geq X^2(t)$ for all $t$,
recalling the definition \eqref{eq:bigpsidef} of $\Psi$.
Additionally, whenever $\cA_j$ holds, we have $\tilde{\Phi}(t) \leq \frac{H}{2} - 1$ for all $t \in [j\Delta-1, j\Delta+1]$
by item \ref{item:lessh2}.
However, the function $\tilde{\Phi}$ is still not smooth since the function $\phi$ is only $C^1$ and moreover
the points where the supremum in \eqref{eq:tildephidef} switches from one curve to another will result in corner points
where $\tilde{\Phi}$ is not even differentiable.
So we will apply one more smoothing operation to the whole function $\tilde{\Phi}$ resulting in a suitable smooth
function $\Phi$.
This is done by the following lemma.

\begin{lemma}
\label{lem:bigphi}
Fix $\alpha, M$ as above and let $H$ be the constant introduced in Lemma \ref{lem:smallphi}
to define $\phi$.
Further let $\tilde{\Phi}$ be the random function defined in \eqref{eq:tildephidef}.
Then there is a random smooth function $\Phi$ which satisfies
\begin{enumerate}[label=(B.\arabic*)]
    \item For all $t \in \R$, we have $\Phi(t) \geq \tilde{\Phi}(t)$, \label{item:tildgtr}
    \item For all $t \in \R$, we have $\Phi''(t) \leq \frac{1}{2}$, \label{item:tildd2}
    \item For all $j \in \Z$ such that $\cA_j$ holds, we have $\Phi(j\Delta) \leq \frac{H}{2}$. \label{item:tildh2}
\end{enumerate}
\end{lemma}

\begin{proof}[Proof of Lemma \ref{lem:bigphi}]
To see that such a function exists, we may take $\Phi = \tilde{\Phi} * \chi + 1$, where $*$ denotes
convolution and $\chi$ 
is a smooth nonnegative bump function supported in $\left[-\frac{1}{2},\frac{1}{2}\right]$ with
total integral $1$.
Then item \ref{item:tildgtr} follows by item \ref{item:deriv1} which implies that
$\tilde{\Phi}$ is $1$-Lipschitz, so
\begin{equation}
    (\chi * \tilde{\Phi})(t) \geq \min_{s \in \left[t-\frac{1}{2},t+\frac{1}{2}\right]} \tilde{\Phi}(s)
    \geq \max_{s \in \left[t-\frac{1}{2},t+\frac{1}{2}\right]} \tilde{\Phi}(s) - 1
    \geq \tilde{\Phi}(t) - 1.
\end{equation}
Additionally, item \ref{item:tildh2} follows from \ref{item:lessh2} because $\Phi(t)-1$ is an average
of the values of $\tilde{\Phi}$ in the interval $[t-\frac{1}{2},t+\frac{1}{2}]$.
Finally, item \ref{item:tildd2} follows from item \ref{item:deriv2} because $\tilde{\Phi}$ is twice differentiable
except for on a discrete set $E$ of points in $\R$, so we have
\begin{align}
    \Phi''(t) &= \int_{-\infty}^\infty \chi''(t-s) \tilde{\Phi}(s) \,ds
    = \int_{\R \setminus E} \chi''(t-s) \tilde{\Phi}(s) \,ds \\
    &= \int_{\R \setminus E} \chi'(t-s) \tilde{\Phi}'(s) \,ds
        + \sum_{r \in E} \left( \chi'(t-r) \tilde{\Phi}(r) - \chi'(t-r) \tilde{\Phi}(r) \right) \\
    &= \int_{\R \setminus E} \chi(t-s) \tilde{\Phi}''(s) \,ds
        + \sum_{r \in E} \left( \chi(t-r) \tilde{\Phi}'(r+) - \chi(t-r) \tilde{\Phi}(r-) \right),
\end{align}
where $f(r+) = \lim_{u \searrow r} f(u)$ and $f(r-) = \lim_{u \nearrow r} f(u)$.
Now the integral above is an average of a quantity which is $\leq \frac{1}{2}$ by item \ref{item:deriv2},
and the sum is negative.
To see why, note that $\tilde{\Phi}$ is only nondifferentiable at the points $r$ where the minimum
of \eqref{eq:tildephidef} switches.
By symmetry, we thus have $\tilde{\Phi}'(r+) = - \tilde{\Phi}'(r-) \leq 0$.
\end{proof}

With this final smoothing in hand we may put the pieces together and prove Lemma \ref{lem:smoothcontrol}.

\begin{proof}[Proof of Lemma \ref{lem:smoothcontrol}]
We take $\Phi$ to be as constructed in Lemma \ref{lem:bigphi} above.
Item \ref{item:bound} which states that $\Phi(t) \geq X^2(t)$ for all $t \in \R$ follows immediately
from the construction, by item \ref{item:tildgtr} of Lemma \ref{lem:bigphi} as well as the fact mentioned
below \eqref{eq:tildephidef} that $\tilde{\Phi}(t) \geq \Psi(t) \geq X^2(t)$.
Item \ref{item:smallcurv} which states that $\Phi''(t) \leq \frac{1}{2}$ for all $t \in \R$ also immediately
follows from item \ref{item:tildd2} of Lemma \ref{lem:bigphi}.
Finally, recall that Lemma \ref{lem:2mmnew} states that
\begin{equation}
    \P \left[ \# \left\{ j \in \Z : |j\Delta| \leq T \text{ and } \cA_j \text{ holds} \right\}
    \geq \frac{\alpha-2}{\alpha} \left( \frac{2T}{\Delta} - 1 \right)
    \,\middle|\, X^1
    \right] \geq \left( \frac{1-\Delta/2T}{\alpha} \right)^2
\end{equation}
almost surely.
Since item \ref{item:tildh2} of Lemma \ref{lem:bigphi} states that $\Phi(j\Delta) \leq \frac{H}{2}$
if $\cA_j$ holds, this immediately implies item \ref{item:ghp} of the lemma, finishing the proof.
\end{proof}

%% file: sections/4_correlation_0.tex
\section{Exponential decay of correlations}
\label{sec:correlation}

Let $\X$ denote the stationary infinite-line $\lambda$-tilted ensemble of Definition \ref{def:main}.
In this section we prove Theorem \ref{thm:correlation}, showing that for large enough $\lambda$
the lines of $\X$ exhibit exponential decay of correlations.
Although this is stated in that theorem also for the finite-line ensemble $\X_n$ with $n$ lines,
we will give the full proof only for the infinite-line case.
The finite-line case is almost exactly the same, and we will highlight where any modifications
are necessary.

\input{sections/4_correlation_1_overview}
\input{sections/4_correlation_2_branching}
\input{sections/4_correlation_3_topline}

%% file: sections/4_correlation_1_overview.tex
\subsection{Overview of the proof of correlation decay}
\label{sec:correlation_overview}
As discussed in Section  \ref{sec:intro_iop_reverse} the proof will proceed by comparing  $\X$ with an ensemble $\Y$ under which $\Y(0)$ and $\Y(t)$ are independent.
More precisely, let $\Y$ denote the $\lambda$-tilted line ensemble on $[-\frac{t}{2}, \frac{3t}{2}]$
with infinitely many lines, all of which are pinned to zero at the three points $-\frac{t}{2}$, $\frac{t}{2}$,
and $\frac{3t}{2}$.
Note that in the proof for the finite-line case of Theorem \ref{thm:correlation}, we should instead
consider an ensemble $\Y_n$ with the same number $n$ of lines as $\X_n$; the pinning does not change.

\subsubsection{Reversing stochastic domination}
\label{sec:correlation_overview_domination}

By monotonicity, $\X$ stochastically dominates $\Y$, but we will demonstrate a coupling under which
the top $k \geq \max\{i,j\}$ lines of $\X$ are (almost) below those of $\Y$ at the points $0$ and $t$.
Intuitively, the existence of these two couplings with opposite ordering behavior
shows that the joint distribution of $(X^i(0), X^j(t))$ must be close to that of $(Y^i(0),Y^j(t))$.
This idea was made precise in \cite[Section 6]{CG25Uniqueness} for the top line (i.e.\ the case $i=j=1$),
and we now extend this analysis to our slightly more general setting.
Aside from being interesting in its own right, this extension will be crucial for our later application
to bounding the spectral gap in Theorem \ref{thm:spectralgap}.
We begin by reducing the question of covariance to bounding various expected differences under an
arbitrary coupling.

\begin{lemma}
\label{lem:delta}
Under any coupling between $\X$ and $\Y$, we have
\begin{align}\label{decom4631}
    \left|
        \Cov[X^i(0), X^j(t)]
    \right|
    &\leq
    2 \sqrt{
        \E \left[ X^i(0) (X^i(0) - Y^i(0)) \right] \cdot
        \E \left[ X^j(t) (X^j(t) - Y^j(t)) \right]
    } \\
    \label{eq:deltastmt2}
    &\qquad + \E \left[ Y^i(0) (X^j(t) - Y^j(t)) \right]
        + \E \left[ (X^i(0) - Y^i(0)) Y^j(t) \right] \\
    \label{eq:deltastmt3}
    &\qquad + \E\left[Y^i(0)\right] \E\left[ X^j(t) - Y^j(t) \right]
    + \E \left[ X^j(t) \right] \E \left[ X^i(0) - Y^i(0) \right].
\end{align}
Moreover, both factors appearing under the square root in \eqref{decom4631} are nonnegative.
\end{lemma}

\begin{proof}[Proof of Lemma \ref{lem:delta}]
First notice that we have
\begin{align}
    \Cov \left[ X^i(0), X^j(t) \right]
    &= \E \left[ X^i(0) X^j(t) \right] - \E \left[ X^i(0) \right] \E \left[ X^j(t) \right] \\
    &= \E \left[ (X^i(0) - Y^i(0)) (X^j(t) - Y^j(t)) \right] \label{eq:delta1} \\
    &\qquad + \E \left[ Y^i(0) (X^j(t) - Y^j(t)) \right]
        + \E \left[ (X^i(0) - Y^i(0)) Y^j(t) \right] \label{eq:delta2} \\
    &\qquad + \E \left[ Y^i(0) \right] \E \left[ Y^j(t) \right]
        - \E \left[ X^i(0) \right] \E \left[ X^j(t) \right], \label{eq:delta3}
\end{align}
using independence of $\Y(0)$ and $\Y(t)$.
Note that \eqref{eq:delta2} is already of the desired form \eqref{eq:deltastmt2}.
Next, we have
\begin{align}
    \left| \E \left[ Y^i(0) \right] \E \left[ Y^j(t) \right] - \E \left[ X^i(0) \right] \E \left[ X^j(t) \right] \right|
    &\leq \left| \E \left[ Y^i(0) \right] \E \left[ Y^j(t) \right] - \E \left[ Y^i(0) \right] \E \left[ X^j(t) \right] \right| \\
    &\qquad \qquad + \left| \E \left[ Y^i(0) \right] \E \left[ X^j(t) \right] - \E \left[ X^i(0) \right] \E \left[ X^j(t) \right] \right| \\
    &= \E\left[Y^i(0)\right] \E\left[ X^j(t) - Y^j(t) \right] + \E \left[ X^j(t) \right] \E \left[ X^i(0) - Y^i(0) \right],
\end{align}
which bounds \eqref{eq:delta3} by \eqref{eq:deltastmt3}.
Finally let us apply the Cauchy--Schwarz inequality to line \eqref{eq:delta1},
obtaining
\begin{equation}
    \left| \E \left[ (X^i(0) - Y^i(0)) (X^j(t) - Y^j(t)) \right] \right|
    \leq \sqrt{\E \left[ (X^i(0) - Y^i(0))^2 \right] \cdot \E\left[ (X^j(t) - Y^j(t))^2 \right]}.
\end{equation}
For the first factor above, using stochastic dominance again which shows that $\E[Y^i(0)^2] \leq \E[X^i(0)^2]$,
\begin{align}
    \E \left[ (X^i(0) - Y^i(0))^2 \right]
    &\leq 2 \E\left[ X^i(0)^2 \right] - 2 \E \left[ X^i(0) Y^i(0) \right] \\
    &= 2 \E \left[ X^i(0) (X^i(0) - Y^i(0)) \right],
\end{align}
which also shows that the expression on the right-hand side is nonnegative.
Similarly we have
\begin{equation}
    \E \left[ (X^j(t) - Y^j(t))^2 \right] \leq 2 \E \left[ X^j(t) (X^j(t) - Y^j(t)) \right].
\end{equation}
Combining these, we bound \eqref{eq:delta1} by \eqref{decom4631}, obtaining the result.
\end{proof}

We finally arrive at the following all-important proposition 
which allows us to bound all the expectations appearing in the right-hand side
of the bound in Lemma \ref{lem:delta}.

\begin{proposition}[Reversal coupling]
\label{prop:reversal}
There is some $\lambda_0 > 1$ and $C, c > 0$ such that for all $\lambda \geq \lambda_0$,
all $\eps > 0$, all $k \in \N$, and all $t > 0$, there is a coupling (depending on $\eps$, $k$, and $t$)
between $\X$ and $\Y$ such that
\begin{equation}
    \P \left[
        X^i(0) - Y^i(0) \leq \eps \text{ and }
        X^i(t) - Y^i(t) \leq \eps
        \text{ for all } i \leq k
    \right] \geq 1 - C e^{-c t}.
\end{equation}
\end{proposition}

We remark that in principle, one should be able to take $\eps = 0$ and $k = \infty$ in the above statement,
possibly by taking a suitable subsequential weak limit of the couplings given by Proposition \ref{prop:reversal}.
However, since the $\lambda$-BG property may only be used to resample finitely many lines at a time,
the above statement is more straightforward using our methods and will suffice for our purposes.
The remainder of the section will be devoted to the proof of the above proposition.
First, we quickly show how its application to the right-hand side of
Lemma \ref{lem:delta} proves Theorem \ref{thm:correlation}.

\begin{proof}[Proof of Theorem \ref{thm:correlation}]
Let us apply the coupling given by Proposition \ref{prop:reversal} with $\eps > 0$ and $k > \max\{i,j\}$ to be chosen later.
This yields the following bound for the first factor under the square root
appearing in the right-hand side of the bound in Lemma \ref{lem:delta}.
In this derivation, we use the positivity of $X^1(0)$ for the first term and
apply the Cauchy--Schwarz inequality to the second term below:
\begin{align}
    \E \left[ X^i(0) (X^i(0) - Y^i(0)) \right]
    &= \E \left[ X^i(0) (X^i(0) - Y^i(0)) \left( \ind{X^i(0) - Y^i(0) \leq \eps }  + \ind{X^i(0) - Y^i(0) > \eps }\right) \right] \\
    &\leq \eps \E \left[ X^i(0) \right] + \sqrt{C e^{-ct} \cdot \E \left[ \left( X^i(0) (X^i(0) - Y^i(0)) \right)^2 \right]} \\
    &\leq \eps \E \left[ X^i(0) \right] + C e^{-c t} \sqrt{\E\left[ X^i(0)^4 \right] + \E \left[ X^i(0)^2 Y^i(0)^2 \right]}
\end{align}
after adjusting the constants.
Now by monotonicity in the form of Lemma \ref{lem:secondline_domination} and rescaling, the upper tail bound 
of Theorem \ref{thm:tailbound_onepoint} implies that $\E \left[ X^i(0)^4 \right] \leq C \lambda^{-4(i-1)/3}$, and
the same bound holds for $\E \left[ X^i(0)^2 Y^i(0)^2 \right]$ as the same tail bound applies to $Y^i(0)$ by
another application of monotonicity.
So by choosing $\eps$ small enough a priori, we find
\begin{equation}
    \E \left[ X^i(0) (X^i(0) - Y^i(0)) \right] \leq C e^{-c t} \lambda^{-2(i-1)/3},
\end{equation}
and by the same argument,
\begin{equation}
    \E \left[ X^j(t) (X^j(t) - Y^j(t)) \right] \leq C e^{-c t} \lambda^{-2(j-1)/3}.
\end{equation}
Additionally, a similar argument proves that
\begin{align}
    \E \left[ Y^i(0) (X^j(t) - Y^j(t)) \right] &\leq C e^{-c t} \lambda^{-(i+j-2)/3},
    & & &
    \E \left[ (X^i(0) - Y^i(0)) Y^j(t) \right] &\leq C e^{-c t} \lambda^{-(i+j-2)/3}, \\
    \E \left[ Y^i(0) \right] \E \left[ X^j(t) - Y^j(t) \right] &\leq C e^{-c t} \lambda^{-(i+j-2)/3},
    & & &
    \E \left[ X^j(t) \right] \E \left[ X^i(0) - Y^i(0) \right] &\leq C e^{-c t} \lambda^{-(i+j-2)/3},
\end{align}
again by choosing $\eps$ small enough a priori so that all of these hold.
So, adjusting the constants again, by Lemma \ref{lem:delta} we find that
\begin{equation}
    \left| \Cov \left[ X^i(0), X^j(t) \right] \right| \leq C e^{- c t} \lambda^{-(i+j-2)/3}.
\end{equation}
This finishes the proof of the upper bound.

To see that the covariance is nonnegative, we appeal to Proposition \ref{prop:FKG} below,
which is a version of the FKG inequality stating that continuous increasing functions of a
finite line ensemble have nonnegative covariance, if their covariance exists.
In order to recover the nonnegativity result for the infinite ensemble $\X$,
we may express it as a monotone limit of $\mathbf{Z}_{n,T}$,
a zero-boundary ensemble on $[-T,T]$ with $n$ lines.
The nonnegativity is preserved because in fact
\begin{equation}
    \Cov\left[ X^i(0), X^j(t) \right] =
    \lim_{n,T \to \infty} \Cov\left[Z_{n,T}^i(0), Z_{n,T}^j(t) \right] \geq 0.
\end{equation}
The above convergence holds by the monotone convergence theorem
applied to the individual expectations comprising the covariance.
For a similar argument with more detail provided, the reader may consult
the proof of Proposition \ref{prop:lipschitzdecay} below.
\end{proof}

\subsubsection{Resampling at times of reversal}
\label{sec:correlation_overview_reversaltime}

Proposition \ref{prop:reversal} will be proved by resampling $\X$ and $\Y$ on two stopping domains
containing $0$ and $t$ respectively, where on the boundary we have a reversal of the top $k$ lines.
Explicitly, define
\begin{align}
    \tau_\ell^{0,k} &= \inf \left\{
        s \in \left[-\tfrac{t}{2}, 0\right]
        : X^i(s) \leq Y^i(s) \text{ for all } i \leq k \right\} {\wedge 0}, \\
    \tau_r^{0,k} &= \sup \left\{
        s \in \left[0, \tfrac{t}{2}\right]
        : X^i(s) \leq Y^i(s) \text{ for all } i \leq k \right\} {\vee 0}, \\
    \tau_\ell^{t,k} &= \inf \left\{
        s \in \left[\tfrac{t}{2}, t\right]
        : X^i(s) \leq Y^i(s) \text{ for all } i \leq k \right\} {\wedge t}, \\
    \tau_r^{t,k} &= \sup \left\{
        s \in \left[t, \tfrac{3t}{2}\right]
        : X^i(s) \leq Y^i(s) \text{ for all } i \leq k \right\} {\vee t},
\end{align}
where $\wedge$ and $\vee$ denote the minimum and maximum respectively.
Note that these are stopping domains with respect to the collection of $\sigma$-algebras
$\{ \extextalg^n_{\ell,r} = \extalg^n_{\ell,r}(\X) \otimes \extalg^n_{\ell,r}(\Y) \}$,
which is a valid choice for the strong Brownian Gibbs property as discussed in Section
\ref{sec:inputs_strongbg}, using the fact that $\X$ and $\Y$ are independent before resampling.
We remark that this independence is crucial, and we will be applying the resampling exactly once
to construct a non-independent coupling between $\X$ and $\Y$, after which further applications
of the Brownian Gibbs property would not necessarily be valid.

\begin{lemma}[Nontriviality of stopping domains]
\label{lem:correlation_stopping} 
There is some $\lambda_0 > 1$ and $C, c > 0$ such that for all $\lambda \geq \lambda_0$
and all $k \in \N$, when $\X$ and $\Y$ are sampled independently we have
\begin{equation}
    \tau_\ell^{0,k} < 0 < \tau_r^{0,k}
    \qquad \text{and} \qquad
    \tau_\ell^{t,k} < t < \tau_r^{t,k}
\end{equation}
with probability at least $1 - C e^{-ct}$.
\end{lemma}

We will leave the proof of Lemma \ref{lem:correlation_stopping} to the remainder of this section,
but let us first use it to prove Proposition \ref{prop:reversal}.
Intuitively, this is achieved by first uniformly controlling the very low lines of $\X$ (with index $>K$ for some large $K \geq k$),
and then resampling the lines of $\X$ and $\Y$ with index $\leq K$ on the stopping domains
$[\tau_\ell^{0,K}, \tau_r^{0,K}]$ and $[\tau_\ell^{t,K}, \tau_r^{t,K}]$ on which we may ensure
that they remain almost reversed by monotonicity up to an exponentially small shift caused by the effect of the curves with index $>K$.

\begin{proof}[Proof of Proposition \ref{prop:reversal}]
First note that for any $K \in \N$, after rescaling by $\rho = \lambda^{K/3}$ (recalling \eqref{eq:scaling} for the notation of rescaling)
and using monotonicity in the form of Lemma \ref{lem:secondline_domination},
which says that $(X^{K+1})^{(\rho)}$ is stochastically dominated by $X^1$, we may apply Corollary \ref{cor:tailbound_interval}, yielding
\begin{align}
    \P \left[ \max_{s \in [-\frac{t}{2}, \frac{3t}{2}]} X^{K+1}(s) > \eps \right]
    &\leq \P \left[ \max_{s \in [- \frac{t}{2} \lambda^{2K/3}, \frac{3t}{2} \lambda^{2K/3}]} X^1(s) > \eps \lambda^{K/3} \right] \\
    &\leq C 2 t \lambda^{2K/3} \Exp{- c \left( \eps \lambda^{K/3} - C \log(2t \lambda^{2K/3}) \right)^{3/2}} \\
    &\leq \Exp{ - c \left( \eps \lambda^{K/3} - C K  - C \log t \right)^{3/2} + C K + C \log t}
\end{align}
after adjusting the constants.
Since $\lambda^{K/3}$ grows exponentially in $K$, we may ensure that the above is as small as we like by choosing $K$ to be a very large integer
depending on both $\eps$ and $t$.
In particular, we may fix some $K \geq k$ (where $k \in \N$ is given in the statement of the proposition) for which the above is at most $e^{-t}$.
So by Lemma \ref{lem:correlation_stopping}, after sampling
$\X$ and $\Y$ independently, we have
\begin{equation}
\label{eq:rev1}
    \tau_\ell^{0,K} < 0 < \tau_r^{0,K}, \qquad
    \tau_\ell^{t,K} < t < \tau_r^{t,K}, \qquad \text{and} \qquad
    \max_{s \in [-\frac{t}{2},\frac{3t}{2}]} X^{K+1}(s) < \eps
\end{equation}
with probability at least $1 - C e^{-ct}$.
We may then resample the top $K$ lines of $\X$ and $\Y$ on both
$[\tau_\ell^{0,K},\tau_r^{0,K}]$ and $[\tau_\ell^{t,K},\tau_r^{t,K}]$, which are stopping domains with respect to the
$\sigma$-algebras $\{ \extextalg^n_{\ell,r} \}$ mentioned above.
On the event \eqref{eq:rev1}, the ensemble $\X^{\leq K}$ (consisting of the top $K$ lines of $\X$)
is stochastically dominated by the ensemble $\Y^{\leq K} + \eps$ on these stopping domains.
So monotonicity ensures that, given \eqref{eq:rev1}, we may couple the resampling 
so that $X^i(s) \leq Y^i(s) + \eps$ for all $i \leq K$ and
$s \in [\tau_\ell^{0,K}, \tau_r^{0,K}] \cup [\tau_\ell^{t,K}, \tau_r^{t,K}]$.
\end{proof}

%% file: sections/4_correlation_2_branching.tex
\subsection{A branching process for reversing lines}
\label{sec:correlation_branching}

In this section we prove the following lemma which allows us to find reversal times in $O(1)$ windows,
which will be instrumental in our proof of Lemma \ref{lem:correlation_stopping}.

\begin{proposition}[Existence of reversal times]
\label{prop:correlation_reversaltime}
For all large enough $H > 0$ there are some $\lambda_0 > 1$ and $\delta_0 > 0$
such that for all $\lambda \geq \lambda_0$, the following holds.
Let $\X$ and $\Y$ be independent infinite-line $\lambda$-tilted line ensembles on $[-2H,2H]$
with arbitrary boundary data subject to the restriction $X^1(\pm2H) \leq H$.
Then
\begin{equation}
    \P \left[ \exists s \in [-2H,2H] \text{ such that } X^k(s) \leq Y^k(s) \text{ for all } k \geq 1 \right]
    \geq \delta_0.
\end{equation}
\end{proposition}

This section is dedicated to the proof of Proposition \ref{prop:correlation_reversaltime}, and later in Section \ref{sec:correlation_topline}
we will use this proposition to prove Lemma \ref{lem:correlation_stopping} (which implies Theorem \ref{thm:correlation}
as discussed in the previous section).
Therein, we will break up the interval $\left[ - \frac{t}{2}, \frac{3t}{2} \right]$ into many
intervals of length $4H$ and apply Proposition \ref{prop:correlation_reversaltime} within each of them.
Since Proposition \ref{prop:correlation_reversaltime} allows for {arbitrary boundary data} {satisfying the bound $X^1(\pm 2H) \leq H$},
the existence of reversal times within each separate interval may be taken to be roughly independent, leading to the exponential bound on
the probability in Lemma \ref{lem:correlation_stopping}.

The broad idea of the proof of Proposition \ref{prop:correlation_reversaltime}, alluded to in Section \ref{sec:intro_iop_branching},
is that we will decompose the interval $[-2H,2H]$ into a nested
family of intervals capturing the natural scales of lines in $\X$ and $\Y$ via the scaling relation described in Section \ref{sec:inputs_scaling}.
Then every time $X^k$ and $Y^k$ remain reversed on an interval of their natural scale $\lambda^{-2(k-1)/3}$, the lower lines $X^{k+1}$ and $Y^{k+1}$ will have many
attempts to remain reversed on an interval of \emph{their} natural scale $\lambda^{-2k/3}$, since this scale is a factor of $\lambda^{-2/3}$ smaller.
Thus the set of intervals where reversal occurs is related to a branching process which we will ensure is supercritical
by taking $\lambda_0$ large enough.
If the branching process survives for infinitely many generations, then the intersections of the corresponding
intervals will result in a point $s$ where all lines are reversed.
Note that in the finite-line case of Theorem \ref{thm:correlation}, we simply require the branching
process to survive to generation $n$, where $n$ is the number of lines.

\subsubsection{Tree-indexed families of intervals}
\label{sec:correlation_branching_intervals}

In this section we will create two $b$-ary tree-indexed collections of intervals on which our proof
will be built. 
As just indicated above, in principle we should be able to take $b \approx \lambda^{2/3}$
to capture the relationship between scales at different indices.
However, we did not attempt to optimize our arguments for this purpose and will be satisfied
with any $b$ which grows with $\lambda$; in particular, we will take $b = \left\lfloor \frac{\lambda^{1/3}}{4M} \right\rfloor$,
where $M$ is a constant that will be determined later.

The vertices of the $b$-ary tree will be identified with elements of
$\treeb \coloneqq \bigcup_{k = 0}^\infty [b]^k$, where $[b] = \{1,\dotsc,b\}$,
and $[b]^k$ consists of length-$k$ strings of the symbols in $[b]$.
For $k=0$ we have $[b]^0 = \{ \emptyset \}$, where $\emptyset$ denotes the empty string.
For each $\tau \in [b]^k$ and $\sigma \in [b]$, we denote by $\tau \sigma \in [b]^{k+1}$ the string
consisting of $\tau$ with the symbol $\sigma$ appended to the end.
In addition, for $\tau \in [b]^k$ we set $|\tau| = k$, i.e.\ $|\tau|$ denotes the length of the string $\tau$,
and we say that $\tau \preceq \tau'$ if the string $\tau'$ begins with the substring $\tau$.

We will create two $\treeb$-indexed families of intervals $\{ I_\tau \}$ and $\{ J_\tau \}$
such that $J_\tau \subseteq I_\tau \subseteq [-2H,2H]$ for each $\tau \in \treeb$.
We will also ensure that $\{ I_\tau : |\tau| = k \}$ have pairwise disjoint interiors for each fixed $k$,
and that the length of $J_\tau$ is $|J_\tau| = \lambda^{-2|\tau|/3}$ for each
$\tau \in \treeb$.
First, as a base case we simply set $I_\emptyset = [-2H,2H]$, and we will define $J_\emptyset$
to be the middle interval of length $1$ inside of $I_\emptyset$, i.e.\ $J_\emptyset = [-\frac{1}{2},\frac{1}{2}]$.
Then inductively, if we have defined $J_\tau$ for some $\tau \in \treeb$, for each $\sigma \in [b]$
we will define $I_{\tau\sigma}$ to be the $\sigma$th interior-disjoint interval of length
$4 M \lambda^{1/3} \lambda^{-2|\tau\sigma|/3}$ inside of $J_\tau$.
Since
\begin{equation}
    \frac{|I_{\tau\sigma}|}{|J_\tau|}
    = \frac{4 M \lambda^{1/3} \lambda^{-2|\tau\sigma|/3}}{\lambda^{-2|\tau|/3}}
    = \frac{4 M}{\lambda^{1/3}} \leq \frac{1}{b},
\end{equation}
there are indeed at least $b$ such interior-disjoint intervals, and we simply ignore the unused
space if there is any.
After defining $I_{\tau\sigma}$ we simply let $J_{\tau\sigma}$ be the middle interval
of length $\lambda^{-2|\tau\sigma|/3}$ inside of $I_{\tau\sigma}$.
See Figure \ref{fig:correlation_intervals} for an illustration of this construction.

\begin{figure}
    \centering
    \includegraphics[width=0.75\textwidth]{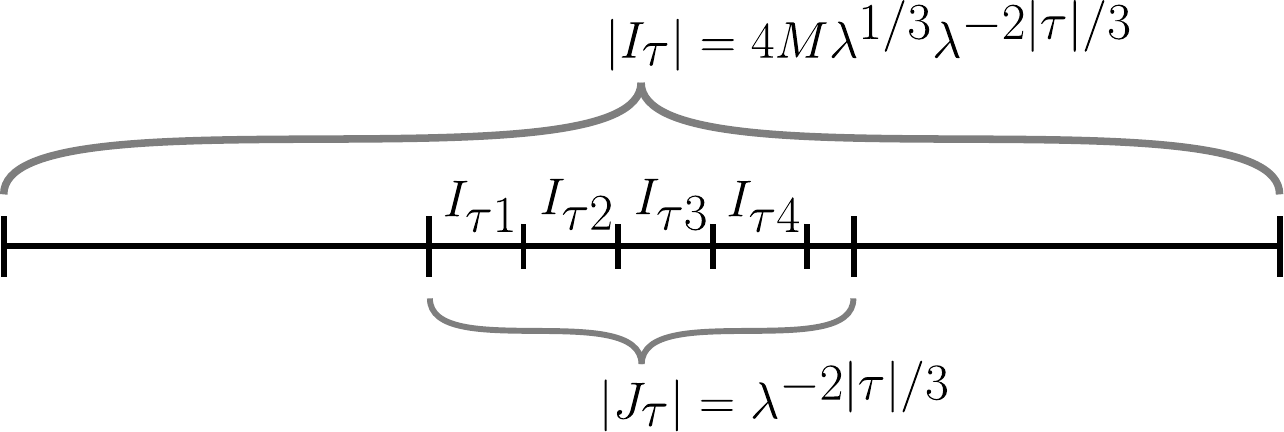}
    \caption{
        One iteration of the nested intervals in the construction,
        with $b = 4$ for the purpose of illustration; note that
        there is leftover room at the right end of $J_\tau$ which is
        not taken up by any interval $I_{\tau\sigma}$.
    }
\label{fig:correlation_intervals}
\end{figure}

\subsubsection{A branching process of reversal events}
\label{sec:correlation_branching_events}

Now we introduce the on-scale reversal events indexed by $\tau \in \treeb$
which will form the branching process of our argument.
The events depend on the same global choice of the constant $M > 0$ which was introduced
in the definition of the intervals in Section \ref{sec:correlation_branching_intervals} above,
and which will be specified later.
Let us first define
\begin{equation}
    \Low_\tau(\X) \coloneqq \left\{
        X^{|\tau|+1}(s) \leq M \lambda^{-|\tau|/3}
        \text{ for all } s \in J_\tau
    \right\}
\end{equation}
and
\begin{equation}
    \High_\tau(\Y) \coloneqq \left\{
        Y^{|\tau|+1}(s) \geq M \lambda^{-|\tau|/3}
        \text{ for all } s \in J_\tau
    \right\}.
\end{equation}
Note that both of these events are ``on-scale'' the sense that the typical height of the curve with index $k+1$
is of order $\lambda^{-k/3}$ as indicated by the scaling relation of Lemma \ref{lem:scaling}.
We then set
\begin{equation}
    \Reverse_\tau(\X,\Y) \coloneqq
    \Low_\tau(\X) \cap \High_\tau(\Y),
\end{equation}
and note that $\Reverse_\tau(\X,\Y)$ implies that $X^{|\tau|+1}(s) \leq Y^{|\tau|+1}(s)$
for all $s \in J_\tau$.
See Figure \ref{fig:reversal} for an illustration of this event.

\begin{figure}
    \centering
    \includegraphics[width=0.6\textwidth]{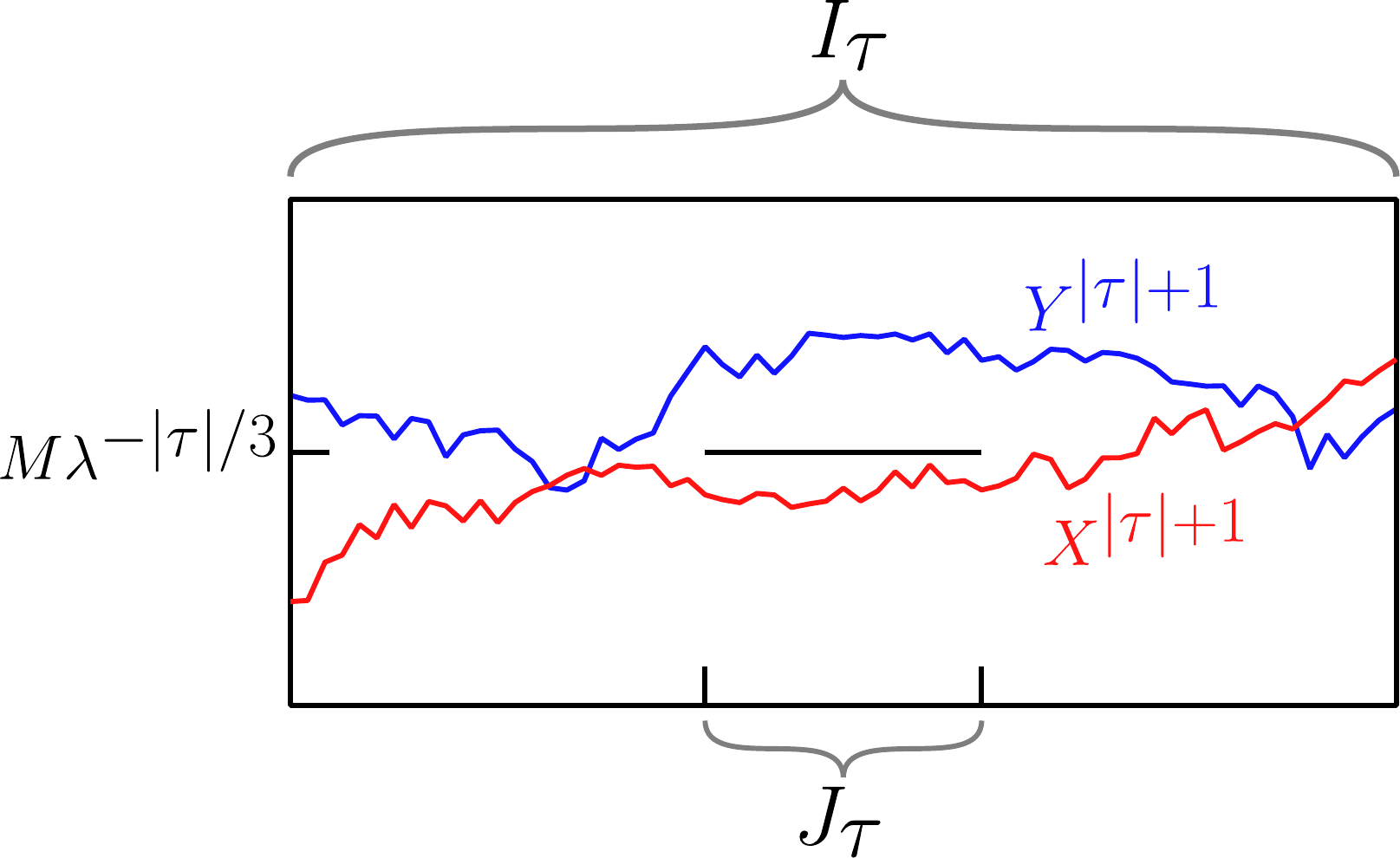}
    \caption{
        The event $\Reverse_\tau(\X,\Y)$,
        where $\X$ is drawn in red and $\Y$ is drawn in blue.
    }
\label{fig:reversal}    
\end{figure}

The main lemma in our argument is that the set of $\tau \in \treeb$ for which $\Reverse_\tau(\X,\Y)$ holds
stochastically dominates a branching process, which can be made supercritical by taking $\lambda$ large enough.
To state this precisely, {for $p, q > 0$} let us define the branching process $\BP^{p,q}_b \subseteq \treeb$
to be {a standard $\mathrm{Binomial}(b,q)$ branching process except that the root may also be absent with probability $1-p$; in other words, $\BP^{p,q}_b$ is}
the random set of vertices defined by the following properties:
\begin{enumerate}
    \item First, $\emptyset \in \BP^{p,q}_b$ with probability $p$.
    \item For any $\tau \in \treeb$, if $\tau \in \BP^{p,q}_b$ then for each $\sigma \in [b]$ independently,
    we have $\tau\sigma \in \BP^{p,q}_b$ with probability $q$.
    \item If $\tau \notin \BP^{p,q}_b$, then no descendant of $\tau$ (i.e.\ string which begins with $\tau$)
    is in $\BP^{p,q}_b$ either.
\end{enumerate}

Recall that in our setting, $b = \left\lfloor \frac{\lambda^{1/3}}{4M} \right\rfloor$ depends on $\lambda$ and
$M$, where $M$ is the parameter in the above constructions whose value has not yet been set.

\begin{lemma}
\label{lem:correlation_branchingprocess}
For all large enough $M, H > 0$, there is some $\lambda_0 > 1$ and $p, q > 0$ such that
for all $\lambda \geq \lambda_0$ the following holds.
Let $\X$ and $\Y$ be independent infinite-line $\lambda$-tilted line ensembles on $[-2H,2H]$
with arbitrary boundary data subject to the restriction $X^1(\pm 2H) \leq H$.
Then there is a coupling between $(\X,\Y)$ and $\BP^{p,q}_b$ such that,
for all $\tau \in \treeb$, $\tau \in \BP^{p,q}_b$ implies $\Reverse_\tau(\X,\Y)$.
\end{lemma}

This lemma will be proved inductively via Lemmas \ref{lem:correlation_inductivestep} and \ref{lem:correlation_basecase}
below, which are the inductive step and the base case.
Namely, Lemma \ref{lem:correlation_inductivestep} will show that if $\Reverse_\tau(\X,\Y)$ holds, then
$( \ind*{\Reverse_{\tau\sigma}(\X,\Y)} : \sigma \in [b] )$ stochastically dominates a tuple of $b$ independent
$\mathrm{Bernoulli}(q)$ variables, for some fixed $q$ as long as $\lambda$ is large enough.
And Lemma \ref{lem:correlation_basecase} will show that $\P[\Reverse_\emptyset(\X,\Y)] \geq p$ for some fixed $p$.

Before proceeding with these lemmas,
let us now see how Lemma \ref{lem:correlation_branchingprocess} implies Proposition \ref{prop:correlation_reversaltime}.

\begin{proof}[Proof of Proposition \ref{prop:correlation_reversaltime}]
Let $H > 0$ be large enough and fix some choice of $M > 0$ large enough that Lemma \ref{lem:correlation_branchingprocess}
holds with the constants $\lambda_0 > 1$ and $p, q > 0$.
Note that if $\BP^{p,q}_b$ contains an infinite ray $\emptyset = \tau_0 \preceq \tau_1 \preceq \tau_2 \preceq \dotsb$
where $\tau_k \in [b]^k$ and $\tau \preceq \tau'$ means that $\tau'$ begins with $\tau$, then
$\Reverse_{\tau_k}(\X,\Y)$ holds for all $k \geq 0$.
Thus $X^k(s) \leq Y^k(s)$ for all $s \in J_{\tau_k}$.
Since $J_{\tau_0} \supseteq J_{\tau_1} \supseteq J_{\tau_2} \supseteq \dotsb$ are closed intervals,
there is some point $s \in \bigcap_{k \geq 0} J_{\tau_k}$, where all curves are reversed.

So it remains to show that (possibly after increasing the value of $\lambda_0$), there is some $\delta_0 > 0$
such that for all $\lambda \geq \lambda_0$, the branching process $\BP^{p,q}_b$ contains an infinite
ray with probability at least $\delta_0$.
Note that the average number of offspring of each element of the branching process is
\begin{equation}
    q b \geq \frac{q \lambda^{1/3}}{4M}.
\end{equation}
Since $M$ is fixed and $q$ does not depend on $\lambda$ as long as $\lambda \geq \lambda_0$,
this is greater than $1$ as soon as $\lambda$ is large enough.
So we may increase $\lambda_0$ until the above is greater than $1$ whenever $\lambda \geq \lambda_0$.
Then if $\emptyset \in \BP^{p,q}_b$, there is some chance $\eta > 0$ for $\BP^{p,q}_b$ to contain an infinite ray.
Finally, $\emptyset \in \BP^{p,q}_b$ with probability $p$ which does not depend on $\lambda$,
so we may take $\delta_0 = p \eta$.
\end{proof}

\subsubsection{Inductive step: on-scale lemmas}
\label{sec:correlation_branching_inductive}

In this section we prove the inductive step for Lemma \ref{lem:correlation_branchingprocess}.
For the statement, let us recall from \eqref{eq:defextalg} the exterior $\sigma$-algebra defined
for any interval $I \subseteq \R$
and $a < b \in \N \cup \{\infty\}$ via
\begin{equation}
    \extalg^{a,b}_{I}(\X) = \sigma \left(
        X^k(t) : \text{ either } t \notin I
        \text{ or } k \notin [a,b]
    \right).
\end{equation}
The following lemma states that if we have a reversal of one curve, we will have a chance
to have a reversal of the next curve on each on-scale interval, even after conditioning on the
values of the ensembles at the endpoints of the intervals, which makes these events independent.

\begin{lemma}[Inductive step for branching process]
\label{lem:correlation_inductivestep}
For all large enough $M$ in the above definitions, there exist constants $\lambda_0 > 1$
and $q > 0$ such that the following holds for all $\lambda \geq \lambda_0$.
For any independent $\lambda$-tilted line ensembles $\X$ and $\Y$, we have the following for all
$\tau \in \treeb$ and $\sigma \in [b]$ {almost surely}:
\begin{equation}
    \P \left[
        \Reverse_{\tau\sigma}(\X,\Y)
        \;\middle|\;
        \Reverse_\tau(\X,\Y), \;
        \extalg^{|\tau|+2,\infty}_{I_{\tau\sigma}}(\X), \;
        \extalg^{|\tau|+2,\infty}_{I_{\tau\sigma}}(\Y)
    \right] \geq q.
\end{equation}
\end{lemma}

This lemma immediately implies the following corollary which will form the backbone of our proof of
the coupling in Lemma \ref{lem:correlation_branchingprocess}:

\begin{corollary}
\label{cor:inductivestep}
For all large enough $M$ in the above definitions, there exist constants $\lambda_0 > 1$
and $q > 0$ such that the following holds for all $\lambda \geq \lambda_0$.
For any independent $\lambda$-tilted line ensembles $\X$ and $\Y$, any $\tau \in \treeb$,
if $\Reverse_\tau(\X,\Y)$ holds then {conditionally}, $( \ind*{\Reverse_{\tau\sigma}(\X,\Y)} : \sigma \in [b] )$
stochastically dominates a tuple of $b$ independent $\mathrm{Bernoulli}(q)$ random variables.
\end{corollary}

Now, since $\X$ and $\Y$ are independent, to prove Lemma \ref{lem:correlation_inductivestep} it suffices to prove
the following two lemmas showing that the low and high events each have a lower bound on their probabilities.

\begin{lemma}
\label{lem:correlation_low}
For all large enough $M$ in the above definitions, there exist constants $\lambda_0 > 1$
and $q_1 > 0$ such that the following holds for all $\lambda \geq \lambda_0$.
For any $\lambda$-tilted line ensemble $\X$, any $\tau \in \treeb$, and any $\sigma \in [b]$,
\begin{equation}
\label{eq:correlation_low}
    \P \left[
        \Low_{\tau\sigma}(\X)
        \; \middle| \;
        \Low_\tau(\X), \;
        \extalg^{|\tau|+2,\infty}_{I_{\tau\sigma}}(\X)
    \right] \geq q_1.
\end{equation}
\end{lemma}

\begin{lemma}
\label{lem:correlation_high}
For all large enough $M$ in the above definitions, there exist constants $\lambda_0 > 1$
and $q_2 > 0$ such that the following holds for all $\lambda \geq \lambda_0$.
For any $\lambda$-tilted line ensemble $\Y$, any $\tau \in \treeb$, and any $\sigma \in [b]$,
\begin{equation}
\label{eq:correlation_high}
    \P \left[
        \High_{\tau\sigma}(\Y)
        \; \middle| \;
        \High_\tau(\Y), \;
        \extalg^{|\tau|+2,\infty}_{I_{\tau\sigma}}(\Y)
    \right] \geq q_2,
\end{equation}
\end{lemma}

With Lemmas \ref{lem:correlation_low} and \ref{lem:correlation_high} in hand,
Lemma \ref{lem:correlation_inductivestep} follows immediately:

\begin{proof}[Proof of Lemma \ref{lem:correlation_inductivestep}]
We may simply take $\lambda_0$ to be the maximum of the two values given in
Lemmas \ref{lem:correlation_low} and \ref{lem:correlation_high}, and take
the success probability $q$ to be the product $q_1 q_2$.
\end{proof}

We now turn to the proofs of Lemmas \ref{lem:correlation_low} and \ref{lem:correlation_high}.
Both of these will involve rescaling and then applying various tail bounds or coming-down estimates from Section \ref{sec:inputs}.

\begin{proof}[Proof of Lemma \ref{lem:correlation_low}]
This statement concerns $X^{|\tau\sigma|+1}$ which has area tilt strength $\lambda^{|\tau\sigma|}$,
on intervals $I_{\tau\sigma}$ and $J_{\tau\sigma}$ which have lengths $4 M \lambda^{1/3} \lambda^{-2|\tau\sigma|/3}$
and $\lambda^{-2|\tau\sigma|/3}$ respectively.
Since $\Low_\tau(\X)$ holds and we are conditioning on $X^{|\tau|+1}$, this curve acts as a ceiling for $X^{|\tau\sigma|+1}$
of height at most $M \lambda^{-|\tau|/3}$ on $I_{\tau\sigma}$.
Since $\Low_{\tau\sigma}(\X)$ is a decreasing event, by monotonicity the probability we want to lower bound can only
decrease if we flatten this ceiling to be at constant height $M \lambda^{-|\tau|/3}$, and increase all the boundary conditions
of $\X$ on the interval $I_{\tau\sigma}$ to this height as well.

Now we will apply a rescaling by $\rho = \lambda^{|\tau\sigma|/3}$ (recalling the notation for rescaling in \eqref{eq:scaling}),
resulting in a curve $\left(X^{|\tau\sigma|+1}\right)^{(\rho)}$ with area tilt strength $1$.
The intervals $I_{\tau\sigma}$ and $J_{\tau\sigma}$ rescale to intervals of length $4 M \lambda^{1/3}$ and $1$ respectively,
and the ceiling and boundary conditions all rise to height $M \lambda^{1/3}$.

So it suffices to prove that for all $M$ large enough, if $\mathbf{Z}$ is a $\lambda$-tilted line ensemble on $[-2M\lambda^{1/3},2M\lambda^{1/3}]$ with all
boundary conditions $M \lambda^{1/3}$ and a ceiling of height $M \lambda^{1/3}$, then
\begin{equation}
\label{eq:low_goal}
    \P \left[ Z^1(s) \leq M \text{ for all } |s| \leq \tfrac{1}{2} \right] \geq q_1,
\end{equation}
where $q_1$ may be taken uniform in all $\lambda \geq \lambda_0$, for some fixed $\lambda_0$ (possibly depending on $M$).

For this, we first note that $Z^1(0) \leq M \lambda^{1/3}$ by the imposition of the ceiling, and so we may apply Lemma \ref{lem:comingdown}
to both of the intervals $[-2M \lambda^{1/3}, 0]$ and $[0, 2M \lambda^{1/3}]$, which implies that
\begin{equation}
    \P \left[ Z^1(\pm M \lambda^{1/3}) \geq \tfrac{M}{2} \right] \leq 2 C e^{- (M/2)^c},
\end{equation}
where we may take $C, c > 0$ uniformly in $\lambda \geq 2$, for instance.
So let us assume that $M$ is large enough so that the above probability is at most $\frac{1}{3}$.
On the event that $Z^1(\pm M \lambda^{1/3}) \geq \frac{M}{2}$, we may apply Corollary \ref{cor:tailbound_interval},
which shows that
\begin{equation}
    \P\left[ Z^1(s) \leq M \text{ for all } |s| \leq \tfrac{1}{2} \right] \geq 1 - C e^{-c (M/2)^{3/2}},
\end{equation}
again with constants not depending on $\lambda$ as long as $\lambda \geq 2$.
Thus as long as $M$ is large enough that the above is at least $\frac{2}{3}$,
we may take $q_1 = \frac{1}{3}$ in \eqref{eq:low_goal}, finishing the proof.
\end{proof}

\begin{proof}[Proof of Lemma \ref{lem:correlation_high}]
As in the proof of Lemma \ref{lem:correlation_low} above, we will first rescale by $\rho = \lambda^{|\tau\sigma|/3}$.
Before rescaling, since $\High_\tau(\Y)$ holds and we are conditioning on $Y^{|\tau|+1}$, this curve acts as a ceiling
for $Y^{|\tau\sigma|+1}$ of height at least $M \lambda^{-|\tau|/3}$ on $I_{\tau\sigma}$.
Since $\High_{\tau\sigma}(\Y)$ is an increasing event, by monotonicity the probability we want to lower bound
can only decrease if we flatten this ceiling to be at constant height $M \lambda^{-|\tau|/3}$, and decrease all
the boundary conditions of $\Y$ on $I_{\tau\sigma}$ to be zero.
Additionally, if we remove all curves below $Y^{|\tau\sigma|+1}$, then the probability will also decrease.

So after rescaling it suffices to prove that for all $M$ large enough, if $Z$ is a Ferrari--Spohn diffusion
on $[-2M\lambda^{1/3}, 2M\lambda^{1/3}]$ with zero boundary conditions and a ceiling of height $M \lambda^{1/3}$,
then
\begin{equation}
\label{eq:high_goal}
    \P \left[ Z(s) \geq M \text{ for all } |s| \leq \tfrac{1}{2} \right] \geq q_2,
\end{equation}
where $q_2$ may be taken uniform in all $\lambda \geq \lambda_0$ for some fixed $\lambda_0$ (possibly depending on $M$).

For this, first let $W$ denote a Ferrari--Spohn diffusion on $[-2M\lambda^{1/3}, 2M\lambda^{1/3}]$ with zero boundary conditions
and no ceiling.
By Corollary \ref{cor:fs_uppertail_interval}, we have
\begin{equation}
    \P \left[ \max_{|s| \leq 2 M \lambda^{1/3}} W(s) \leq M \lambda^{1/3} \right] \geq 1 - C M \lambda^{1/3} e^{- c M^{3/2} \lambda^{1/2}},
\end{equation}
so by rejection sampling we may couple $Z$ and $W$ to be equal with at least the probability on the right-hand side above.
Now by Lemma \ref{lem:fs_coupling}, we may couple $W$ with $\FS$, a stationary Ferrari--Spohn diffusion,
so that they agree on $[-\frac{1}{2},\frac{1}{2}]$ with probability at least $1 - C e^{- c M \lambda^{1/3}}$.
Thus we may couple $Z$ and $\FS$ to be equal on $[-\frac{1}{2},\frac{1}{2}]$ with probability at least
\begin{equation}
\label{eq:highbound}
    1 - C e^{- c M \lambda^{1/3}} - C M \lambda^{1/3} e^{- c M^{3/2} \lambda^{1/2}}.
\end{equation}
Now the probability that $\FS(s) \geq M$ on $[-\frac{1}{2},\frac{1}{2}]$ is some constant $\eta = \eta(M) > 0$,
and we may ensure that \eqref{eq:highbound} is at least $1 - \frac{\eta}{2}$ for all $\lambda \geq \lambda_0$,
by choosing $\lambda_0$ to be large enough without changing the value of $M$.
Thus we may take $q_2 = \frac{\eta}{2}$ in \eqref{eq:high_goal}, finishing the proof.
\end{proof}

\subsubsection{Base case}
\label{sec:correlation_branching_basecase}

We now show the base case for Lemma \ref{lem:correlation_branchingprocess},
and then finish with a short proof of that lemma by
combining the inductive step above with the base case.

\begin{lemma}[Base case for branching process]
\label{lem:correlation_basecase}
For all large enough $M,H > 0$ there is some $p > 0$
such that for all $\lambda \geq 2$, the following holds.
Let $\X$ and $\Y$ be independent infinite-line $\lambda$-tilted line ensembles on $[-2H,2H]$
with arbitrary boundary data subject to the restriction $X^1(\pm2H) \leq H$.
Then
\begin{equation}
    \P \left[ \Reverse_\emptyset(\X,\Y) \right] \geq p.
\end{equation}
\end{lemma}

\begin{proof}[Proof of Lemma \ref{lem:correlation_basecase}]
First, Lemma \ref{lem:comingdown} implies that
\begin{equation}
    \P\left[ X^1(0) > H \right] \leq C e^{- H^c}
\end{equation}
for some constants $C, c > 0$ which we may take to be uniform in $\lambda \geq 2$, so we may assume that the above probability is $\leq \frac{1}{2}$
for $H$ large enough.
On the event that $X^1(0) \leq H$, then another application of Lemma \ref{lem:comingdown}
implies that
\begin{equation}
    \P \left[ X^1(\pm H) > \tfrac{M}{2} \right] \leq 2 C e^{- (M/2)^c},
\end{equation}
and we may also assume that $M$ is large enough that the above probability is $\leq \frac{1}{2}$.
Now if $X^1(\pm H) \leq \frac{M}{2}$, then Corollary \ref{cor:tailbound_interval} implies that
\begin{equation}
    \P \left[ X^1(s) \leq M \text{ for all } |s| \leq \tfrac{1}{2} \right] \geq 1 - C e^{- c M^{3/2}},
\end{equation}
so let us further assume that $M$ is large enough that the above is $\leq \frac{1}{2}$.
All of this implies that
\begin{equation}
    \P \left[ X^1(s) \leq M \text{ for all } |s| \leq \tfrac{1}{2} \right] \geq \tfrac{1}{8},
\end{equation}
i.e.\ that the probability of $\Low_\emptyset(\X)$ is at least $\frac{1}{8}$, recalling that
$J_\emptyset = \left[-\frac{1}{2},\frac{1}{2}\right]$.

As for $\High_\emptyset(\Y)$, note that $Y^1$ stochastically dominates a Brownian excursion with area tilt strength $2$
on $[-2H,2H]$, and since $M$ and $H$ are constants that are fixed by the above derivation,
there is some positive probability $\eta = \eta(M,H) > 0$ that this excursion remains above height $M$ on $J_\emptyset$.
Thus if $\X$ and $\Y$ are sampled independently, we have
\begin{equation}
    \P \left[
        \Reverse_\emptyset(\X,\Y)
    \right] \geq \tfrac{\eta}{8}.
\end{equation}
Thus we may take $p = \frac{\eta}{8}$.
\end{proof}

Finally we turn to the proof of Lemma \ref{lem:correlation_branchingprocess} itself.

\begin{proof}[Proof of Lemma \ref{lem:correlation_branchingprocess}]
Let us fix large enough $M, H > 0$ and let $\lambda_0$ and $q$ be as given by Corollary \ref{cor:inductivestep}
for this choice of $M$.
We will build the coupling between $(\X,\Y)$ and $\BP^{p,q}_b$ inductively to ensure that, for all $\tau \in \treeb$, if $\tau \in \BP^{p,q}_b$ then
$\Reverse_\tau(\X,\Y)$ holds.
For this, it suffices to ensure that this implication is hereditary, i.e.\ if $\tau \in \BP^{p,q}_b \Rightarrow \Reverse_\tau(\X,\Y)$,
then for all $\sigma \in [b]$, we have $\tau\sigma \in \BP^{p,q}_b \Rightarrow \Reverse_{\tau\sigma}(\X,\Y)$.
Additionally, we need to ensure that this implication is satisfied for $\tau = \emptyset$.

First, for this latter fact, Lemma \ref{lem:correlation_basecase} implies that there is some $p > 0$ such that
$\P\left[ \Reverse_\emptyset(\X,\Y) \right] \geq p$; note that for this we may assume that $\lambda_0 \geq 2$.
So we begin the coupling between $(\X,\Y)$ and $\BP^{p,q}_b$ by letting $\emptyset \in \BP^{p,q}_b$
with probability $p$ such that $\emptyset \in \BP^{p,q}_b$ implies $\Reverse_\emptyset(\X,\Y)$.

Now to prove that the implication is hereditary, let us suppose that the implication $\tau \in \BP^{p,q}_b \Rightarrow \Reverse_\tau(\X,\Y)$ holds for some $\tau \in \treeb$.
Then if $\tau \notin \BP^{p,q}_b$ the implication for $\tau\sigma$ holds vacuously since $\tau\sigma \notin \BP^{p,q}_b$,
so let's suppose that $\tau \in \BP^{p,q}_b$.
Then $\Reverse_\tau(\X,\Y)$ holds, so Corollary \ref{cor:inductivestep} implies that
\begin{equation}
    \left( \ind*{\Reverse_{\tau\sigma}}(\X,\Y) : \sigma \in [b] \right)
    \qquad \text{stochastically dominates} \qquad
    \left( B_{\tau\sigma} : \sigma \in [b] \right),
\end{equation}
the latter being a tuple of $b$ independent $\mathrm{Bernoulli}(q)$ random variables.
We may thus assume these are coupled so that $B_{\tau\sigma} = 1$ implies $\Reverse_{\tau\sigma}(\X,\Y)$.
So we may generate the children of $\tau$ in $\BP^{p,q}_b$ simply by including $\tau\sigma$ exactly
when $B_{\tau\sigma} = 1$.
This ensures that the desired implication is hereditary, finishing the proof.
\end{proof}

%% file: sections/4_correlation_3_topline.tex
\subsection{Many trials of the branching process}
\label{sec:correlation_topline}

In this section we prove Lemma \ref{lem:correlation_stopping} using Proposition \ref{prop:correlation_reversaltime},
thus finishing the proof of Proposition \ref{prop:reversal} and hence of Theorem \ref{thm:correlation}.
Essentially, we will prove that there are $\Omega(t)$ {independent copies} of the branching process described
in the previous section, each one having a positive chance to succeed.
We will use Proposition \ref{prop:topline} to control the top line in order to achieve this.
Note that for the proof of the finite-line case of Theorem \ref{thm:correlation}, we may also
use this proposition as the top line of the $n$-line ensemble is stochastically dominated by
the top line of the infinite-line ensemble.

Let us first divide up the interval $\left[ - \tfrac{t}{2}, \frac{3t}{2} \right]$ into subintervals
of length $4H$; specifically, for each $j \in \Z$ let us set $I^j = \left[ 4Hj, 4H(j+1) \right]$, and consider
the sets of integers
\begin{align}
    L^0_\ell &\coloneqq \left\{
        j \in \Z : I^j \subseteq \left[-\tfrac{t}{2}, 0 \right]
    \right\},
    &
    L^0_r &\coloneqq \left\{
        j \in \Z : I^j \subseteq \left[0, \tfrac{t}{2} \right]
    \right\}, \\
    L^t_\ell &\coloneqq \left\{
        j \in \Z : I^j \subseteq \left[\tfrac{t}{2}, t \right]
    \right\},
    &
    L^t_r &\coloneqq \left\{
        j \in \Z : I^j \subseteq \left[t, \tfrac{3t}{2} \right]
    \right\},
\end{align}
and set $L = L^0_\ell \cup L^0_r \cup L^t_\ell \cup L^t_r$.
For each $j \in L$, let us define
\begin{equation}
    \AllReversed^j(\X,\Y) \coloneqq \left\{ \exists s \in I^j \text{ such that } X^k(s) \leq Y^k(s) \text{ for all } k \geq 1 \right\}.
\end{equation}
We would like to apply Proposition \ref{prop:correlation_reversaltime} in each interval $I^j$ to show that
$\AllReversed^j(\X,\Y)$ holds with some positive probability.
In addition, we would like the trials of the branching processes in each interval $I^j$ to be independent
of one another to obtain the exponential bound on the probability of finding reversal times in Lemma \ref{lem:correlation_stopping}.
For this, we will introduce some auxiliary line ensembles where this independence does hold.

First, define $\ssY$ to be an infinite-line $\lambda$-tilted line ensemble on $[-\frac{t}{2},\frac{3t}{2}]$
which is pinned to zero at every multiple of $4H$ (and also at $-\frac{t}{2}, \frac{t}{2}, \frac{3t}{2}$);
thus $\ssY$ behaves independently on each interval $I^j$.
Note also that $\ssY$ is stochastically dominated by $\Y$, and we will assume that these have been
coupled so that $\ssY \preceq \Y$.

Next, we will define another auxiliary ensemble $\ssX$, though its definition will be a bit more involved
than that of $\ssY$.
Recall that to apply Proposition \ref{prop:correlation_reversaltime}
we need to ensure that the boundary conditions on each interval $I^j$ are $\leq H$.
However, this may not hold for all $j \in L$, but we will shortly apply Proposition \ref{prop:topline} to show that it happens for a good 
proportion of indices, assuming that $H$ is large enough.
For now, let us define
\begin{equation}
    \Endpoints^0_\ell(\X) \coloneqq \left\{ j \in L^0_\ell : X^1(4Hj), X^1(4H(j+1)) \leq H \right\},
\end{equation}
and similarly for $\Endpoints^0_r(\X), \Endpoints^t_\ell(\X)$, and $\Endpoints^t_r(\X)$,
and let us also set $\Endpoints(\X)$ to be the union of these four sets.
We first sample the values of $\X$ at all of the points $4Hi \in [-\frac{t}{2},\frac{3t}{2}]$ with integer $i$,
which allows us to determine whether or not each $j \in L$ lies in $\Endpoints(\X)$.
If $j \in \Endpoints(\X)$, then we set $\sX^k(4Hj) = \sX^k(4H(j+1)) = H$ for all $k \geq 1$.
For all other integer points $4Hi$ whose values have not been set in this way, we simply define $\ssX(4Hi) = \X(4Hi)$.
Then we sample both $\ssX$ and $\X$ on the intervals between these points; since the values of $\ssX$ at the endpoints
of all intervals are higher than those of $\X$, we may use monotonicity to ensure that $\X \preceq \ssX$.

Note that $(\X,\ssX)$ is still independent of $(\Y,\ssY)$, and if $\sX^i(s) \leq \sY^i(s)$ then $X^i(s) \leq Y^i(s)$
as well by stochastic domination.
So to prove Lemma \ref{lem:correlation_stopping}, it suffices to show the following lemma:

\begin{lemma}
\label{lem:correlation_stopping_tilde}    
There is some $\lambda_0 > 1$ and $C, c > 0$ such that for all $\lambda \geq \lambda_0$,
\begin{equation}
    \P \left[
        \exists j \in L^0_\ell \text{ such that } \AllReversed^j(\ssX,\ssY) \text{ holds}
    \right] \geq 1 - C e^{-c t},
\end{equation}
and the same holds with $L^0_\ell$ replaced by $L^0_r$, $L^t_\ell$, or $L^t_r$.
\end{lemma}

\begin{proof}[{Proof of Lemma \ref{lem:correlation_stopping_tilde}}]
After conditioning on the value of $\Endpoints(\X)$, the behavior of $\ssX$ is independent
on each interval $I^j$ with $j \in \Endpoints(\X)$.
Further, if $j \in \Endpoints(\X)$, then we may apply Proposition \ref{prop:correlation_reversaltime}
(which requires the boundary conditions be $\leq H$)
on the interval $I^j$, which shows that there is some positive probability $\delta_0 > 0$
for $\AllReversed^j(\ssX,\ssY)$ to hold as long as $\lambda \geq \lambda_0$.
Therefore Lemma \ref{lem:correlation_stopping_tilde} follows from the next lemma,
which shows that $\Endpoints(\X)$ is a large set with high probability.
\end{proof}

\begin{lemma}
\label{lem:correlation_init}
For all $H > 0$ large enough,
there are some $C, c > 0$  such that for all $\lambda \geq 2$,
if $\X$ is the stationary infinite-line $\lambda$-tilted ensemble, then we have
\begin{equation}
    \P \left[ \left| \Endpoints^0_\ell(\X) \right| \geq \tfrac{t}{16 H} \right] \geq 1 - C e^{-ct},
\end{equation}
and the same holds with $\Endpoints^0_\ell(\X)$ replaced by $\Endpoints^0_r(\X)$, $\Endpoints^t_\ell(\X)$,
or $\Endpoints^t_r(\X)$.
\end{lemma}

\begin{proof}[Proof of Lemma \ref{lem:correlation_init}]
We will present the proof for $\Endpoints^0_\ell(\X)$, with the other cases following identically.
Let us apply Proposition \ref{prop:topline} with $\eta = \frac{1}{4}$ and $\lambda_0 = 2$
(not the same $\lambda_0$ as in the statement of Proposition \ref{prop:correlation_reversaltime}).
The result is that there is some $H > 0$ and $C, c > 0$ such that for all $\lambda \geq 2$,
all $\Delta \geq 1$, and all large enough $t$ we have
\begin{equation}
    \P \left[
        \# \left\{ j \in \Z : j\Delta \in \left[-\tfrac{t}{2},0\right] \text{ and } X^1(j\Delta) > H \right\}
        > \tfrac{1}{4} \tfrac{t/2}{\Delta}
     \right] \leq C e^{-c t}.
\end{equation}
Note that we have used $T = \frac{t}{4}$ here.
Note also that the above bound remains true if we increase $H$,
and that we may choose any $\Delta \geq 1$, even one which depends on $H$.
We will plug in $\Delta = 4 H$ to obtain
\begin{equation}
    \P \left[
        \# \left\{ j \in \Z : 4Hj \in \left[-\tfrac{t}{2},0\right] \text{ and } X^1(4Hj) > H \right\}
        > \tfrac{1}{4} \tfrac{t}{8H}
     \right] \leq C e^{-c t}.
\end{equation}
So by a union bound, we have
\begin{equation}
    \P \left[
        \# \left\{ j \in \Z : j \in L^0_\ell \text{ and either } X^1(4Hj) > H \text{ or } X^1(4H(j+1)) > H \right\}
        > \tfrac{1}{2} \tfrac{t}{8H}
     \right] \leq C e^{-c t}.
\end{equation}
But the condition $X^1(4Hj) > H$ or $X^1(4H(j+1)) > H$
is equivalent to $j \notin \Endpoints^0_\ell(\X)$,
so this finishes the proof.
\end{proof}

This finishes the proof of Lemma \ref{lem:correlation_stopping},
which in turn finishes the proof of Theorem \ref{thm:correlation}.

%% file: sections/5_spectralgap_0.tex
\section{Spectral gap}
\label{sec:spectralgap}

In this section we prove Theorem \ref{thm:spectralgap}, showing that the generator $\gennl$
of the $n$-line ensemble (thought of as a Langevin diffusion with stationary distribution $\pinl$)
has a spectral gap $\gapnl$ which is bounded away from zero uniformly in $n$, as long as $\lambda$
is large enough.
Throughout this section, we will use $\X_n$ to denote this $n$-line ensemble for some fixed $n$.
Recall from \eqref{eq:gennldef} that the generator $\gennl$ is a densely defined self-adjoint operator on $L^2(\A_+^n,\pinl)$,
defined by
\begin{equation}
    (\gennl f)(\x) = \lim_{t \to 0} \frac{\E \left[ f(\X_n(t)) \middle| \X_n(0) = \x \right] - f(\x)}{t},
\end{equation}
whenever the limit of functions of $\x \in \A_+^n$ on the right-hand side exists in $L^2(\A_+^n,\pinl)$.
The set of $f$ for which this limit exists is the operator's domain, denoted by $\cD(\gennl)$.
Additionally, recall from \eqref{eq:gapnldef} that the spectral gap $\gapnl$ is defined by
\begin{equation}
    \gapnl = \inf \left\{
        \left< - \gennl f, f \right> :
        {f \in {\cD(\gennl)}} \text{ with } \left< f, 1 \right> = 0
        \text{ and } \left\| f \right\| = 1
    \right\}.
\end{equation}
Here and in what follows, we will abbreviate $\left< \cdot, \cdot \right>_\pinl$ as $\left< \cdot, \cdot \right>$
and similarly $\left\| \cdot \right\|_\pinl$ as $\left\| \cdot \right\|$.
We will show that there is some $\lambda_0 > 1$ and $\gamma > 0$
such that for all $\lambda \geq \lambda_0$ and all $n \in \N$ we have $\gapnl \geq \gamma$.
In fact, $\lambda_0$ and $\gamma$ will take the same values as in Theorem \ref{thm:correlation}.

The spectral gap is related to the exponential decay rate of the quantity $\langle e^{t \gennl} f, f \rangle$
for $f \in L^2(\A_+^n, \pinl)$ with $\left< f, 1 \right> = 0$ and $\left\| f \right\| = 1$.
This is the same as $\Cov[f(\X_n(0)), f(\X_n(t))].$ Thus, Theorem \ref{thm:spectralgap} essentially demands a vast generalization of Theorem \ref{thm:correlation} with the observables $X^{i}$ replaced by arbitrary, significantly more complicated, functions $f$.

To handle this, we first rely on some general approximation theory to reduce to the case of functions which are
\emph{Lipschitz continuous} in each coordinate, which is implied by being smooth and compactly supported. This is done in Section \ref{sec:spectralgap_operators}.
We then rely crucially on a  classical correlation inequality due to  
\cite{N80Normal,BS98Asymptotical} 
which allows us to bound the covariance of Lipschitz continuous functions in terms of the covariances of the individual coordinates, which Theorem \ref{thm:correlation} provides.
We state and prove this classical correlation inequality in Section \ref{sec:spectralgap_inequality} for completeness.
The hypothesis of this inequality is positive correlations for increasing functions, which is typically afforded
by the FKG inequality.  Since the FKG inequality, although expected, has not yet appeared in the literature for area-tilted line ensembles, and is likely to be of independent
interest and future use, we provide a self-contained proof of it in Section \ref{sec:spectralgap_discrete}.

\input{sections/5_spectralgap_1_operators}
\input{sections/5_spectralgap_2_inequality}
\input{sections/5_spectralgap_3_discrete}

%% file: sections/5_spectralgap_1_operators.tex
\subsection{Spectral gap via covariance of Lipschitz functions}
\label{sec:spectralgap_operators}

For notational convenience, let us define the Markov evolution operators $\pnl^t$ for $t \geq 0$ by
\begin{equation}
    (\pnl^t f)(\x) = \E \left[ f(\X_n(t)) \middle| \X_n(0) = \x \right];
\end{equation}
these are bounded self-adjoint positive semidefinite operators on $L^2(\A_+^n,\pinl)$ satisfying
the semigroup property, namely $\pnl^t \pnl^s = \pnl^{t+s}$.
The first result of this section will be that we can find some Lipschitz continuous function which serves
as an approximate witnesses to the spectral gap of $\gennl$.

\begin{lemma}
\label{lem:lipschitz_witness}
For all $\eps > 0$, there is some $s > 0$ and a Lipschitz continuous function $h : \A_+^n \to \R$
with $\left< h, 1 \right> = 0$, $\| h \| = 1$, and
\begin{equation}
    \left< \pnl^s h, h \right> \geq e^{-(\gapnl + \eps) s}.
\end{equation}
\end{lemma}

We will use the density of compactly supported smooth functions in $L^2$, and
for completeness, we next state and prove the following standard approximation lemma.  

\begin{lemma}
\label{lem:cinfdense}
For any $\lambda > 1$ and $n \in \N$, {$C_c^\infty(\R^n)$} is dense in $L^2(\A_+^n, \pinl)$, where $C_c^\infty(\R^n)$ denotes the set of all compactly supported smooth functions, and 
where we identify $L^2(\A_+^n, \pinl)$ with $L^2(\R^n, \pinl)$ by letting $\pinl$ assign zero mass to $\R^n \setminus \A_+^n$.
\end{lemma}

\begin{proof}
This holds because $\pinl$ is absolutely continuous with respect to the Lebesgue measure on $\R^n$.
Indeed, as mentioned in \cite[end of Section 1]{DLZ24Limiting} it has a positive $C^2$ density function $\rho$ 
which is the square of the solution to an elliptic PDE with Dirichlet boundary conditions in $\A_+^n$. 
Thus the density $\rho$ is bounded on every compact set, so let us express $\R^n$ as the union of an increasing sequence $(K_j)_{j \in \N}$
of compact sets.
For any $f \in L^2(\R^n,\pinl)$ and any $\eps > 0$ we may find some $N$ such that $\| f - f|_{K_N} \| \leq \frac{\eps}{2}$.
Then, since $C_c^\infty(K_N)$ is dense in $L^2(K_N)$, we may find some $g \in C_c^\infty(K_N)$ for which
\begin{equation}
    \int_{K_N} (g(x) - f(x))^2 \,dx \leq \frac{\eps/2}{\sup_{x \in K_N} \rho(x)},
\end{equation}
which implies that $\| f - g \| \leq \eps$.
\end{proof}

Now we turn to the proof of Lemma \ref{lem:lipschitz_witness}.

\begin{proof}[Proof of Lemma \ref{lem:lipschitz_witness}]
By the definition of the spectral gap $\gapnl$, there is some $g \in \cD(\gennl) \subseteq L^2(\A_+^n, \pinl)$
with $\left< g, 1 \right> = 0$ and $\left\| g \right\| = 1$ satisfying
\begin{equation}
    \left< - \gennl g, g \right> \leq \gapnl + \tfrac{\eps}{4}.
\end{equation}
And by the definition of $\gennl$, there is some $s > 0$ such that
\begin{equation}
    \tfrac{1}{s} \left< g - \pnl^s g, g \right>
    \leq \gapnl + \tfrac{\eps}{2}.
\end{equation}
Rearranging this equation, we find that
\begin{equation}
\label{eq:opderiv1}
    \left< \pnl^s g, g \right> \geq 1 - \left( \gapnl + \tfrac{\eps}{2} \right) s.
\end{equation}
Now recall that $\pnl^s$ is the kernel of a Markov process, so its operator norm is bounded by $1$ as can be seen e.g.\ by Jensen's inequality.
So since $\left< \cdot, 1 \right>$, $\left\| \cdot \right\|^2$, and $\langle \pnl^s \cdot, \cdot \rangle$
are all continuous functions on $L^2(\A_+^n,\pinl)$,
by Lemma \ref{lem:cinfdense} we can find some $u \in C_c^\infty(\R^n)$ such that
\begin{equation}
    \left| \left< u, 1 \right> \right| \leq \frac{\eps s}{1000},
    \qquad
    \left| \left\| u \right\|^2 - 1 \right| \leq \frac{\eps s}{1000},
    \qquad \text{and} \qquad
    \left| \left< \pnl^s u, u \right> - \left< \pnl^s g, g \right> \right| \leq \frac{\eps s}{1000}.
\end{equation}
So, setting
\begin{equation}
    h = \frac{u - \left< u, 1 \right> }{\left\| u - \left< u, 1 \right> \right\|},
\end{equation}
we find that $\left< h, 1 \right> = 0$, $\left\| h \right\| = 1$, and
\begin{align}
    \left| \left< \pnl^s h, h \right> - \left< \pnl^s g, g \right> \right|
    &\leq \left| \left< \pnl^s h, h \right> - \left< \pnl^s u, u \right> \right|
    + \left| \left< \pnl^s u, u \right> - \left< \pnl^s g, g \right> \right| \\
    &\leq \frac{\left| \left( 1 - \left\| u - \left< u, 1 \right> \right\|^2 \right) \left< \pnl^s u, u \right>
    - 2 \left< \pnl^s u, \left< u, 1 \right> \right>
    + |\left< u, 1 \right>|^2
    \right|}
    {\left\| u - \left< u, 1 \right> \right\|^2}
    + \frac{\eps s}{1000} \\
    &\leq \frac{\left( \left| 1 - \| u \|^2 \right| + |\left< u, 1 \right>|^2 + 2 \| u \| \left< u, 1 \right> \right) \| u \|^2
    + 2 \| u \| \left< u, 1 \right>
    + |\left< u, 1 \right>|^2
    }
    {\left\| u - \left< u, 1 \right> \right\|^2}
    + \frac{\eps s}{1000} \\
    &\leq \frac{\eps s}{4}
\end{align}
as long as $\eps$ is small enough.
Therefore, recalling \eqref{eq:opderiv1}, we have
\begin{equation}
    \left< \pnl^s h, h \right> \geq 1 - \left( \gapnl + \tfrac{3\eps}{4} \right) s.
\end{equation}
Note that the resulting $h$ may not be compactly supported since we are adding a constant.
However, since all smooth compactly supported functions such as $u$ are Lipschitz continuous,
$h$ will retain this property.

Note also that we may assume that $s$ is small enough in the first step of the above derivation so that
we may approximate the right-hand side from below by an exponential with a slightly faster decay rate.
In particular, we find that for any $\eps > 0$,
there is some $s > 0$ and a Lipschitz continuous function $h$ with $\left< h, 1 \right> = 0$ and $\left\| h \right\| = 1$ for which
\begin{equation}
    \left< \pnl^s h, h \right> \geq e^{- (\gapnl + \eps) s}.
\end{equation}
Finally, although $h$ was defined as a function from $\R^n \to \R$, we may simply consider its restriction to $\A_+^n$
as $\pinl$ assigns zero mass to $\R^n \setminus \A_+^n$.
This finishes the proof.
\end{proof}

Now note that by Jensen's inequality applied to the spectral measure of $h$ constructed in Lemma \ref{lem:lipschitz_witness}
(see e.g. \cite[Theorem VII.7]{RS81Functional}), for any positive integer $m$ we have
\begin{equation}
\label{eq:opbd1}
    \left< (\pnl^s)^m h, h \right> \geq
    \left< \pnl^s h, h \right>^m \geq
    e^{- (\gapnl+\eps) s m},
\end{equation}
using the fact that the map $x \mapsto x^m$ is convex for $x \geq 0$ (which is where the spectral measure
is supported as $\pnl^s$ is positive semidefinite) as well as the fact that $\left\| h \right\| = 1$
so that the spectral measure is a probability measure.
On the other hand, we have
\begin{equation}
\label{eq:opderiv2}
    \left< (\pnl^s)^m h, h \right> = \left< \pnl^{sm} h, h \right>
    = \E \left[ h(\X_n(0)) h(\X_n(sm)) \right].
\end{equation}
Our main proposition, stated below, will imply that the right-hand side above decays exponentially.
This fact was alluded to in \eqref{eq:lipschitzdecay}.

\begin{proposition}
\label{prop:lipschitzdecay}
Let $\lambda_0 > 1$ and $\gamma > 0$ be as in the statement of Theorem \ref{thm:correlation},
and let $\X_n$ be the stationary $n$-line $\lambda$-tilted ensemble for some $n \in \N$.
Then for any $h : \A_+^n \to \R$ which is Lipschitz continuous in all coordinates,
there is some constant $C_h$ such that for all $t > 0$, we have
\begin{equation}
    \left| \Cov \left[ h(\X_n(0)), h(\X_n(t)) \right] \right| \leq C_h e^{- \gamma t}.
\end{equation}
\end{proposition}

Note that the right-hand side of \eqref{eq:opderiv2} is indeed the covariance between
$h(\X_n(0))$ and $h(\X_n(sm))$ since we have $\E[h(\X_n(0))] = \left< h, 1 \right> = 0$.
So Proposition \ref{prop:lipschitzdecay}
immediately implies Theorem \ref{thm:spectralgap}, as we now quickly demonstrate.

\begin{proof}[Proof of Theorem \ref{thm:spectralgap}]
Applying Proposition \ref{prop:lipschitzdecay} with $t = sm$ to the inequality obtained by
combining \eqref{eq:opbd1} and \eqref{eq:opderiv2}, we find that
\begin{equation}
    e^{- (\gapnl + \eps) s m } \leq C_h e^{- \gamma sm}.
\end{equation}
Taking logarithms, dividing by $s m$, and sending $m$ to infinity, we find that
\begin{equation}
    \gapnl + \eps \geq \gamma,
\end{equation}
and since $\eps > 0$ was arbitrary, this finishes the proof.
\end{proof}

The rest of this section is dedicated to the proof of Proposition \ref{prop:lipschitzdecay}
via our main result, Theorem \ref{thm:correlation}, in conjunction with a classical correlation inequality,
which is stated and proved for completeness in Section \ref{sec:spectralgap_inequality}.
This inequality is a consequence of the FKG positive association inequality, a version of which
is proved for area-tilted line ensembles in Section \ref{sec:spectralgap_discrete} below.

%% file: sections/5_spectralgap_2_inequality.tex
\subsection{Lipschitz covariance inequality via FKG inequality}
\label{sec:spectralgap_inequality}

We now turn to the proof of Proposition \ref{prop:lipschitzdecay}.
For any domain $\Omega \subseteq \R^N$ and a
Lipschitz continuous function $F : \Omega \to \R$, it will be helpful to
consider its Lipschitz constants in each direction.
These are defined for each $i \in \{ 1, \dotsc, N \}$
(using $x^i$ to denote the $i$th coordinate of $\x \in \R^N$ as usual)
by
\begin{equation}
    \Lip_i^\Omega(F) \coloneqq \sup \left\{
        \frac{|F(\x) - F(\y)|}{|x^i - y^i|}
    : \x, \y \in \Omega \text{ with } x^j = y^j \text{ for all } j \neq i
    \text{ and } x^i \neq y^i
    \right\}.
\end{equation}
We will also consider increasing functions $f : \Omega \to \R$, by which we mean
increasing in each coordinate.
The following result provides a bound on the covariance of Lipschitz functions
under the assumption that the measure has positive correlations between all
\emph{increasing} Lipschitz functions.
A weaker form of this was first stated in \cite{N80Normal} without proof;
later, a short proof was given by \cite{BS98Asymptotical}.
We repeat this proof for completeness, 
based on an English translation in \cite[Theorem 5.3]{BS07Limit}.

\begin{lemma}
\label{lem:NBS}
Suppose that $\mu$ is a {probability} {measure} on $\Omega \subseteq \R^N$ such
that for all {Lipschitz} continuous increasing functions
$f, g : \Omega \to \R$ and $\x \sim \mu$ we have
\begin{equation}\label{fkg12}
    \Cov\left[ f(\x), g(\x) \right] \geq 0.
\end{equation}
Then for all {Lipschitz continuous} functions $F, G : \Omega \to \R$ (not necessarily increasing), we have
\begin{equation}
    \left| \Cov \left[ F(\x), G(\x) \right] \right|
    \leq \sum_{i,j=1}^N \Lip_i^\Omega(F) \Lip_j^\Omega(G) \Cov\left[ x^i, x^j \right].
\end{equation}
\end{lemma}

\begin{proof}[Proof of Lemma \ref{lem:NBS}]
Let us define increasing functions $F_+, G_+$ and decreasing functions $F_-, G_-$ by
\begin{equation}
    F_\pm(\x) \coloneqq F(\x) \pm \sum_{i=1}^N \Lip_i^\Omega(F) x^i
    \qquad \text{and} \qquad
    G_\pm(\x) \coloneqq G(\x) \pm \sum_{j=1}^N \Lip_j^\Omega(G) x^j.
\end{equation}
Note also that these functions are Lipschitz continuous. 
Thus by the assumption \eqref{fkg12}, we have
\begin{align}
    0 &\geq \Cov\left[ F_+(\x), G_-(\x) \right] + \Cov\left[ F_-(\x), G_+(\x) \right] \\
    &= 2 \Cov\left[ F(\x), G(\x) \right] - 2 \sum_{i,j=1}^N \Lip_i^\Omega(F) \Lip_j^\Omega(G) \Cov[x^i, x^j],
\end{align}
where the final equality is straightforward algebra. This covers the case where $\Cov[F(\x),G(\x)] \geq 0$.
For the other case, we have
\begin{align}
    0 &\leq \Cov\left[ F_+(\x), G_+(\x) \right] + \Cov\left[ F_-(\x), G_-(\x) \right] \\
    &= 2 \Cov\left[ F(\x), G(\x) \right] + 2 \sum_{i,j=1}^N \Lip_i^\Omega(F) \Lip_j^\Omega(G) \Cov[x^i, x^j],
\end{align}
which finishes the proof.
\end{proof}

To verify the hypothesis of Lemma \ref{lem:NBS}, we prove a version of the FKG inequality
for area-tilted line ensembles.
This is related to the monotonicity property which was mentioned in
Section \ref{sec:inputs_monotonicity} and used throughout the proof of our main result,
but we did not find a statement of this precise form in the literature, so we provide a proof
in Section \ref{sec:spectralgap_discrete} below for completeness.
Note that the statement we provide is much more general than we actually require for
our present purposes, but we expect this statement will be of broader interest.

\begin{proposition}[FKG inequality]
\label{prop:FKG}
Let $\Y \sim \atmeasure^{\x,\y,\bl}_{n,\ell,r,\f,\g}$, recalling Definition \ref{def:extFLE}
of this measure with $n$ curves having boundary conditions $\x, \y \in \overline{\A_+^n}$
on $[\ell,r]$, Lipschitz floors and ceilings given by $\f, \g : [\ell,r] \to \A_+^n$ with $\f \prec \g$,
and varying area tilt strengths given by $\bl : [\ell,r] \to \R_+^n$. 
We may also consider the measure without floors or ceilings, or with floors but not ceilings, et cetera.
Consider any functions $f, g$ on the space of continuous curves $[\ell,r] \to \A_+^n$
which are continuous in the topology of uniform convergence, increasing
with respect to the ordering $\preceq$, and such that $\Cov[f(\Y),g(\Y)]$ is well-defined.
Then $\Cov \left[ f(\Y), g(\Y) \right] \geq 0$.
\end{proposition}

\begin{remark}
\label{rmk:cov}
Since $\Y$ satisfies strong tail bounds (Theorem \ref{thm:tailbound_onepoint}), a function $f$ is square integrable as soon as $f$ satisfies mild growth bounds. 
For instance, for our purposes it is enough  that it is satisfied for Lipschitz continuous
functions (thus having linear growth) of $\Y(t_i)$ for finitely many $t_i$ when $\Y = \Y_{n,T}$ is the ensemble with $n$ lines
on $[-T,T]$ with zero boundary conditions and a floor at zero.
For such functions $f$ and $g$ in this setting, $\Cov[f(\Y_{n,T}),g(\Y_{n,T})]$ thus indeed exists.
\end{remark}

Now before proving the FKG inequality of Proposition \ref{prop:FKG}, let us see how it and Lemma \ref{lem:NBS} allow us to prove
Proposition \ref{prop:lipschitzdecay}, thus finishing the proof of Theorem \ref{thm:spectralgap}.

\begin{proof}[Proof of Proposition \ref{prop:lipschitzdecay}]
Let us introduce the truncation
\begin{equation}
\label{eq:truncation}
    h_M(x) = \max\{ - M, \min\{M, h(x) \} \}.
\end{equation}
Now recall from the discussion below Definition \ref{def:finiteLE} that $\X_n$ is the local weak limit as $T \to \infty$ of $\Y_{n,T}$ which is an $n$-line $\lambda$-tilted ensemble
on $[-T,T]$ with zero boundary conditions.
Thus since $h_M$ is a bounded continuous function of the paths at a single point, we have
\begin{equation}
    \Cov \left[ h_M(\X_n(0)), h_M(\X_n(t)) \right]
    = \lim_{T \to \infty} \Cov \left[ h_M(\Y_{n,T}(0)), h_M(\Y_{n,T}(t)) \right].
\end{equation}
Also, for each $i,j \in \{1,\dotsc,n\}$ we also have
\begin{align}
    \Cov\left[ X_n^i(0), X_n^j(t) \right]
    &= \E \left[ X_n^i(0) X_n^j(t) \right]
    - \E \left[ X_n^i(0) \right] \E \left[ X_n^j(t) \right] \\
    &= \lim_{T \to \infty} \E \left[ Y_{n,T}^i(0) Y_{n,T}^j(t) \right]
    - \lim_{T \to \infty} \E \left[ Y_{n,T}^i(0) \right] \cdot \lim_{T \to \infty} \E \left[ Y_{n,T}^j(t) \right] \label{limits} \\
    &= \lim_{T \to \infty} \Cov \left[ Y_{n,T}^i(0), Y_{n,T}^j(t) \right],
\end{align}
where the equality \eqref{limits} holds by the monotone convergence theorem
since $\Y_{n,T}$ may be coupled to increase to $\X_n$ along a subsequence $T_j \nearrow \infty$
and the expectations are increasing in $T$ by monotonicity.

Let us take $N = 2n$ and consider the domain $\Omega = \A_+^n \times \A_+^n$; for notational convenience,
we will write $(\x,\y)$ for an element of $\Omega$.
Now let $\mu$ be the joint distribution of $(\Y_{n,T}(0),\Y_{n,T}(t))$.
We will apply Lemma \ref{lem:NBS} with this setup, so we must verify the hypothesis of positive covariances
for Lipschitz continuous increasing functions.
This holds by Proposition \ref{prop:FKG}, specialized to functions of $\Y_{n,T}(0)$ and $\Y_{n,T}(t)$,
with Remark \ref{rmk:cov} ensuring that the covariance exists.
So we may indeed apply Lemma \ref{lem:NBS}.

Therefore, taking $F(\x,\y) = h_M(\x)$ and $G(\x,\y) = h_M(\y)$, we find that
\begin{equation}
    \left| \Cov \left[ h_M(\Y_{n,T}(0)), h_M(\Y_{n,T}(t)) \right] \right|
    \leq \sum_{i,j=1}^n \Lip_i(h) \Lip_j(h) \Cov \left[ Y_{n,T}^i(0), Y_{n,T}^j(t) \right].
\end{equation}
Here for notational convenience we set $\Lip_i(h) = \Lip_i^{\A_+^n}(h)$
and use the fact that for $i,j \in \{1,\dotsc,2n\}$ we have
\begin{equation}
    \Lip_i^\Omega(F) = \begin{cases}
        \Lip_i(h_M) \leq \Lip_i(h) &\text{if } i \leq n \\
        0 &\text{if } i > n,
    \end{cases}
    \qquad \text{and} \qquad
    \Lip_j^\Omega(G) = \begin{cases}
        0 &\text{if } j \leq n \\
        \Lip_{j-n}(h_M) \leq \Lip_{j-n}(h) &\text{if } j > n.
    \end{cases}
\end{equation}
Now taking the limit $T \to \infty$ on both sides, we find that
\begin{equation}
    \left| \Cov\left[ h_M(\X_n(0)), h_M(\X_n(t)) \right] \right|
    \leq \sum_{i,j=1}^n \Lip_i(h) \Lip_j(h) \Cov \left[ X_n^i(0), X_n^j(t) \right].
\end{equation}
Finally, by the dominated convergence theorem, we may replace the left-hand side in the above inequality
with
$\left| \Cov\left[ h(\X_n(0)), h(\X_n(t)) \right] \right|$ by taking the limit as $M \to \infty$.
Therefore by Theorem \ref{thm:correlation} we may take
\begin{equation}
    C_h = C \sum_{i,j=1}^n \Lip_i(h) \Lip_j(h) \lambda^{-(i+j-2)/3},
\end{equation}
where $C$ is the constant in that theorem.
\end{proof}

%% file: sections/5_spectralgap_3_discrete.tex
\subsection{FKG inequality via discrete approximation}
\label{sec:spectralgap_discrete}

We now turn to the proof of Proposition \ref{prop:FKG},
which yields positive correlations for increasing functions.
There are various ways to prove such a result,
and a method due to Holley \cite{H74Remarks} which has been applied in many contexts constructs
stochastic domination couplings via Glauber dynamics.
This method has inspired proofs in the line ensemble literature of monotonicity properties
related to the positive correlation inequality we aim to prove, such as \cite{CH14Brownian}
which proceeded via Glauber dynamics on a discrete approximation to the continuous line ensemble.
Later \cite{CIW19Confinement} applied the same method to area-tilted line ensembles,
resulting in the monotonicity property discussed in Section \ref{sec:inputs_monotonicity} above.

While we will also work with a discrete approximation to the area-tilted line
ensemble, instead of proceeding by analyzing  Glauber dynamics, we will directly
verify using the form of the measure that it satisfies the \emph{FKG lattice
condition}. That this implies positive correlations is by now classical and goes
back to \cite{FKG71Correlation}, coming to be known as the classical FKG
inequality.

Let us now set up the discrete model.
We will work with a fixed interval $[\ell,r] \subseteq \R$.
For notational convenience, for any $N \in \N$
let us use $[\ell,r]_N$ to denote $\frac{1}{N} \Z \cap [\ell,r]$.
Additionally, we let $\ell_N = \min [\ell, r]_N$ and $r_N = \max [\ell, r]_N$.
Now for $\x,\y \in \A_+^n$ and $\f, \g$ appropriate floor and ceiling functions (which may also take the values $\pm \infty$
which means there is no corresponding floor or ceiling),
we now define a set $\cW_N = \cW_{n,\ell,r,\f,\g}^{\x,\y}(N)$ of all possible diffusively scaled
random walk ensemble trajectories subject to the boundary, floor, and ceiling conditions.
Later, in \eqref{eq:discat} below, we will consider an area tilted measure on this space.

\begin{definition}[The space $\cW_N = \cW_{n,\ell,r,\f,\g}^{\x,\y}(N)$]
\label{def:wn}
The set $\cW_N$ consists of all collections $\W$ of $n$ random walk trajectories $\W = (W^i : i \in \{1,\dotsc,n\})$,
where each $W^i$ is a function $[\ell,r]_N \to \frac{1}{\sqrt{N}} \Z$, such that $\W$ satisfies the following conditions.
\begin{enumerate}
    \item For each $i \in \{1,\dotsc,n\}$ and all $t_1, t_2 \in [\ell,r]_N$ with $|t_1 - t_2| = \frac{1}{N}$, we have
    $\left| W^i(t_1) - W^i(t_2) \right| = \frac{1}{\sqrt{N}}$, so that $W^i$ is a diffusively rescaled random walk trajectory.
    \item For each $t \in [\ell,r]_N$, we have the nonintersection constraint $W^1(t) > \dotsb > W^n(t)$.
    \item For each $i \in \{1,\dotsc,n\}$, the trajectory $W^i$ has appropriate boundary conditions, which we now explain.
    We would like to have $W^i(\ell_N) = x^i + O(\frac{1}{\sqrt{N}})$ and $W^i(r_N) = y^i + O(\frac{1}{\sqrt{N}})$,
    but some care is needed here.
    First, if $x^i$ and $x^{i+1}$ are very close (or equal), then choosing arbitrary discrete approximations of the boundary conditions
    may violate the strict ordering of the random walk trajectories.
    Furthermore, since the parity of the two boundary conditions must be compatible in order for $W^i$ satisfying the above conditions to exist,
    we will need to correct this issue if it arises.
    So we will impose the boundary conditions
    \begin{equation}
        W^i(\ell_N) = \frac{1}{\sqrt{N}} \left( \left\lfloor \sqrt{N} x^i \right\rfloor + 2 (n-i) \right)
        \qquad \text{ and } \qquad
        W^i(r_N) = \frac{1}{\sqrt{N}} \left( \left\lfloor \sqrt{N} y^i \right\rfloor + 2(n-i) + \epsilon^i \right),
    \end{equation}
    where $\epsilon^i \in \{0,1\}$ is chosen (depending on $x^i, y^i$, and $N$) to mitigate the parity issue,
    and the term $2(n-i)$ is added to ensure strict ordering of the boundary conditions.
    \item For each $t \in [\ell,r]_N$ and each $i \in \{1,\dotsc,n\}$ we have the floor and ceiling constraints.
    To state these constraints, we introduce the linear interpolation $\ssW = (\sW^i)_{i=1}^n$ of $\W$, which is a continuous function
    $[\ell,r] \to \A_+^n$ obtained by interpolating $\W$ linearly between points of $[\ell,r]_N$, and extending by a constant
    in $[\ell,r] \setminus [\ell,r]_N$ (i.e.\ at the edges).
    The floor and ceiling constraint is then that $\sW^i(t) \in (f^i(t), g^i(t))$ for all $t \in [\ell,r]$. \label{def:lerp}
\end{enumerate}
This completes the definition of our set of possible states $\cW_N$ in the
discrete model of a line ensemble. Note that due to our choices of boundary
conditions,
$\cW_N$ is nonempty as long as $N$ is large enough and the
boundary conditions are compatible with the floors and ceilings.

For example,
to construct one element $\W \in \cW_N$, we may first let $W^n_\ell$ take
downwards steps away from the left boundary, only taking upwards steps when a
downwards step would intersect $f^n$.  Similarly, we may let $W^n_r$ take
downwards steps away from the right boundary, and set $W^n =
\max\{W^n_\ell,W^n_r\}$.  Then we may inductively construct $W^i$ in the same
way, ensuring that it also does not intersect $f^i$ or $W^{i+1}$.  As long as
$N$ is large enough, the ceiling constraint will also be satisfied since there
is some constant minimum amount of room between $f^i$ and $g^i$ as $\f \prec
\g$ and these functions are continuous on a compact interval.
\end{definition}

We now introduce a discrete probability distribution $\mu_N = \mu_N^{\bl}$ on the space of continuous paths $[\ell,r] \to \A_+^n$
which is concentrated on the linear interpolations of elements of $\cW_N$, which were defined in part \ref{def:lerp} of
Definition \ref{def:wn} above.
The following measure uses the same exponential tilting factor as the area-tilted line ensemble
$\atmeasure^{\x,\y,\bl}_{n,\ell,r,\f,\g}$ of Definition \ref{def:extFLE}, but is simply restricted to the aforementioned discrete set.
\begin{equation}
\label{eq:discat}
    \mu_N [ \mathbf{Z} ] \propto \Exp{- 2 \sum_{i=1}^n \int_\ell^r \lambda^i(t) Z^i(t) \,dt }
    \cdot \ind{\mathbf{Z} = \ssW \text{ for some } \W \in \cW_N},
\end{equation}
We also obtain a probability distribution on $\cW_N$ itself via $\W \mapsto \mu_N\left[\ssW\right]$.

In the remainder of this section we will prove two lemmas.
First, in Lemma \ref{lem:inv} we will show that $\mu_N$ converges weakly to the desired area-tilted line ensemble using standard measure-theoretic
arguments and Donsker's invariance principle \cite[Theorem 8.1.4]{D19Probability}.
Next, in Lemma \ref{lem:linear}, we will prove the classical FKG lattice criterion
which, by \cite[Proposition 1]{FKG71Correlation}, implies
positive correlations for any increasing functions $f$ and $g$ of a random element $\W \in \cW_N$
such that $\ssW \sim \mu_N$.
These lemmas will then allow us to prove Proposition \ref{prop:FKG} at the end of the section.

First, to show that $\mu_N$ converges to $\atmeasure^{\x,\y,\bl}_{n,\ell,r,\f,\g}$, we will rely on the following
measure-theoretic fact.

\begin{lemma}
\label{lem:mtfact}
Suppose that $\Omega$ is a topological space and that $(\nu_N)_{N \in \N}$ are probability measures on $\Omega$
converging weakly to the probability measure $\nu$.
If $f : \Omega \to \R$ is a bounded continuous function and $A \subseteq \Omega$ is a Borel set
with $\nu(\partial A) = 0$, then
\begin{equation}
    \mu_N(dx) \propto f(x) \ind*{A}(x) \nu_N(dx)
    \qquad \text{converges weakly to} \qquad
    \mu(dx) \propto f(x) \ind*{A}(x) \nu(dx),
\end{equation}
where $\mu_N$ and $\mu$ are normalized to be probability measures.
\end{lemma}

\begin{proof}[Proof of Lemma \ref{lem:mtfact}]
First we show that $\mu_N'(dx) \propto \ind*{A}(x) \nu_N(dx)$ converges weakly to $\mu'(dx) \propto \ind*{A}(x) \nu(dx)$.
Note that the normalization constant for $\mu_N'$ is $\nu_N(A)$ which converges to $\nu(A)$ by the portmanteau lemma \cite[Theorem 3.2.11]{D19Probability}.
Now suppose that $B$ is a Borel set in $\Omega$ with $\mu'(\partial B) = 0$.
Then, since
\begin{equation}
    \partial(A \cap B) \subseteq \partial A \cup (\partial B \cap A),
\end{equation}
we have $\nu(\partial(A \cap B)) = 0$ as well.
Therefore by the portmanteau lemma we have
\begin{equation}
    \mu_N'(B) = \frac{1}{\nu_N(A)} \nu_N(A \cap B)
    \xlongrightarrow{N \to \infty}
    \frac{1}{\nu(A)} \nu(A \cap B) = \mu'(B),
\end{equation}
i.e.\ $\mu_N'$ converges weakly to $\mu'$.
Now note that for any bounded continuous function $h : \Omega \to \R$,
the function $f h$ is also bounded and continuous, so we have
\begin{equation}
    \int h(x) \mu_N(dx) = \int h(x) f(x) \mu_N'(dx)
    \xlongrightarrow{N \to \infty} \int h(x) f(x) \mu'(dx)
    = \int h(x) \mu(dx),
\end{equation}
which finishes the proof.
\end{proof}

\begin{lemma}
\label{lem:inv}
If $\x, \y \in \A_+^n$ (not $\overline{\A_+^n}$)
then the measure $\mu_N$ converges weakly to $\atmeasure^{\x,\y,\bl}_{n,\ell,r,\f,\g}$
as $N \to \infty$.
\end{lemma}

\begin{proof}[Proof of Lemma \ref{lem:inv}]
First note that {by Donsker's invariance principle \cite[Theorem 8.1.4]{D19Probability},} without the nonintersection conditioning or area tilt,
a uniform sample of $n$ random walk trajectories converges to $n$ independent Brownian bridges with endpoints
given by $\x,\y$.
Thus we may apply Lemma \ref{lem:mtfact} with $\nu_N$ being uniform on $\cW_N$ and $\nu$ being $n$ independent Brownian bridges,
with $f$ being the exponential area tilt factor in \eqref{eq:discat}, and $A$ being the event of nonintersection of paths, plus
the floor and ceiling constraints.
Note that the boundary of $A$ is the set of paths which intersect but do not cross (or do the same for the floors or ceilings),
which has measure zero under $\nu$ by Blumenthal's 0-1 law \cite[Theorem 7.2.3]{D19Probability} (using the fact that the floors
and ceilings are Lipschitz here).
Since $\bl$ takes positive values and the floors are all at least $0$, the function $f$ is bounded (by $1$)
and continuous on the space of paths.
Therefore Lemma \ref{lem:mtfact} applies which finishes the proof by the definition of
$\atmeasure^{\x,\y,\bl}_{n,\ell,r,\f,\g}$.
\end{proof}

To show that $\mu_N$ has positive correlations for increasing functions,
we will simply invoke the classical FKG criterion due to \cite{FKG71Correlation}, as mentioned above.
Specifically, \cite[Proposition 1]{FKG71Correlation} states that for any measure $\nu$ on a finite distributive
lattice (meaning a set with minimum and maximum operations $\wedge$ and $\vee$ satisfying appropriate conditions),
we have positive correlations for increasing functions as soon as
\begin{equation}
    \nu[x \wedge y] \cdot \nu[x \vee y] \geq \nu[x] \cdot \nu[y].
\end{equation}
For our purposes, we will work with the measure $\nu[\W] = \mu_N[\ssW]$.
Note that the underlying lattice may be taken to be finite because every element of $\cW_N$
is upper bounded by some constant multiple of $N$.
So it suffices to show that for any two collections of discrete
random walk paths $\W_1, \W_2 \in \cW_N$, we have
\begin{equation}
\label{eq:fkgcond}
    \mu_N \left[ \widetilde{\W_1 \wedge \W_2} \right] \cdot \mu_N \left[ \widetilde{\W_1 \vee \W_2} \right]
    \geq \mu_N \left[\ssW_1\right] \cdot \mu_N \left[\ssW_2\right],
\end{equation}
where $\wedge$ and $\vee$ denote the pointwise minimum and maximum of paths.
First, we demonstrate that these maxima and minima cannot escape $\cW_N$:

\begin{lemma}
\label{lem:stayin}
If $\W_1, \W_2 \in \cW_N$, then $\W_1 \wedge \W_2, \W_1 \vee \W_2 \in \cW_N$ as well.
Moreover, we also have
\begin{equation}
    \widetilde{\W_1 \wedge \W_2} = \ssW_1 \wedge \ssW_2
    \qquad \text{and}  \qquad
    \widetilde{\W_1 \vee \W_2} = \ssW_1 \vee \ssW_2,
\end{equation}
i.e.\ the maximum or minimum can be taken before or after linear interpolation.
\end{lemma}

\begin{proof}[Proof of Lemma \ref{lem:stayin}]
The boundary conditions do not change when taking the pointwise maximum or minimum,
and likewise the constraints of staying between the floors and ceilings hold trivially.
Additionally, the constraint of remaining ordered is retained because if $a_1 < b_1$ and $a_2 < b_2$
then $a_1 \wedge a_2 < b_1 \wedge b_2$ and similarly for $\vee$.

The last property to check is that the increments are of size exactly $\frac{1}{\sqrt{N}}$.
First, the increments cannot be of size $0$ since $\sqrt{N} W_1^i(t)$ and $\sqrt{N} W_2^i(t)$
have the same parity for all $t \in [\ell,r]_N$ since the paths $W_1^i$ and $W_2^i$ have the same
boundary conditions, and the maximum and minimum thus have this same parity as well.
Next, the increments cannot be of size greater than $\frac{1}{\sqrt{N}}$ because the continuous
functions corresponding to $W_1^i$ and $W_2^i$ are $\sqrt{N}$-Lipschitz and the minimum or maximum
of any two $L$-Lipschitz functions are also $L$-Lipschitz.

This proves the first claim, and the second follows because of the preservation of parity as well,
which implies that the linear segment between two discrete points in $[\ell,r]_N$ in $\ssW_1 \vee \ssW_2$
is the same as one of the corresponding linear segments in $\ssW_1$ or $\ssW_2$, so the curve which attains
the maximum cannot change \emph{between} two discrete points, and similarly for the minimum.
\end{proof}

Finally we can verify the hypothesis \eqref{eq:fkgcond} of the classical FKG inequality,
in fact showing equality whenever the right-hand side is nonzero.

\begin{lemma}
\label{lem:linear}
For $\W_1, \W_2 \in \cW_N$, we have
\begin{equation}
    \mu_N \left[ \widetilde{\W_1 \wedge \W_2} \right] \cdot \mu_N \left[ \widetilde{\W_1 \vee \W_2} \right]
    = \mu_N \left[\ssW_1\right] \cdot \mu_N \left[\ssW_2\right],
\end{equation}
\end{lemma}

\begin{proof}[Proof of Lemma \ref{lem:linear}]
By Lemma \ref{lem:stayin}, all of the relevant indicators in the definition of $\mu_N$ are simply $1$, 
and so we have
\begin{align}
    \frac{ \mu_N \left[ \widetilde{\W_1 \wedge \W_2} \right] \cdot \mu_N \left[ \widetilde{\W_1 \vee \W_2} \right] }
    { \mu_N \left[ \ssW_1 \right] \cdot \mu_N \left[ \ssW_2 \right] }
    &= \frac{ \mu_N \left[ \ssW_1 \wedge \ssW_2 \right] \cdot \mu_N \left[ \ssW_1 \vee \ssW_2 \right] }
    { \mu_N \left[ \ssW_1 \right] \cdot \mu_N \left[ \ssW_2 \right] } \\
    &\hspace{-2cm} = \Exp{
        - 2 \sum_{i=1}^n \int_\ell^r \lambda^i(t)
        \left( \sW_1^i(t) \wedge \sW_2^i(t) + \sW_1^i(t) \vee \sW_2^i(t)
        - \sW_1^i(t) - \sW_2^i(t) \right) \,dt
    },
\end{align}
and the expression inside of the inner parentheses above is zero.
\end{proof}

We conclude by collecting these lemmas into a proof of Proposition \ref{prop:FKG},
which finishes the proof of Proposition \ref{prop:lipschitzdecay} and in turn Theorem \ref{thm:spectralgap}.

\begin{proof}[Proof of Proposition \ref{prop:FKG}]
Let us first consider the truncations $f_M, g_M$ as in \eqref{eq:truncation}; these are
now bounded, continuous, and increasing functions on the space of paths.
As discussed above, Lemma \ref{lem:linear} and the classical FKG lattice criterion imply that
\begin{equation}
    \Cov \left[ f_M(\mathbf{Z}), g_M(\mathbf{Z}) \right] \geq 0
\end{equation}
for $\mathbf{Z} \sim \mu_N$, using the increasing nature of $f_M$ and $g_M$.
Let us first assume that $\x,\y \in \A_+^n$ (not its closure).
Then taking the limit $N \to \infty$ and using the bounded continuous nature of $f_M$ and $g_M$,
Lemma \ref{lem:inv} implies that
\begin{equation}
\label{eq:bdcov}
    \Cov \left[ f_M(\Y), g_M(\Y) \right] \geq 0
\end{equation}
as well.
If $\x, \y \in \overline{\A_+^n}$ we may take another {monotone} limit from within
$\A_+^n$ {(using monotonicity which implies that both terms in the covariance expression converge)}
to conclude \eqref{eq:bdcov} in this case as well.

Finally, since $\Cov[f(\Y),g(\Y)]$ exists by assumption, all of $f(\Y)$, $g(\Y)$, and $f(\Y) g(\Y)$
are integrable.
Thus by the dominated convergence theorem, we have
\begin{equation}
    \Cov \left[f(\Y), g(\Y)\right] = \lim_{M \to \infty} \Cov \left[ f_M(\Y), g_M(\Y) \right] \geq 0,
\end{equation}
finishing the proof.
\end{proof}